\documentclass[a4paper,11pt,reqno]{amsart}
\usepackage[margin=3cm]{geometry}

\newcommand{\version}{\today}

\usepackage{amsthm,amsfonts,amsmath, amscd, bbold}
\usepackage{mathrsfs}
\usepackage{accents}

\usepackage{mathrsfs}
\usepackage{amscd}
\usepackage[active]{srcltx}
\usepackage{verbatim}
\usepackage[colorlinks,linkcolor={black},citecolor={black},urlcolor={black}]{hyperref}

\usepackage[latin1]{inputenc}
\usepackage{tikz}
\usetikzlibrary{trees}

\usepackage{mathtools}

\swapnumbers
                              
\theoremstyle{plain}
\newtheorem{thm}{THEOREM}[section]
\newtheorem{lm}[thm]{LEMMA}

\DeclareMathOperator{\dive}{div}

\theoremstyle{definition}
\newtheorem{defi}[thm]{DEFINITION}
\newtheorem{remark}[thm]{Remark}
\theoremstyle{remark}                       
\newcommand{\upchi}{\raise1pt\hbox{$\chi$}}
\newcommand{\R}{{\mathord{\mathbb R}}}
\newcommand{\C}{{\mathord{\mathbb C}}}
\newcommand{\Z}{{\mathord{\mathbb Z}}}
\newcommand{\N}{{\mathord{\mathbb N}}}

\newcommand{\tr}{{\rm Tr}}
\numberwithin{equation}{section}
\newcommand{\dd}{\:{\rm d}}

\renewcommand{\ln}{\log}
\newcommand{\eps}{\varepsilon}
\newcommand{\un}{{\rm 1\kern -2.5pt l}}

 \newcommand{\rb}{\rho_{2}}
  
\def\Tr{{\rm Tr}}

 \newcommand{\Dens}{{\mathfrak S}}
  \newcommand{\sder}{{\widecheck\partial}}

\newcommand{\ip}[1]{\langle {#1}\rangle}

\newcommand{\bip}[1]{\big\langle {#1}\big\rangle}
\newcommand{\Bip}[1]{\Big\langle {#1}\Big\rangle}

\begin{document}
\markboth{\scriptsize{CM \version}}{\scriptsize{CM September 2, 2016}}
\def\Z{{\mathbb Z}}
\def\R{{\mathbb R}}
\def\C{{\mathbb C}}
\def\sg{\sigma}
\def\S{\mathcal{S}}

\newcommand{\E}{{\mathcal E}}
\def\Ex{{\mathbb E}}


\renewcommand{\a}{\alpha}
\renewcommand{\b}{\beta}
\newcommand{\ga}{\gamma}
\newcommand{\Ga}{\Gamma}
\renewcommand{\d}{\delta}
\newcommand{\e}{\varepsilon}
\renewcommand{\l}{\lambda}
\newcommand{\om}{\omega}
\renewcommand{\O}{\Omega}
\newcommand{\Om}{\Omega}
\renewcommand{\th}{\theta}
\newcommand{\VV}{\mathbf{V}}
\newcommand{\UU}{\mathbf{U}}

\definecolor{jan}{rgb}{0.0,0.3,0.8}

\def\AAA{\color{blue}}   
\def\JJJ{\color{red}}
\def\MMM{\color{red}}
\def\EEE{\color{black}\normalsize}
\def\DDD{\color{magenta}\footnotesize}

\makeatletter
\DeclareRobustCommand\widecheck[1]{{\mathpalette\@widecheck{#1}}}
\def\@widecheck#1#2{%
    \setbox\z@\hbox{\m@th$#1#2$}%
    \setbox\tw@\hbox{\m@th$#1%
       \widehat{%
          \vrule\@width\z@\@height\ht\z@
          \vrule\@height\z@\@width\wd\z@}$}%
    \dp\tw@-\ht\z@
    \@tempdima\ht\z@ \advance\@tempdima2\ht\tw@ \divide\@tempdima\thr@@
    \setbox\tw@\hbox{%
       \raise\@tempdima\hbox{\scalebox{1}[-1]{\lower\@tempdima\box
\tw@}}}%
    {\ooalign{\box\tw@ \cr \box\z@}}}
\makeatother


\newcommand{\beq}{\begin{equation}}
\newcommand{\eeq}{\end{equation}}
\newcommand{\bal}{\begin{aligned}}
\newcommand{\eal}{\end{aligned}}
\newcommand{\ben}{\begin{enumerate}}
\newcommand{\beni} {\begin{enumerate}[(i)]}
\newcommand{\een}{\end{enumerate}}
\newcommand{\bit}{\begin{itemize}}
\newcommand{\eit}{\end{itemize}}
\newcommand{\beqw}{\begin{equation*}}
\newcommand{\eeqw}{\end{equation*}}
\newcommand{\bthm}{\begin{theorem}}
\newcommand{\ethm}{\end{theorem}}
\newcommand{\bpr}{\begin{proposition}}
\newcommand{\epr}{\end{proposition}}
\newcommand{\ble}{\begin{lemma}}
\newcommand{\ele}{\end{lemma}}
\newcommand{\blem}{\begin{lemma}}
\newcommand{\elem}{\end{lemma}}
\newcommand{\bpf}{\begin{proof}}
\newcommand{\epf}{\end{proof}}
\newcommand{\bex}{\begin{example}}
\newcommand{\eex}{\end{example}}
\newcommand{\bre}{\begin{example}}
\newcommand{\ere}{\end{example}}
\newcommand{\bma}{\begin{bmatrix}}
\newcommand{\ema}{\end{bmatrix}}
\newcommand\T{{\mathbb T}}

\newcommand{\ddt}{\frac{\mathrm{d}}{\mathrm{d}t}}
\newcommand{\ddtr}{\frac{\mathrm{d}^+}{\mathrm{d}t}}
\newcommand{\ddhr}{\frac{\mathrm{d}^+}{\mathrm{d}h}}
\newcommand{\ddtt}{\frac{\mathrm{d^2}}{\mathrm{d}t^2}}
\newcommand{\ddr}{\frac{\mathrm{d}}{\mathrm{d}r}}


\renewcommand{\aa}{{\boldsymbol\alpha}}
\newcommand{\Dom}{{\mathsf D}}
\newcommand{\wt}{\widetilde}
\newcommand{\one}{{{\bf 1}}}
\newcommand{\embed}{\hookrightarrow}
\newcommand{\Null}{\mathsf{N}}
\newcommand{\rank}{\mathrm{rank}\;}
\renewcommand{\Re}{{\rm{Re}}\;}
\renewcommand{\Im}{{\rm{Im}}\;}
\newcommand{\sgn}{{\rm sgn}}
\renewcommand{\hat}{\widehat}
\newcommand{\ot}{\otimes}
\newcommand{\hN}{\widehat\Num}
\newcommand{\G}{\Gamma}
\renewcommand{\phi}{\varphi}
\def\bh{{\bf h}}
\def\tand{\quad{\rm and}\quad}
\def\L{\mathcal{L}}\def\bk{\rangle\langle}
\newcommand{\cA}{{\mathord{\mathscr A}}}
\def\ib{\iota_\beta}
\def\H{\mathcal{H}}
\def\E{\mathcal{E}}
\def\Eb{\mathcal{E}_\beta}
\def\rb{\sigma_\beta}
\newcommand{\bx}{{\boldsymbol{x}}}
\newcommand{\cM}{\mathcal{M}}
\newcommand{\aA}{\mathcal{A}}
\newcommand{\aB}{\mathcal{B}}
\newcommand{\cG}{\mathcal{G}}
\newcommand{\cP}{{\mathord{\mathscr P}}}
\newcommand{\cQ}{{\mathord{\mathscr Q}}}
\newcommand{\cE}{{\mathcal{E}}}
\newcommand{\cL}{{\mathord{\mathscr L}}}
\newcommand{\cK}{{\mathord{\mathscr K}}}
\newcommand{\cB}{{\mathord{\mathscr B}}}
\newcommand{\cJ}{{\mathcal{J}}}
\newcommand{\cC}{{\mathcal{C}}}

\newcommand{\B}{{\mathcal{B}}}
\newcommand{\calS}{{\mathcal{S}}}
\newcommand{\F}{{\mathcal{F}}}
\newcommand{\fH}{{\mathfrak{H}}}
\newcommand{\bbA}{{\bf A}}
\newcommand{\Cln}{\mathfrak{C}^n}  
\newcommand{\err}{{r}}

\newcommand{\PX}{\cP(\cX)}
\newcommand{\PXs}{\cP_*(\cX)}
\newcommand{\hrhom}{\widehat{\rho_{-j}}}
\newcommand{\hrhop}{\widehat{\rho_{+j}}}
\newcommand{\hrhopm}{\widehat{\rho_{\pm j}}}
\newcommand{\hrhomp}{\widehat{\rho_{\mp j}}}
\newcommand{\lrho}{\check{\rho}}
\newcommand{\ulrho}{\breve{\rho}}
\newcommand{\ulV}{\breve{V}}
\newcommand*{\bdot}[1]{%
  \accentset{\mbox{\large\bfseries .}}{#1}}

\title
[Gradient flow for quantum Markov semigroups]
{Gradient flow and entropy inequalities for quantum Markov semigroups with detailed balance}

\author{Eric A. Carlen}
\address{Department of Mathematics\\ 
Hill Center\\
Rutgers University\\
110 Frelinghuysen Road\\
Piscataway\\
NJ 08854-8019\\
USA}
\email{carlen@math.rutgers.edu}

\author{Jan Maas}
\address{
Institute of Science and Technology Austria (IST Austria)\\
Am Campus 1\\ 
3400 \newline Klos\-ter\-neu\-burg\\ 
Austria}
\email{jan.maas@ist.ac.at}

\begin{abstract}
We study a class of ergodic quantum Markov semigroups on finite-dimensional unital $C^*$-algebras. These semigroups have a unique stationary state $\sigma$, and we are concerned with those that satisfy a quantum detailed balance condition with respect to $\sigma$. 
We show that the evolution on the set of states that is given by such a quantum Markov semigroup is gradient flow for the relative entropy with respect to $\sigma$ in a particular Riemannian metric on the set of states. 
This metric is a non-commutative analog of the $2$-Wasserstein metric, and in several interesting cases we are able to show, in analogy with work of Otto on gradient flows with respect to the classical $2$-Wasserstein metric, that the relative entropy is strictly and uniformly convex with respect to the Riemannian metric introduced here. 
As a consequence, we obtain a number of new inequalities for the decay of relative entropy for ergodic quantum Markov semigroups with detailed balance. 
\end{abstract}

\maketitle

\medskip
\centerline{Keywords: quantum Markov semigroup, entropy, detailed balance, gradient flow.}

\centerline{Subject Classification Numbers: 46L57, 81S22, 34D05, 47C90}


\maketitle

\section{introduction}  Let $\aA$ be a finite-dimensional $C^*$-algebra with unit $\one$. 
We may identify $\aA$  with a $C^*$-subalgebra of $\cM_n(\C)$, the $C^*$-algebra of 
$n \times n$ matrices, for some $n$.   
In finite dimension, there is no difference between weak and norm closure, and so 
$\aA$ is also a von Neumann algebra.  
A {\em Quantum Markov Semigroup} (QMS) is a continuous one-parameter semigroup of 
linear transformations $(\cP_t)_{t \geq 0}$ on $\aA$ such that for each $t\ge 0$, 
$\cP_t$ is completely positive and $\cP_t\one = \one$. 
Associated to any QMS $\cP_t = e^{t\cL}$, is the dual semigroup $\cP^\dagger_t$ 
acting on $\Dens_+(\aA)$, the set of faithful states of $\aA$.  
(When there is no ambiguity, we simply write $\Dens_+$.)
The QMS $\cP_t$ is {\em ergodic} in case $\one$ spans the eigenspace of $\cP_t$ for the eigenvalue $1$. 
In that case, there is a unique invariant state $\sigma$.
While $\sigma$ need not be faithful, a natural projection operation allows us to assume, 
effectively without loss of generality, that $\sigma\in \Dens_+(\aA)$. Characterizations of the 
generators of quantum Markov semigroups on the $C^*$-algebra $\cM_n(\C)$ 
of all $n\times n$ matrices were given at the same time by Gorini, Kossakowski and Sudershan 
\cite{GKS76}, and by Lindblad \cite{Lin76} in a more general setting (but still assuming 
norm continuity of the semigroup). Such semigroups are often called Lindblad semigroups. 

The notion of {\em detailed balance} in the theory of classical Markov processes has 
several different quantum counterparts, as discussed below.
One of these is singled out here, with a full discussion of how it relates to other 
variants and why it is physically natural.
Suffice it to say here that, as we shall see, the class of ergodic QMS that satisfy the 
detailed balance condition includes a wide variety of examples arising in physics.

The set of faithful states $\Dens_+(\aA)$ may be identified with the set of invertible density matrices 
$\sigma$ on $\C^n$ that belong to $\aA$, as is recalled below. 
For $\rho,\sigma \in \Dens_+(\aA)$, the {\em relative entropy of $\rho$ with respect to $\sigma$} is the functional
\begin{equation}\label{relent}
	D(\rho \| \sigma) = \tr[\rho(\log \rho - \log \sigma)]\ .
\end{equation}

We show in Theorem~\ref{gradflowentgen} that associated to any QMS 
$\cP_t = e^{t\cL}$ satisfying detailed balance, there is a Riemannian metric $g_\cL$ on 
$\Dens_+$ such that the flow on 
$\Dens_+$ induced by the dual semigroup $\cP_t^\dagger$ is 
{\em gradient flow} for the metric $g_{\cL}$ of the relative entropy 
$D(\cdot \|\sigma)$ with respect to the invariant state $\sigma\in \Dens_+$.  
In several cases, we shall show that the relative entropy functional is geodesically convex for 
$g_\L$, and as a consequence, we shall deduce a number of functional 
inequalities that are useful for studying the evolution governed by $\cP_t$. 
In particular, we shall deduce several sharp relative entropy dissipation inequalities. 
Some of these are new; see e.g. Theorem~\ref{BoseEnt} and Theorem~\ref{FermiEnt}. 

The Riemannian distance corresponding to $g_\cL$ will be seen to be a very natural analog of the 
$2$-Wasserstein distance on the space of probability densities on $\R^n$ \cite[Chapter 6]{V}. 
Otto showed \cite{O01} that a large number of classical evolution equations could be viewed as gradient flow 
in the $2$-Wasserstein metric for certain functionals, and that when the functionals were geodesically uniformly convex for this geometry, a host of useful functional inequalities were consequently valid. 
This is for instance the case when the functional driving the flow is classical relative entropy with respect to a Gaussian reference measure.
In this case, the flow is given by the classical Ornstein-Uhlenbeck semigroup, described by the Fokker-Planck equation with linear drift. 
One of the inequalities that is a consequence of the uniform geodesic convexity of the relative entropy is the sharp bound on entropy dissipation for solutions of  the Fokker-Planck equation. 

The present paper also greatly extends our previous paper \cite{CM13} in which we obtained a gradient flow structure for the Fermi Ornstein-Uhlenbeck semigroup. 
In particular, we prove a sharp geodesic convexity result for the von Neumann entropy in this setting, thereby solving a problem that was left open in \cite{CM13}. Thus, our results can be viewed as a non-commutative extension of Otto's investigation of classical gradient flows. 

Mielke \cite{Mie13QM,Mie15QM} investigated related variational formulations for dissipative quantum systems based on \"Ottinger's so-called GENERIC framework \cite{Oet,Oet10}. 
However, the gradient flow structure considered in Mielke's papers is in general different from the one introduced in the present paper, as the approach in \cite{Mie13QM,Mie15QM} gives rise to nonlinear evolution equations that are different from the linear Lindblad equations that we obtain here. Also, Junge and Zeng \cite{JuZe15} have recently developed an approach to some non-commutative functional inequalities involving a non-commutative analog of the $1$-Wasserstein metric.

We restrict our attention to the case that $\aA$ is a finite-dimensional $C^*$-algebra because this setting already includes many examples of physical interest, and we wish to explain our methods in a way that does not encumber them with the host of technical and topological difficulties that would follow in an infinite-dimensional setting.
For example, in an infinite-dimensional setting, it would matter that $\aA$ is a von Neumann algebra and not only a $C^*$-algebra with a unit, so that $\cA$ would have a pre-dual, and then what we call $\Dens_+(\aA)$ here would be the set of normal states; 
i.e., states that are continuous with respect to the weak-$*$ topology on $\aA$, and we would require that $\cP_t$ be normal for each $t$; 
i.e., continuous with respect to the weak-$*$ topology on $\aA$. 
The innocent formula $\cP_t = e^{t\cL}$ would require closer scrutiny, and so forth. 
Many of the issues involved in extending our results to a more general infinite-dimensional setting are standard but not all of them. 
This will be done elsewhere.
We also refer to a forthcoming paper for a  treatment of more general transport metrics, general entropy functionals, and their geodesic convexity properties \cite{CM16}.

Here instead, we present a development of our methods that 
will be readily understood by people who are interested in the many physically interesting finite-dimensional examples.  
We shall also show that our methods may be applied in the infinite-dimensional setting to particular examples without first generalizing the whole theory. 
We illustrate this with the family of Bose Ornstein-Uhlenbeck semigroups, which are always infinite-dimensional since all non-trivial representations of the Canonical Commutation Relations are infinite-dimensional. 
We use our methods to prove a new sharp relative entropy dissipation inequality for this family that had recently been conjectured in \cite{HKV16}; see Theorem~\ref{BoseEnt}. We also prove some other sharp inequalities of this type; e.g., Theorem~\ref{FermiEnt}. 

While this introduction has, in the interest of brevity, used a considerable amount of terminology without explanation, the rest of the paper is elementary and quite self-contained. 
We do assume a basic familiarity with $C^*$-algebras, von Neumann algebras and completely positive maps. 
Many readers will have this familiarity, but for background on $C^*$ and von Neumann algebras, we refer to Sakai \cite{Saka71}, and for completely positive maps to Paulsen \cite{Paul02}. 
Terminology and facts that are used here without explanation can be found in these references. 

We close the introduction with some preliminary definitions and by establishing some notation that will be in use throughout the paper.

\subsection{Preliminary definitions and notation}

$\aA$ will always denote a finite-dimensional $C^*$-algebra with unit $\one$, regarded as subalgebra of some matrix algebra $\cM_n(\C)$. 
The center of $\aA$ is generated by a single self-adjoint element $Z\in \aA$, and its spectral projections yield 
a decomposition of $\aA$ into a finite direct sum of {\em factors}:
Let $\{\lambda_1, \dots,\lambda_p\}$ for some $1 \leq p \leq n$ be the distinct eigenvalues of $Z$, and let $\H_j$ be the eigenspace with eigenvalue $\lambda_j$. 
Then each $\H_j$ is invariant under $\aA$, and letting $\aA_j$ denote the restriction of $\aA$ to $\H_j$, $\aA_j$ has a trivial center; that is, $\aA_j$ is a factor. 

The structure of finite-dimensional factors is well-known: $\H_j$ is unitarily equivalent to $\C^{\ell_j}\otimes \C^{r_j}$, and there is a unitary equivalence in which each $A\in \aA_j$ takes the form $I_{\ell_j}\otimes \widehat{A}$ where $\widehat{A} \in \cM_{r_j}(\C)$. 
That is, each factor may be identified in a natural way with a full matrix algebra. 
However, not all finite-dimensional $C^*$-algebras arising in mathematical physics are factors. 
For example, a Clifford algebra with an odd number $k$ of generators, arising in the description of a system of $k$ Fermi degrees of freedom, is not a factor. 
This example will be discussed in detail later on.

Let $\tau$ denote the normalized trace on $A$ given by $\tau(A) = \tr[A]/\tr[\one]$ for all $A\in \aA$. 
Then $\tau$ is a faithful state on $\aA$. 
Let $\fH_\aA$ denote the Hilbert space formed by equipping $\aA$ with the \emph{GNS inner product} determined by $\tau$. 
That is, for all $A,B\in \aA$,
\begin{equation*}
\langle A,B\rangle_{\fH_{\aA}} = \tau[A^*B]\ .
\end{equation*}
We could use this inner product to identify $\fH_\aA$ with the space of linear functionals on $\aA$, but to keep contact with the physics literature, we do something slightly different: 
We use the {\em un-normalized trace} to write the general linear functional $\phi$ on $\aA$ in the form
\begin{equation*}
\phi(B) = \tr[AB]
\end{equation*} 
for some $A\in \aA$. 
It is easy to see that $\phi$ is a positive linear functional (in the sense that $\phi(B) \geq 0$ whenever $B \geq 0$) if and only if $A\geq 0$ in $\aA$.  
It follows that the set of faithful states on $\aA$ may be identified with the set of 
strictly positive elements $A$ of $\aA$ such that $\tr[A] =1$.  
Since $A\in \aA$ is strictly positive in $\aA$ if and only if it is strictly positive in $\cM_n(\C)$, we may identify the set of faithful states on $\aA$ with the set of invertible density matrices $\rho$ on $\cM_n(\C)$ {\em that belong to the subalgebra $\aA$}. 
In the following, we always write $\Dens_+$ to denote this set, whether we think of it as a set of faithful states or as a set of density matrices. 

\begin{defi}[Modular operator and modular group]\label{modu} Let $\sigma\in \Dens_+$. 
Define a linear operator $\Delta_\sigma$ on $\fH_{\aA}$, or, what is the same thing, on $\aA$, by
\begin{equation*}
\Delta_\sigma(A) = \sigma A \sigma^{-1}
\end{equation*}
for all $A\in \aA$. $\Delta_\sigma$ is called the {\em modular operator}.  The {\em modular generator} is the self-adjoint element $h \in \aA$ given by
\begin{equation*}
h = -\log \sigma\ .
\end{equation*}
The  {\em modular automorphism group} $\alpha_t$ on $\aA$ is the group defined by
\begin{equation}\label{modula}
\alpha_t(A) = e^{it h} A e^{-it h}
\end{equation}
\JJJ for $t \in \C$. \EEE Note that $\Delta_\sigma = \alpha_i$. 
\end{defi}

The modular operator and the modular automorphism group are central to the characterization of an important class of quantum Markov semigroups. 

First, Davies \cite{Da74} identified a class of quantum Markov semigroups that he rigorously showed to arise from coupling a finite quantum system with internal Hamiltonian $h$ to an infinite fermion heat bath and then taking a weak coupling limit. 
The whole class of quantum Markov semigroups he found has the property that the semigroup commutes with the modular automorphism group $\alpha_t$ given by (\ref{modula}) where $h$ is the internal Hamiltonian, and then $\sigma = e^{-h}/ \tr[e^{-h}]$. 
(Note that adding a scalar constant to $h$ has no effect on (\ref{modula}).) 
Of course an operator on $\fH_\aA$ commutes with the modular group if and only if it commutes with the modular operator. 
  
Second, Alicki \cite{A78} showed that  commutativity with respect to the modular operator is central to one natural extension of the notion of {\em detailed balance} from classical Markov chains to the quantum setting. 
Before explaining this fact, which is important for our work here, we first introduce some more terminology and notation. 

A linear operator $\cK$ on $\aA$ is {\em positivity preserving} in case $\cK A \geq 0$ whenever $A\geq 0$.
A linear operator $\cK$ on $\aA$ is {\em self-adjointness preserving} in case $(\cK A)^* = \cK A^*$, or, equivalently, 
in case $\cK A$ is self-adjoint whenever $A$ is self-adjoint. 
Evidently, any positivity preserving operator is self-adjointness preserving. 
When $\cP_t = e^{t\cL}$ is a self-adjointness preserving semigroup, then its generator $\cL = 
\lim_{t\to 0}t^{-1}(\cP_t - I)$ is self-adjointness preserving as well.

Let $\sigma \in \Dens_+$ and note that  $(\Delta_\sigma A)^* = \Delta_{\sigma}^{-1}(A^*)$ for all $A\in \aA$. Moreover,  for all $A,B\in \aA$,
$$\tr[A^* \Delta_\sigma B] = \tr[(\Delta_\sigma A)^* B] \tand \tr[A^*\Delta_\sigma A] = \tr [|\sigma^{1/2}A\sigma^{-1/2}|^2] \ ,$$
so that $\Delta_\sigma$ is a positive operator on  $\fH_{\aA}$. 
A dagger $\dagger$ will be used to denote the adjoint with respect to the inner product in $\fH_{\aA}$, or, 
what is the same, with respect to the Hilbert-Schmidt inner product. 
We will encounter many other inner products on $\aA$, but the GNS inner product associated to $\tau$ has a special role. 
We may then rewrite one of the conclusions from just above as $\Delta_{\sigma}^\dagger = \Delta_\sigma$. 

Because of the self-adjointness of $\Delta_\sigma$
there is an orthonormal basis $\{E_1,\dots, E_m\}$ of $\fH_\aA$ consisting of eigenvectors of $\Delta_\sigma$. 
Since $\Delta_\sigma \one = \one$, we may always assume that $E_1 = \one$. 
In this case, $\tr[E_\gamma] =\tau(E_\gamma) = 0$ for all $\gamma > 1$. 
Furthermore, since $\Delta_\sigma$ is strictly positive, all  eigenvalues of $\Delta_\sigma$ are strictly positive, hence we may write them in the form $e^{-\omega_\gamma}$. 
Since $(\Delta_\sigma A)^* = \Delta_\sigma^{-1}A^*$, it follows that for all $E \in \fH_\aA$,
\begin{equation}\label{conjeig}
\Delta_\sigma E = e^{-\omega}E \quad \iff \quad \Delta_\sigma E^* = e^{\omega}E^*\ .
\end{equation}
In particular, if for $\omega \neq 0$, $e^{-\omega}$ is an eigenvalue of $\Delta_\sigma$, then so is $e^{\omega}$, and the two eigenspaces are orthogonal and have the same dimension. 
It follows that there exists an orthonormal basis of $\fH_\aA$ with the properties listed in the next definition:

\begin{defi}[Modular basis]\label{modbadef}  
Let $\aA$ be a finite-dimensional $C^*$-algebra and let $\sigma\in \Dens_+(\aA)$.
Then there exists an orthonormal basis $\{E_1,\dots, E_m\}$ of $\fH_\aA$ with the following properties:

\smallskip
\noindent{\it (i)} $\{E_1,\dots, E_m\}$ consists of eigenvectors of $\Delta_\sigma$.

\smallskip
\noindent{\it (ii)} $E_1 = \one$.

\smallskip
\noindent{\it (iiI)} $\{E_1,\dots, E_m\} = \{E_1^*,\dots, E_m^*\}$

\end{defi}

\section{Detailed balance}

There is a large literature on the {\em detailed balance condition} in a quantum setting, starting with the work of Agarwal \cite{Ar73}, which initiated a number of investigations in the 1970's including
\cite{A78, CW76,KFGV,SpLe77}. Already in these papers, one finds several different, and non-equivalent,  definitions. A number of more recent investigations \cite{AFQ16, FR15, FU10,  MS98, TKRWV}  have added to the variety
of meanings attached to this term. We therefore carefully explain the context of the definition that we use here, and why it is the most natural for our purposes. 

Let $P_{i,j}$ be the Markov transition matrix for a Markov chain on a finite state space 
$S$ with elements $\{x_1,\dots,x_n\}$.
Suppose that $\sigma$ is a probability density on $S$ that is invariant under this transition function: 
$\sigma_j = \sum_{i=1}^n\sigma_i P_{i,j}$ for all $i$.  The transition matrix satisfies the {\em detailed balance condition with respect to $\sigma$} in case
\begin{equation}\label{clasde1}
\sigma_i P_{i,j} =    \sigma_j P_{j,i}  \qquad {\rm for\ all}\ i,j\ .
\end{equation}
Let $X_n$ be the Markov process started from the initial distribution $\sigma$, so that  the process is stationary.  
Let ${\rm Pr}$ be measure on the path space of the process. Then (\ref{clasde1}) is equivalent to
\begin{equation*}
{\rm Pr}\{X_n =i,X_{n+1}= j\} = {\rm Pr}\{X_n =j,X_{n+1} = i\}   \qquad {\rm for\ all}\ i,j\ {\rm and \ all}\ n\ .
\end{equation*}
In other words, $(X_n, X_{n+1})$ has the same joint distribution as $(X_{n+1}, X_n)$, so that (\ref{clasde1}) characterizes {\em time reversal invariance}. 
There is another characterization of (\ref{clasde1}) in terms of self-adjointness:  The matrix $P$ is self-adjoint on $\C^n$ equipped with the inner product 
\begin{equation}\label{classical}
\langle f,g\rangle_\sigma = \sum_{k=1}^n \sigma_k \overline{ f_k} g_k \ ,
\end{equation}
if and only if (\ref{clasde1}) is satisfied.

There are a number of different ways one might generalize the inner product (\ref{classical}) to 
the quantum setting, and these give different notions of self-adjointness.
Let $\cP_t$ be a QMS on $\aA$. 
A state $\sigma\in \Dens_+(\aA)$ is {\em invariant under $\cP_t^\dagger$} in case  $\cP_t^\dagger \sigma = \sigma$ for all $t \geq 0$, or equivalently, $\cL^\dagger \sigma = 0$.
These conditions are equivalent to $\tr[\sigma \cP_t(A)] = \tr[\sigma A]$ for all $t>0$ and all $A\in \aA$, or equivalently, $\tr[\sigma \cL(A)] = 0$ for all $A\in \aA$. 

\begin{defi}[Compatible inner product]   
An inner product $\langle \cdot, \cdot \rangle$ is {\em compatible} with $\sigma\in \Dens_+(\aA)$ 
in case for all $A\in \aA$, 
$\tr[\sigma A]  = \langle \one, A\rangle$.
\end{defi}

If a quantum Markov semigroup $\cP_t$ is self-adjoint with respect to an inner product $\langle \cdot,\cdot\rangle$ that is compatible with $\sigma\in \Dens_+$, then for all $A\in \aA$, 
$$\tr[\sigma A] = \langle \one ,A\rangle =  \langle \cP_t \one,A\rangle =\langle \one ,\cP_t A\rangle = \tr[\sigma\cP_tA]\ ,$$
and thus $\sigma$ is invariant under $\cP_t^\dagger$.

\begin{defi}
Let $\sigma \in \Dens_+$ be a non-degenerate density matrix. 
For each $s\in \R$, and each $A,B\in \aA$, define
\begin{equation}\label{detailB}
 \langle A,B\rangle_s  
  =  \tr[(\sigma^{(1-s)/2} A\sigma^{s/2})^* (\sigma^{(1-s)/2} B \sigma^{s/2}) ]  = \tr[ \sigma^s A^* \sigma^{1-s} B]\ .
\end{equation} 
\end{defi}

Evidently each of these inner products is compatible with $\sigma$. 
At $s=1$, this is the \emph{GNS inner product} associated to the state $\varphi(A) = \tr[\sigma A]$. 
The value $s=1/2$ is also special; 
$\langle \cdot,\cdot\rangle_{1/2}$ is called the \emph{KMS inner product}. A number of its properties are developed in \cite{Pe84}.

The inner products in (\ref{detailB}) can be written as 
\begin{equation}\label{detailBB}
\langle A,B\rangle_s = \tr[ A^* \Delta_{\sigma}^{1-s}B\sigma]\ .
\end{equation}
More generally, given any  function $f:(0,\infty) \to (0,\infty) $, define
\begin{equation}\label{detailB3}
\langle A,B\rangle_f = \tr[ A^* f(\Delta_{\sigma})B \sigma]\ .
\end{equation}
Note that $\langle \cdot,\cdot\rangle_1$ is the $\sigma$-GNS inner product whether $1$ is interpreted as a number, as in \eqref{detailBB}, or as the constant function $f(t) = 1$, as in \eqref{detailB3}.

Let $R_A$ denote right multiplication by $A$, and define $\Omega_\sigma^f = R_\sigma \circ  f(\Delta_\sigma)$.
Then another way to write (\ref{detailB3}) is $\langle A,B\rangle_f  = \tr[A^*\Omega_\sigma^f B]$.  
For all linear operators $\cK$ on $\aA$, 
\begin{align*}
\langle A, \cK B\rangle_f = 
 \langle [\Omega_\sigma^f]^{-1} \cK^\dagger(\Omega_\sigma^f A), B\rangle_f \ .
\end{align*}
It follows immediately that $\cK$ is self-adjoint with respect to the inner product $\langle \cdot,\cdot \rangle_f$  if and only if 
\begin{align*}
\Omega_\sigma^f \circ \cK = \cK^\dagger \circ \Omega_\sigma^f\ .
\end{align*}

\begin{remark}
For $f(t) = t^s$, 
the adjoint $\cK'$  of $\cK$ with respect to the inner product $\langle \cdot,\cdot\rangle_s$ is
$$\cK' B =  \sigma^{s-1} \cK^\dagger  (\sigma^{1-s} B \sigma^s)\sigma^{-s}\ .$$
For $s=1/2$, this reduces to $\cK' B = \sigma^{-1/2} \cK^\dagger  (\sigma^{1/2} B \sigma^{1/2})\sigma^{-1/2}$. Since
$\cK^\dagger$ is positivity preserving if and only if $\cK$ is positivity preserving, it is evident that for $s=1/2$, 
$\cK'$ is positivity preserving if and only if $\cK$ is positivity preserving. However, for other values of $s$, 
$\cK'$ need not be positivity preserving when $\cK$ is. This is one feature that sets the KMS inner product 
apart from the inner products
$\langle \cdot, \cdot \rangle_s$ for other values of $s$. 
\end{remark}

In \cite[Definition 16]{TKRWV}, detailed balance is defined in terms of self-adjointness with respect to 
$\langle \cdot, \cdot \rangle_f$, yielding a number of {\it a priori} different notions of detailed balance 
depending on the choice of $f$.   
The example following Proposition 18 in \cite{TKRWV} shows that different choices of $f$ can yield distinct classes of self-adjoint operators, and hence distinct notions of detailed balance.   
Specifically, it is shown in \cite{TKRWV} that self-adjointness with respect to $\langle \cdot, \cdot\rangle_{1/2}$ is not the same as self-adjointness with respect to $\langle \cdot, \cdot \rangle_f$ where $f(t) = (1+t)/2$, corresponding to the 
{\em Bures metric}, as discussed in \cite{TKRWV}. 
The authors of \cite{TKRWV} conclude: ``The family of quantum detailed balance conditions is therefore much richer than the classical counterpart''.  

However, it turns out that self-adjointness with respect to the GNS inner product, or indeed with respect to $\langle \cdot,\cdot\rangle_{s}$ for any $s\neq 1/2$ implies self-adjointness with respect to $\langle \cdot,\cdot\rangle_{f}$ for all $f$. 
The argument leading to this conclusion, for which we have found no reference, is simple and will prove useful.

The following is a simple variant on the well-known KMS symmetry condition:

\begin{lm}\label{lem:alpha-id} For $s\in \R$, let $\langle \cdot, \cdot \rangle_s$ be the inner product defined in (\ref{detailB}). 
Then for all $t\in \R$ and all $A,B\in \aA$, 
\begin{equation}\label{modulaC}
\langle \alpha_{it}(A), B\rangle_s = \langle A,B\rangle_{s-t} = \langle A, \alpha_{it}(B)\rangle_s \ .
\end{equation}
In particular, $\alpha_{it}$ is self-adjoint with respect to $\langle \cdot, \cdot \rangle_s$.
\end{lm}

\begin{proof} 
Using the definitions we obtain
\begin{align*}
\langle \alpha_{it}(A), B\rangle_s 
= \tr[ \sigma^{s}(\sigma^t A \sigma^{-t})^* \sigma^{1-s}B]  
= \tr[ \sigma^{s-t} A^* \sigma^{1-s+t}B] 
= \langle A,B\rangle_{s-t} 
= \langle A, \alpha_{it}(B)\rangle_{s} \ ,
	\end{align*}
which is the desired identity	
\end{proof}

The following lemma is a result of Alicki \cite{A78} for $s=1$, and the generalization  to $s\in [0,1]$, $s\neq 1/2$, can be found in \cite[Propositon 8.1]{FU07}. The following short proof, a simple adaptation of Alicki's argument, is included for the reader's convenience. 

\begin{lm}\label{Al2}  Let $\sigma \in \Dens_+$ be a non-degenerate density matrix, and let $s\in [0,1]$, $s\neq 1/2$. 
Let $\cK$ be any operator on $\aA$ that is self-adjoint with respect to $\langle \cdot, \cdot\rangle_s$ and also preserves self-adjointness. 
Then $\cK$ commutes with $\alpha_t$, for all $t$, real and complex. 
\end{lm}

\begin{proof}
For any $A,B\in \aA$,
\begin{eqnarray*}
\langle  \cK \alpha_{i(2s-1)}(A), B\rangle_s 
&=& \tr[ \sigma^s (\cK (\sigma^{2s-1}A \sigma^{1-2s}))^* \sigma^{1-s} B]\nonumber \\
&=& \tr[ \sigma^s (\sigma^{2s-1}A \sigma^{1-2s})^*  \sigma^{1-s}\cK ( B)]\nonumber\\
&=& \tr[  \sigma^{1-s}A^* \sigma^{s}  \cK ( B)]
= \tr[\sigma^{s}  (\cK ( B^*))^*  \sigma^{1-s}A^* ]\nonumber\\
&=& \tr[\sigma^{s}  B  \sigma^{1-s}\cK (A^*) ] = \tr[\sigma^{1-s}( \cK(A))^* \sigma^s B]  = \langle \cK (A), B\rangle_{1-s}\nonumber
\end{eqnarray*}
Since $s- (2s-1) = 1-s$, (\ref{modulaC}) yields $\langle \cK (A), B\rangle_{1-s} = \langle  
\alpha_{i(2s-1)}(\cK (A)), B\rangle_s$. 
As $B$ is arbitrary,  $\alpha_{i(2s-1)} \cK  = \cK \alpha_{i(2s-1)}$. 
Since $\cK $ commutes with $\alpha_{i(2s-1)}$, it commutes with every polynomial in the 
positive self-adjoint operator $\alpha_{i(2s-1)}= \Delta_\sigma^{2s-1}$, and hence with 
$ f(\Delta_\sigma^{2s-1})$ for every function $f$. 
In particular, $\cK$ commutes  with $\alpha_t$ for all $t$. 
\end{proof} 

\begin{remark} Davies \cite{Da74} has studied a class of quantum Markov semigroups 
that arise in a general model of an $n$-level quantum system coupled to an infinite heat bath. 
He studied the weak-coupling limit and gave conditions under which the 
weak-coupling limit produces a quantum Markov semigroup.  
This procedure always yields a semigroup that commutes with the evolution given by 
(\ref{modula}) where $h$ is the internal Hamiltonian of the $n$-level system.  
This is true whether or not the semigroup has a particular self-adjointness property. 
In view of Davies' result, it is natural to focus on quantum Markov semigroups that commute 
with the modular operator associated to their invariant states.
\end{remark}

\begin{remark}\label{halfdif} The condition that the generator $\cL$ of a quantum Markov semigroup  
$\cP_t = e^{t\cL}$ commutes with $\Delta_\sigma$ imposes strong restrictions on the structure of $\cL$. 
Consider the case $\aA = \cM_2(\C)$.
Let $\sigma\in \Dens_+$ have two distinct eigenvalues $\lambda_1, \lambda_2 > 0$. 
Let $\{\eta_1,\eta_2\}$ be an orthonormal basis of $\C^2$ consisting of eigenvectors of 
$\sigma$: $\sigma \eta_j = \lambda_j\eta_j$ for $j=1,2$. 
Then $\Delta_\sigma$ has three distinct eigenvalues $1$, 
$\lambda_1/\lambda_2$ and $\lambda_2/\lambda_1$. 
The latter two eigenvalues have one-dimensional eigenspaces spanned by $|\eta_1\bk \eta_2|$ and $|\eta_2\bk \eta_1|$ respectively. 
If $\cL$ commutes with $\Delta_\sigma$, then $|\eta_1\bk \eta_2|$ and $|\eta_2\bk \eta_1|$ must be eigenvectors of $\cL$ with eigenvalues, say, $\nu$ and $\tilde\nu$. 
Since $A := |\eta_1\bk \eta_2| + |\eta_2\bk \eta_1|$ is self-adjoint and $\cL$ is self-adjointness preserving, it follows that $\cL(A) = \nu  |\eta_1\bk \eta_2| + \tilde\nu |\eta_2\bk \eta_1|$ is self-adjoint, which implies that $\tilde\nu = \nu$. 
Moreover, since both $A$ and $\cL(A) = \nu A$ are self-adjoint, it follows that $\nu$ is real.
Thus $\cL$ has at most $3$ eigenvalues $1$, $\mu$ and $\nu$, and $\nu$ has multiplicity $2$ (or $3$ in case $\nu = \mu$). To summarize, it follows that 
\begin{align}\label{lsepc1}
	\cL(|\eta_1\bk \eta_2|) = \nu |\eta_1\bk \eta_2| \tand
	\cL(|\eta_2\bk \eta_1|) = \nu |\eta_2\bk \eta_1| \quad \text{with }
	\nu \in \R \ .
\end{align}

From here it is not hard to see that the value $s=1/2$ is genuinely exceptional in Lemma~\ref{Al2}. 
Suppose that $\cL$ is self-adjoint with respect to the $\sigma$-KMS inner product 
$\langle \cdot,\cdot\rangle_{1/2}$ for $\sigma\in \Dens_+$, where $\sigma$ has distinct eigenvalues and $\{\eta_1,\eta_2\}$ is an orthonormal basis of $\C^2$ consisting of eigenvectors of $\sigma$. 
It follows from the discussion above that if \eqref{lsepc1} is violated, $\cL$ does not commute with $\Delta_\sigma$.  
There are many such operators $\cL$ on $\cM_2(\C)$.
An explicit construction is given in Appendix B.
\end{remark}

\begin{lm} Let $\sigma \in \Dens_+$ be a non-degenerate density matrix. 
Let $\cK$ be any operator on $\aA$ such that $\cK \alpha_t = \alpha_t \cK$ for all $t$, or equivalently, $\Delta_\sigma \cK = \cK \Delta_\sigma$.
If $\cK$ is self-adjoint with respect to the inner product $\langle \cdot,\cdot\rangle_f$ for \emph{some} function $f:(0,\infty) \to (0,\infty)$, then the same holds for \emph{every} function $f:(0,\infty) \to (0,\infty)$. 
\end{lm}

\begin{proof}  Suppose that $\cK$ is self-adjoint with respect to the inner product $\langle \cdot,\cdot\rangle_f$ for some function $f:(0,\infty) \to (0,\infty)$. Let $g : (0,\infty) \to (0,\infty)$ be arbitrary, and write $h = g/f$. Since $\cK$ commutes with $\Delta_\sigma$, it also commutes with $h(\Delta_\sigma)$. Thus, for all $A,B\in \aA$,
\begin{align*}
	\langle A,\cK (B)\rangle_g
&	= \tr[\sigma A^* g(\Delta_\sigma) \cK (B)]
	= \tr[\sigma A^* f(\Delta_\sigma) \cK h(\Delta_\sigma) B]
	= \langle A,\cK h(\Delta_\sigma)B\rangle_f
\\&	= \langle \cK (A),h(\Delta_\sigma) B\rangle_f
	= \tr[\sigma \cK(A)^* f(\Delta_\sigma) h(\Delta_\sigma) B]
    = \langle \cK (A),B\rangle_g \ ,
\end{align*}
which is the desired result.
\end{proof}
We summarize some immediate consequences of the last lemmas  in a theorem:
\begin{thm}\label{detequiv}  
 Let $\sigma \in \Dens_+$ be a non-degenerate density matrix, and let $\cK$ be any operator on $\aA$. Then:

\smallskip
\noindent{\it (1)}  If $\cK$ is self-adjoint with respect to the GNS inner product $\langle \cdot,  \cdot \rangle_1$, 
and $\cK A^* = (\cK A)^*$ for all $A\in \aA$, 
then $\cK$ commutes with the modular automorphism group of $\sigma$ and moreover, $\cK$ is self-adjoint with respect to $\langle \cdot,  \cdot \rangle_f$ for all 
$f: (0,\infty) \to (0,\infty)$.

\smallskip
\noindent{\it (2)} If $\cK$ commutes with the modular automorphism group of $\sigma$, and if $\cK$ is self-adjoint with respect to 
$\langle \cdot,  \cdot \rangle_f$ for some $f: (0,\infty) \to (0,\infty)$, then $\cK$ is self-adjoint with respect to  $\langle \cdot,  \cdot \rangle_g$ for all $g: (0,\infty) \to (0,\infty)$.
\end{thm}

This brings us to the definition of detailed balance, which is one the various definitions that may be found in the physical literature.

\begin{defi}[Detailed balance] \label{detailedbalance} A QMS $\cP_t$ on $\aA$  satisfies the {\em detailed balance condition} with respect to 
$\sigma\in \Dens_+(\aA)$ in case for each $t>0$, $\cP_t$ is self-adjoint in the $\sigma$-GNS inner product $\langle \cdot, \cdot\rangle_1$. 
In this case $\sigma$ is invariant under 
$\cP_t^\dagger$, and we say that the QMS $\cP_t$ satisfies the $\sigma$-DBC. 
\end{defi} 

\begin{remark} Every QMS $\cP_t = e^{t\cL}$ is self-adjointness preserving: $\cP_t A^* = (\cP_t A)^*$. 
Thus, when $\cP_t$ satisfies the $\sigma$-DBC, it follows from Theorem \ref{detequiv} that 
\begin{equation}\label{modcom}
\alpha_{t'}(\cP_t(A)) = \cP_t(\alpha_{t'}(A))
\end{equation}
for all $t,t'$ and all $A\in \aA$.
This crucial observation is due to Alicki \cite{A78}, and it means that $\cP_t$ commutes with the time-translation governed by the Hamiltonian $h$ corresponding to the state 
$\sigma$. The condition (\ref{modcom}) may therefore be viewed as a quantum analog of time-translation 
invariance or {\em stationarity}. 
\end{remark}

Theorem~\ref{detequiv} asserts that in the presence of detailed balance,  
$\cP_t$  is self-adjoint with respect to a wide variety of inner products and also that $\cP_t$ commutes with the modular group.  
Moreover, under the condition that $\cP_t$ commutes with the modular group, self-adjointness 
with respect to any member of this wide family of inner products implies self-adjointness 
with respect to all of them. This is important in what follows.

The inner products defined just above include a number of inner products that arise 
naturally in the theory of operator algebras and mathematical physics.  
One that will be especially useful here  is the Bogoliubov-Kubo-Mori inner product:
\begin{equation*}
\langle A,B\rangle_{{\rm BKM}}  
= \int_0^1 \tr[ \sigma^{1-s}A^* \sigma^{s}B] {\rm d} s = 
\int_0^1 \tr[ \sigma A^* \Delta_\sigma^s B] {\rm d} s = \langle A,B\rangle_{f_{0}} \ ,
\end{equation*}
where
\begin{equation*}
f_{0}(t) = \int_0^1 t^s{\rm d}s =  \frac{t-1}{\log t} \ .
\end{equation*}
The BKM inner product arises naturally in statistical mechanics and our work here as follows: 
For a self-adjoint operator $h$ on a finite-dimensional Hilbert space $\H$, and $\beta > 0$, 
define the {\em Gibbs state} $\sigma_\beta$ by
\begin{equation*}
\sigma_\beta = \frac{1}{\tr[e^{-\beta h}]} e^{-\beta h}\ .
\end{equation*}
The {\em free energy} is the functional $\mathcal{F}(\beta,h) = 
\beta^{-1}\ln(\tr[e^{-\beta h}])$. A simple calculation using Duhamel's formula shows that for any self-adjoint $A$, 
\begin{equation*}
\frac{{\rm d}^2}{{\rm d}s^2}\mathcal{F}(\beta,h+sA)\bigg|_{s=0}  = \beta\left[ \langle A,A\rangle_{{\rm BKM}} - 
(\tr[\sigma_\beta A])^2\right]\ .
\end{equation*}
In particular, up to a factor of $\beta$, $\langle A,A\rangle_{{\rm BKM}}$ arises as the restriction of the Hessian of $\mathcal{F}$ to the space of self-adjoint matrices $A$ satisfying $\tr[\sigma_\beta A] =0$. 
It is well known that for fixed $\beta$, the function $h\mapsto \mathcal{F}(\beta,h)$ is the Legendre transform of the von Neumann entropy $S(\rho)$ on the set of density matrices. 
Since the gradients and Hessians of conjugate convex functions are inverse to one another, the Hessian of the free energy is the inverse of Hessian of the entropy; a fact that explains why the BKM inner product will arise naturally in the study of gradient flows of the relative entropy.

\section{Generators of quantum Markov semigroups satisfying detailed balance}

Since a QMS $\cP_t =e^{t\cL}$ on $\aA$ that satisfies the $\sigma$-DBC for some $\sigma\in \Dens_+(\aA)$
has a generator $\cL$ that commutes with the modular operator $\Delta_\sigma$, and since 
$\Delta_\sigma$ is positive with respect to the GNS inner product, $\Delta_\sigma$ and $\cL$ can be simultaneously diagonalized.  In the case $\aA = \cM_n(\C)$, the diagonalization of
$\Delta_\sigma$ reduces immediately to the diagonalization of $\sigma$:  Let $\sigma = e^{-h}$ be a density matrix 
on $\C^n$. Let $\{\eta_1,\dots,\eta_n\}$ be an orthonormal basis of $\C^n$ consisting of eigenvectors of $h = -\log \sigma$: 
$h\eta_j = \lambda_j \eta_j$. For $\alpha = (\alpha_1,\alpha_2) \in \{(i,j)\ : 1 \leq i ,j \leq n\}$, define numbers
$\omega_\alpha$ (called the {\em Bohr frequencies}) by
\begin{equation}\label{hform2}
\omega_{\alpha} = \lambda_{\alpha_1} - \lambda_{\alpha_2}\ ,
\end{equation}
and rank-one operators $F_\alpha$ given by
$F_\alpha = |\eta_{\alpha_1}\bk \eta_{\alpha_2}|$
where, using a standard physics notation, for $\eta,\xi\in \C^n$, $|\eta \bk \xi |$ is the rank-one operator sending $\zeta$ to $\langle \xi,\zeta\rangle_{\C^n}\eta$. 
Evidently
\begin{equation}\label{hform5}
\Delta_\sigma F_\alpha  = e^{-\omega_\alpha}F_\alpha \qquad{\rm and} \qquad F_\alpha^* = F_{\alpha'} \quad{\rm where}\quad  \alpha' = (\alpha_2,\alpha_1) \ .
\end{equation}

Alicki \cite{A78} exploited such a construction to show that for $\sigma\in \Dens_+$ 
with non-degenerate spectrum (in the sense that each eigenvalue of $\sigma$ is simple), the generator 
$\cL$ of a quantum Markov semigroup on $\cM_n(\C)$ that satisfies the 
$\sigma$-DBC (so that $\cL$ commutes with $\Delta_\sigma)$ has the form described in (\ref{genform}) below; see 
\cite[Theorem 3]{A78}. 
(Alicki actually considered the more general case that $\cL$ is {\em normal} with respect to the 
$\sigma$-GNS inner product. 
His formula reduces to (\ref{genform}) when $\cL$ is self-adjoint.)
The hypothesis that $\sigma$ has non-degenerate spectrum turns out to unnecessary. 
In the context of full matrix algebras this has been shown in 
\cite{KFGV} for the alternate canonical form given in \cite[Theorem 3]{A78}.  
It is possible to give a simple and self-contained proof of Alicki's Theorem in a somewhat more general setting. This is done in Appendix A; the theorem proved there is:

\begin{thm}\label{strucB} Let $\cP_t = e^{t \cL}$ be a QMS on a  unital $C^*$-subalgebra 
$\aA$ of $\cM_n(\C)$. Suppose that $\cP_t$ satisfies the $\sigma$-DBC for 
$\sigma\in \Dens_+(\aA)$ and that $\cP_t$ has an extension $\widehat \cP_t$ to a QMS on $\cM_n(\C)$.
Regard the modular operator $\Delta_\sigma$ as an operator on $\cM_n(\C)$ and
let $\tau$ denote the normalized trace on $\cM_n(\C)$.
Then the generator $\cL$ of $\cP_t$ has the form
 \begin{align}
	\cL A &=  \sum_{j \in \cJ}  \Big(e^{-\omega_j/2}  V_j^* [A, V_j]  
	   + e^{\omega_j/2}   [V_j , A] V_j^*\Big) \label{genform}\\
	   & = 
	 \sum_{j \in \cJ}  e^{-\omega_j/2} 
	 \Big( V_j^* [A, V_j] +  [V_j^* ,A] V_j \Big)  \label{genform2}
\end{align}
where $\omega_j \in \R$ for all $j \in \cJ$, and $\{V_j\}_{j\in \cJ}$ is a set in 
$\cM_n(\C)$ with the properties:

 \smallskip
 \noindent{\it(i)} $\tau[V_j^*V_k]= \delta_{j,k}$ for all $j,k\in \cJ$.
 
 \smallskip
 \noindent{\it(ii)} $\tau[V_j] =0$ for all $j\in \cJ$.
 
 \smallskip
 \noindent{\it(iii)} $\{V_j\}_{j\in \cJ} = \{V_j^*\}_{j\in \cJ}$.
 
 \smallskip
 \noindent{\it(iv)} $\{V_j\}_{j\in \cJ}$ consists of eigenvectors of the modular operator $\Delta_\sigma$ with
\begin{equation}\label{Vcom}
\Delta_\sigma V_j  = e^{-\omega_j}V_j \ .
\end{equation}

Conversely, given any $\sigma \in \Dens_{+}(\aA)$, and any set $\{V_j\}_{j\in \cJ}$ satisfying {\it (iii)} and {\it (iv)} for some $\{\omega_j\}_{j\in \cJ} \subseteq \R$, the operator $\cL$ given by (\ref{genform}) is the generator of a QMS $\cP_t$ that satisfies the $\sigma$-DBC. 
\end{thm}

\begin{remark}\label{rmk:equiv-expressions}
The eigenvectors of $\Delta_\sigma$ with eigenvalues other than $1$ cannot be self-adjoint on account of (\ref{conjeig}). However, when $\sigma$ is the normalized trace, $\Delta_\sigma$ is the identity, so that each $V_j$ is an eigenvector of $\Delta_\sigma$ with eigenvalue $1$, thus $\omega_j = 0$. 
It is then possible to take each $V_j$ to be self-adjoint, so that (\ref{genform2}) reduces to
\begin{equation}\label{genform3}
	\cL A = - \sum_{j \in \cJ}  [V_j, [V_j ,A] ]  \ .
\end{equation}
This formulation arises naturally in various applications, as we shall see in Section \ref{sec:examples}.
\end{remark}

\begin{remark}\label{rmk:dual}
By Theorem \ref{strucB}, the Hilbert-Schmidt adjoint of $\cL$ is given by 
  \begin{eqnarray}
 \cL^\dagger \rho
   &=&
    \sum_{j \in \cJ} \Big( e^{-\omega_j/2} [V_j \rho,  V_j^*] 
	   + e^{\omega_j/2} [V_j^*, \rho V_j] \Big) 
	   \label{eq:adj-2}
   \\&=&
  \sum_{j \in \cJ} e^{-\omega_j/2} \Big( [V_j \rho,  V_j^*] + [V_j, \rho V_j^*] \Big) \ . \label{eq:adj-1}\nonumber
  \end{eqnarray}
\end{remark}

\begin{remark} Note that the operators in $\{V_j\}_{j\in \cJ}$ need not belong to $\aA$ itself.  The Fermi Ornstein-Uhlenbeck semigroup  in the Clifford algebra with an odd number of generators provides an example in which they 
do not, as we discuss shortly.  

\end{remark}

\begin{remark}\label{cannonical} By properties {\it (i)} and  {\it (ii)}, 
the index set $\cJ$ has cardinality 
$|\cJ|$  no greater than $n^2-1$.   While in general the set $\{V_j\}_{j\in \cJ}$ is not uniquely determined, the proof in the appendix shows that $|\cJ|$ is uniquely determined. 
Moreover, if $\{V_j\}_{j\in \cJ}$ and $\{\tilde V_j\}_{j\in \cJ}$ are two such sets, there is an $|\cJ|\times |\cJ|$ unitary matrix $U_{j,k}$ such that for all $j\in \cJ$, ${\displaystyle \tilde V_j = \sum_{k\in \cJ} U_{j,k} V_k}$,
and such that unless $\tilde \omega_k = \omega_j$, $U_{k,j} = 0$. 
Thus, in a strong sense, the sets $\{V_j\}_{j\in \cJ}$ and   $\{\omega_j\}_{j\in \cJ}$
 are canonically associated to $\cL$. (Indeed, since $\Delta_\sigma V_j = e^{-\omega_j}V_j$, the numbers 
 $\{\omega_j\}_{j\in \cJ}$ are fixed once $\sigma$ and the set $\{V_j\}_{j\in \cJ}$ is fixed.)
\end{remark}

\section{Restriction to commutative subalgebras}

Let $\aA$ be a unital $C^*$-algebra, and let $\sigma\in \Dens_+(\aA)$. Let $A_1, A_2$ be eigenvectors of $\Delta_\sigma$: $\Delta_\sigma (A_j) =\lambda_j A_j$, $j=1,2$. 
Then $\lambda_1,\lambda_2>0$, and since the modular operator is an automorphism,
\begin{equation*}
\Delta_\sigma(A_1A_2) = \Delta_\sigma(A_1) \Delta_\sigma(A_2)  = \lambda_1\lambda_2A_1A_2\ .
\end{equation*}
That is, the product of eigenvectors of $\Delta_\sigma$ is again an eigenvector of $\Delta_\sigma$, and moreover, the eigenspace of $\Delta_\sigma$ for the eigenvalue $1$ is an algebra. 
In fact, by (\ref{conjeig}) it is a $*$-algebra, and it consists exactly of those elements  $A\in \aA$ that commute with $\sigma$. Clearly, $\sigma$ itself always belongs to $\aA_\sigma$.

\begin{defi}[Modular subalgebra] Let $\aA$ be a unital $C^*$-algebra, and let $\sigma\in \Dens_+(\aA)$.
The {\em $\sigma$-modular subalgebra of $\aA$}, denoted $\aA_\sigma$, is the $C^*$-subalgebra of $\aA$ consisting of the eigenspace of $\Delta_\sigma$ with eigenvalue $1$. 
\end{defi}
 
Of course if $\sigma = \one$, $\Delta_\sigma$ is the identity on $\aA$, and then $\aA_\sigma = \aA$. 
On the other hand, suppose that  $\aA = \cM_n(\C)$ and let $\{\eta_1,\dots,\eta_n\}$ be an orthonormal basis of $\C^n$ consisting of eigenvectors of $\sigma$ with $\sigma \eta_j = e^{-\lambda_j}\eta_j$ for $j=1,\dots,n$. 
Then for each $1\leq i,j\leq n$, $|\eta_i\bk \eta_j|$ is an eigenvector of $\Delta_\sigma$ with eigenvalue $e^{\lambda_j -\lambda_i}$. 
If the numbers $\{\lambda_1,\dots,\lambda_m\}$, which are the eigenvalues of the modular generator $h$, are all distinct, then the eigenspace of $\Delta_\sigma$ for the eigenvalue $1$ is exactly the span of the set $\{ |\eta_j\bk \eta_j| \}_{j=1}^n$ \cite{A78}. 
In this case $\aA_\sigma$ is an $n$-dimensional commutative $C^*$-subalgebra of $\aA$.

Let $\cP_t=e^{t\cL}$ be an ergodic QMS on $\aA$, and let $\sigma \in \Dens_+$ be its unique invariant state.
Suppose that $\cP_t$ satisfies the $\sigma$-DBC. 
Since $\cL$ and $\Delta_\sigma$ commute by Theorem \ref{detequiv}, the $\sigma$-modular subalgebra $\aA_\sigma$ of $\aA$ is invariant under $\cP_t$ (and under $\cP_t^\dagger$ as well).  
In this case, let $\widecheck \cP_t$ denote the restriction of $\cP_t$ to $\aA_\sigma$. 
 
If $\aA_\sigma$ is commutative, then $\widecheck \cP_t$ is an ergodic QMS on a commutative subalgebra of $\cM_n(\C)$, which may be identified with the transition semigroup of a classical Markov chain. 
This happens whenever each eigenvalue of $\sigma$ is simple \cite{A78}. 

Therefore, consider a QMS $\cP_t = e^{t\cL}$ on $\cM_n(\C)$ that satisfies the $\sigma$-DBC. 
Suppose that $\aA$ is a unital commutative $C^*$-subalgebra of $\cM_n(\C)$ that is invariant under $\cP_t^\dagger$. 
If $\cP_t$ is ergodic,  it follows using the self-adjointness of $\cL$ with respect to $\langle \cdot, \cdot\rangle_1$, that for any $\rho\in \Dens_+(\aA)$, $\sigma = \lim_{t\to\infty}\cP_t^\dagger \rho$, and hence $\sigma\in  \aA$. 

In our finite-dimensional setting, there exists a finite set $\{E_1,\dots,E_m\}$ of {\em minimal} projections in $\aA$ such that $\sum_{k=1}^m E_k = \one$. 
Consequently, $E_jE_k= 0$ for all $j\neq k$, and $\aA$ is the span of $\{E_1,\dots,E_m\}$: For all $A\in \aA$ we have 
\begin{equation}\label{modexp1}
A = \sum_{k=1}^m \frac{a_k}{\tr[E_k]} E_k \quad{\rm where}\quad a_k = \tr[E_k A]\ .
\end{equation}
Define an $m\times m$ matrix $Q$ by
\begin{equation}\label{modexp2}
Q_{k,\ell} = \frac{1}{\tr[E_k]}\tr[ E_k \cL E_\ell]\ .
\end{equation}
A vector $\vec \rho = (\rho_1,\dots,\rho_m) \in \R^m$ is a {\em probability vector} in case $\rho_k \geq 0$ for all $k$, and $\sum_{k=1}^m \rho_k = 1$. 

\begin{thm}\label{restrict}  
Let $\cP_t=e^{t\cL}$ be an ergodic QMS on $\cM_n(\C)$ that satisfies the $\sigma$-DBC for its invariant state $\sigma$.  
Let $\aA$ be a unital commutative $C^*$-subalgebra of $\cM_n(\C)$ that is invariant under $\cP_t^\dagger$. 
Let $\{E_1,\dots,E_m\}$ be a set of minimal projections in $\aA$ such that $\sum_{k=1}^m E_k = \one$.   
The $m\times m$ matrix $Q$ defined by \eqref{modexp2} specifies an ergodic continuous-time Markov chain on $\{1,\dots,m\}$ with jump rates $Q_{k,\ell}$ from $k$ to $\ell$. The corresponding forward equation, governing the evolution of site occupation probabilities is
\begin{equation}\label{modexp2B0}
\frac{{\rm d}}{{\rm d}t}\rho_\ell(t) =
 \sum_{k=1}^m \big(\rho_k(t) Q_{k,\ell} - \rho_\ell(t)Q_{\ell,k} \big) \ .
\end{equation}
A time-dependent probability vector $\vec \rho(t)$ satisfies (\ref{modexp2B0}) if and only if the time-dependent state $\rho(t)$ on $\aA$ given by
\begin{equation}\label{modexp2B1}
	\rho(t) = \sum_{k=1}^m \frac{\rho_k(t)}{\tr[E_k]}E_k
\end{equation}
satisfies $\displaystyle{\ddt \rho(t) = \cL^\dagger \rho(t)}$.  Moreover, the probability vector $\vec \sigma$ given by $\sigma_k = \tr[\sigma E_k]$ for $k=1,\dots,m$ is the unique invariant probability vector for the Markov chain, and the classical detailed balance condition
\begin{equation}\label{modexp3}
\sigma_k Q_{k,\ell} = \sigma_\ell Q_{\ell,k} 
\end{equation}
is satisfied.
\end{thm}
 
\begin{proof} We first show that the matrix $Q$ satisfies $\sum_{\ell =1}^m Q_{k,\ell} = 0$ and that $Q_{k,\ell} \geq 0$ for all $k\neq \ell$, which makes it a transition rate matrix.  

Since $0 = \cL\one = \sum_{\ell=1}^m \cL E_\ell$, we have $\sum_{\ell=1}^m Q_{k,\ell} = 0$. 
Let $\cL$ be given in the form (\ref{genform2}). 
Then for $k\neq \ell$, simple computations yield
$$\tr[ E_k \cL E_\ell] = 
2\sum_{j \in \cJ}  e^{-\omega_j/2} \tr[E_kV_j^* E_\ell V_j]  .$$
Since $E_k$ and $V_j^* E_\ell V_j$  are positive, $\tr[E_kV_j^* E_\ell V_j] \geq 0$, showing that $Q_{k,\ell} \geq 0$ for all $k \neq \ell$.

Now suppose that $a_k := \sum_{\ell \neq k} Q_{k,\ell} = 0$.
Then from the definition, $\tr[\cL^\dagger(E_k) E_\ell] = \tr[E_k\cL(E_\ell)] = 0$ for all $k\neq \ell$, and then, since $\sum_{\ell=1}^m Q_{k,\ell} = 0$, also $\tr[\cL^\dagger(E_k) E_k] = 0$. 
It would follow that $\cL^\dagger(E_k) = 0$. 
Since $\cP_t$ is ergodic, this is impossible unless $\aA$ is spanned by $\one$, a trivial case. 
Hence we may proceed assuming that $a_k > 0$ for all $k$. 

Define the matrix $P_{k,\ell}$ by 
\begin{align*}
{ P_{k,\ell} = \begin{cases} {\displaystyle \frac{1}{a_k} Q_{k,\ell} }  & \ell \neq k\\ 0 &\ell = k\end{cases}} \ .
\end{align*}
Evidently $P$ is an $m\times m$ stochastic matrix. Define the $m\times m$ matrix $M$ by
$M_{k,\ell} = a_k \delta_{k,\ell}$. 
Then $Q = M( P - \one_m)$ where $\one_m$ is the $m\times m$ identity matrix. The equation 
${\displaystyle  \frac{{\rm d}}{{\rm d}t}\vec \rho(t) = Q^\dagger \vec \rho(t)}$ 
is solved in terms of the initial data
$\vec \rho_0$ by $\vec \rho(t) = e^{t Q^\dagger}\vec \rho_0$, 
and it gives the site occupation probabilities for a continuous time Markov chain in which the jump rate for leaving site $k$ is $a_k$, and when such a jump occurs, the probability that site $\ell$ is the new site occupied is given by $P_{k,\ell}$. 

The equation
${\displaystyle \frac{{\rm d}}{{\rm d}t}\vec \rho(t) = Q^\dagger \vec \rho(t)}$ can be written as
\begin{equation}\label{modexp2B}
\frac{{\rm d}}{{\rm d}t}\rho_\ell(t) =  \sum_{k=1}^m \rho_k(t) Q_{k,\ell}  =
 \sum_{k=1}^m \big(\rho_k(t) Q_{k,\ell} - \rho_\ell(t)Q_{\ell,k} \big) \ ,
\end{equation}
where the second equality follows from the first and the fact that $Q_{\ell,\ell}  = -\sum_{k\neq \ell} Q_{\ell,k}$.
This gives us (\ref{modexp2B0}).

To show that this Markov chain satisfies the classical detailed balance condition (\ref{modexp3}), 
observe that $\sigma E_k = \sigma_k (\tr[E_k])^{-1} E_k = E_k \sigma$ by \eqref{modexp1}, and therefore
\begin{align*}
\sigma_k Q_{k,\ell}
& = \frac{\sigma_k}{\tr[E_k]}\tr[E_k\cL(E_\ell)] 
  = \tr[\sigma E_k\cL(E_\ell)]
  = \tr[\sigma \cL(E_k) E_\ell]
\\&  =  \tr[ E_\ell\sigma \cL(E_k)]
  =  \frac{\sigma_\ell}{\tr[E_\ell]} \tr[ E_\ell \cL(E_k)]
  =  \sigma_\ell Q_{\ell,k} \ , 
  \end{align*}
where the third equality holds since $\cL$ is self-adjoint for the $\sigma$-GNS inner product. It follows immediately from \eqref{modexp2B} that $\vec \sigma$ is invariant.

Finally, it is easy to check using the definition of the matrix $Q$ in terms of the generator $\cL$ that $\vec \rho(t)$ satisfies (\ref{modexp2B}) if and only if 
$\ddt \rho(t) = \cL^\dagger \rho(t)$. 
Ergodicity of the Markov chain $\widecheck \cP_t$ now follows from the ergodicity of $\cP_t$. 
\end{proof}

\section{Dirichlet form representation associated to a quantum Markov generator}\label{diffstruc}

Let $\H = L^2(X,\cB,\mu)$ be the space of square integrable real-valued functions on some probability space $(X,\cB,\mu)$. 
A closed and densely defined, symmetric non-negative bilinear form $\E$ on $\H$ defines a non-negative unbounded operator $-A$ through $\E(f,g) = - \langle f, Ag\rangle_{\H}$.  A special case ($\mu$ is a probability measure) of a theorem  of 
Beurling and Deny \cite{BD58,BD59} states that $P_t = e^{tA}$ is a Markov semigroup if and only if for all $f\in \H$,
$\E(\hat f,\hat f) \leq \E(f,f)$ where $\hat f$ denotes the projection of $f$ onto the closed convex set $\{ g\in \H\ :\ 0 \le g \leq 1 \ {\rm a.e.}\}$.

A powerful non-commutative extension of this theory has been developed starting with the early work of Gross \cite{Gr72,Gr75}, and continuing with \cite{AHK,Cip97,Cip08,CS03,DL92,GL93}. We shall not need the whole theory at present, but the Dirichlet form representation of the generator of a QMS will be useful to us. 
 
Let $\cP_t = e^{t\cL}$ be a QMS on $\aA$ that satisfies the $\sigma$-DBC for some $\sigma\in \Dens_+(\aA)$.
The generator $\cL$ can then be written in the canonical form specified in (\ref{genform2}) of Theorem~\ref{strucB}.
Throughout the rest of this section we fix such a generator $\cL$, and the sets $\{V_j\}_{j\in \cJ}$ and  $\{\omega_j\}$  that specify $\cL$ in the form (\ref{genform2}).  

Define operators $\partial_{j}$ on $\aA$ by
\begin{equation*}
\partial_{j}A = [V_j,A] \quad{\rm so\ that}\quad \partial_j^\dagger  A =  [V_j^*,A] \ .
\end{equation*}

The operators $\partial _j$ are derivations, and we may consider them as non-commutative analogs of partial derivatives associated to $\cL$. With respect to the Hilbert-Schmidt inner product, we may then form non-commutative analogs of the gradient, divergence and Laplacian associated to $\cL$. We begin with the Laplacian:

Given the set $\{V_j\}_{j\in \cJ}$, we define an operator $\cL_0$ on $\fH_\aA$ by
\begin{equation*}
\cL_0 A 
 = -\sum_{j\in \cJ}  \partial_j^\dagger \partial_jA 
 = -\sum_{j\in \cJ}  [V_j^*,[V_j,A]] \ .
\end{equation*}
Evidently $\cL_0^\dagger = \cL_0$, and we may write
${\displaystyle 
	\cL_0 A   
	= \sum_{j\in \cJ}  (V_j^*[A,V_j] + [V_j,A]V_j^*)}$.
Thus, by Theorem \ref{strucB}, $\cL_0$ is the generator of a quantum Markov semigroup $\cP_{0,t} = e^{t\cL_0}$ satisfying detailed balance with respect to $\tau$.
We call this semigroup the {\em heat semigroup} associated to $\cP_t = e^{t\cL}$, and the operator $\cL_0$ the {\em Laplace operator associated to $\cL$}.

We define the Hilbert space $\fH_{\aA,\cJ}$ by
$$\fH_{\aA,\cJ} = \bigoplus_{j\in \cJ} \fH_\aA^{(j)} \ ,$$
where each $\fH_\aA^{(j)}$ is a copy of $\fH_\aA$. 
For ${\bf A}\in \fH_{\aA,\cJ}$ and $j \in \cJ$, let $A_j$ denote the component of ${\bf A}$ in $\fH_\aA^{(j)}$. Thus, picking some linear ordering of $\cJ$, we can write
$${\bf A} = ( A_1, \dots, A_{|\cJ|})\ .$$
We equip $\fH_{\aA,\cJ}$ with the usual inner product
${ 
\langle {\bf A}, {\bf B}\rangle_{\fH_{\aA,\cJ}} = \sum_{j\in \cJ}\langle A_j,B_j\rangle_{\fH_\aA}}$.

Define an operator
$\nabla : \fH_\aA \to \fH_{\aA,\cJ}$ by
$$\nabla A = ( \partial_1, \dots, \partial_{|\cJ|}A)\ .$$
Thinking of elements of $\aA$ as non-commutative analogs of functions on a manifold, we may think of
${\bf A} = ( A_1, \dots, A_{|\cJ|})$ as a vector field. This point of view will be justified in the next section.
We define the operator $\dive : \fH_{\aA,\cJ}\to \fH_\aA$ by
$$\dive {\bf A} 
  = - \sum_{j\in \cJ}  \partial_j^\dagger A_j
  = \sum_{j\in \cJ}   [A_j,V_j^*] \ .$$
Note that $\dive$ is minus the adjoint of the map $\nabla: \fH_\aA \to  \fH_{\aA,\cJ}$, so that $\cL_0$ is negative semi-definite.
With these definitions, $\cL_0 = \dive \circ \nabla$. We call $\nabla$ the {\em non-commutative gradient} associated to $\cL$, and $\dive$ the {\em non-commutative divergence} associated to $\cL$.

\begin{remark}\label{cohomology} 
In this finite-dimensional setting, by elementary linear algebra,
\begin{equation}\label{null}
\big({\rm Null}(\dive)\big)^\perp = {\rm Ran}(\nabla)\ .
\end{equation}
In the terminology introduced above, elements of ${\rm Null}(\dive)$ are {\em divergence free vector fields}. 
Then (\ref{null}) says that a vector field ${\bf A}$ is a gradient if and only if it is orthogonal in $\fH_{\aA,\cJ}$ to every divergence free vector field. 
\end{remark}

The differential structure introduced above allows us to write the generator $\cL$ of a QMS in terms of a non-commutative Dirichlet form.

\begin{lm} For all $s \in [0,1]$, all $j \in \cJ$, and  all $A,B\in \aA$ we have
\begin{equation}\label{minadjm}
 \langle \partial_{j}B,A\rangle_s = \langle B, e^{s\omega_j}(e^{-\omega_j}V_j^*A - AV^*_j) \rangle_s \ .
\end{equation}
\end{lm}
 
 \begin{proof}  For any $A,B\in \cM_n(\C)$,
 \begin{eqnarray*}
 \langle \partial_{j}B,A\rangle_s &=& 
 \tr[\sigma^s (\partial_{j}B)^* \sigma^{1-s} A]= \tr[\sigma^s (V_jB - BV_j)^* \sigma^{1-s} A]\\
 &=&   
 \tr[\sigma^s B^*V_j^* \sigma^{1-s} A]  - \tr[\sigma^s V_j^*B^* \sigma^{1-s} A]\\
 &=&  \tr[\sigma^s B^* \sigma^{1-s}\Delta_\sigma^{s-1}(V_j^*) A]  - \tr[\Delta_\sigma^s(V_j^*)\sigma^s B^* \sigma^{1-s} A]\\
 &=&  e^{(s-1)\omega_j}\tr[\sigma^s B^* \sigma^{1-s} V_j^* A]  - e^{s\omega_j}\tr[V_j^*\sigma^s B^* \sigma^{1-s} A]\\
 &=& \bip{ B, e^{s\omega_j}(e^{-\omega_j}V_j^*A - AV_j^*) }_s \ ,
 \end{eqnarray*}
 where in the fourth line we have used (\ref{Vcom}). 
 \end{proof}
 
It follows from \eqref{minadjm} that for all $s\in [0,1]$, and all $A,B\in \aA$, 
 \begin{equation*} 
  e^{(1/2-s)\omega_j} 
  	\ip{ \partial_{j} B, \partial_{j} A }_s
	= - \bip{ B, e^{-\omega_j/2} V_j^* [A, V_j] 
		+ e^{\omega_j/2} [V_j, A] V^*_j }_s \ .
  \end{equation*}
Using the expression for $\cL$ given in \eqref{genform}, we obtain 
\begin{equation*}
  \cE_s(B,A)
  	=  - \ip{ B, \cL A }_s \ , 
\quad  	\text { where } \quad
  \cE_s(B,A) :=  \sum_{j\in \cJ}   e^{(1/2-s)\omega_j} 
     \ip{ \partial_{j} B, \partial_{j} A }_s \ .
  \end{equation*}
In particular, taking $s=1/2$, we see that 
  \begin{equation}\label{KMSdir}
  \cE_{1/2}(B,A)
  	=  - \ip{ B, \cL A }_{1/2} \ , 
\quad  	\text { where } \quad
  \cE_{1/2}(B,A) :=  \sum_{j\in \cJ}  
     \ip{ \partial_{j} B, \partial_{j} A }_{1/2} \ ,
  \end{equation}
which expresses $\cL$ in terms of a Dirichlet form in the sense of \cite{GL93,Cip97,Cip08}. 

As a simple consequence of the Dirichlet form representation, we state an ergodicity result.  
Recall that a QMS $\cP_t = e^{t\cL}$ is {\em ergodic} in case for each $t>0$, the $1$-eigenspace of $\cP_t$ is spanned by the identity, or, what is the same thing, the $0$-eigenspace of $\cL$ is spanned by the identity.  We refer to \cite{FV82} for an early study of ergodicity for quantum dynamical semigroups.

\begin{thm}\label{erg} 
Let $\cP_t = e^{t\cL}$ be QMS on $\aA$ that satisfies the $\sigma$-DBC for $\sigma\in \Dens_+(\aA)$.  
Let $\cL$ be given in the form (\ref{genform2}). 
Then the commutant of $\{V_j\}_{j\in \cJ}$ equals the null space of $\cL$.
In particular, $\cP_t$ is ergodic if and only if the commutant of $\{V_j\}_{j\in \cJ}$ is spanned by the identity.   
\end{thm} 

\begin{proof} 
Suppose that $A$ belongs to the commutant of $\{V_j\}_{j\in \cJ}$. 
By definition, this means that $\partial_j A = 0$ for all $j \in \cJ$, and therefore $A \in {\rm Null}(\cL)$ by \eqref{genform}.

Conversely, if $\cL A = 0$, then by (\ref{KMSdir}),
$$
0   = - \ip{ A, \cL A }_{1/2} 
	= \sum_{j\in \cJ}  
		\ip{ \partial_{j} A, \partial_{j} A }_{1/2} \ , 
$$
which is the case if and only if $[V_j,A] = 0$ for all $j$. This means that $A$ belongs to the commutant of  $\{V_j\}_{j\in \cJ}$. 
\end{proof}

\begin{thm}\label{Poiss}  Let $\cP_t = e^{t\cL}$ be an ergodic QMS on $\aA$ that satisfies the $\sigma$-DBC for $\sigma\in \Dens_+(\aA)$, and let $\cL_0$ be the associated Laplacian.  
Then for given $B\in \fH_\aA$, the equation
\begin{equation*}
\cL_0 X = B
\end{equation*}
has a solution if and only if $\tau[B] =0$. 
Consequently, when $\tau[B] = 0$, there is a non-trivial affine subspace of $\fH_{\aA,\cJ}$ consisting of elements ${\bf A}$ for which $\dive{\bf A} = B$.
\end{thm}

\begin{proof} 
Since $\langle A, \cL_0 A\rangle_{\fH_\aA} = - \langle \nabla A, \nabla A\rangle_{\fH_{\aA,\cJ}}$, we have ${\rm Null}(\cL_0) = {\rm Null}(\nabla)$.
Since $\cP_t$ is ergodic, it follows from Theorem~\ref{erg} that ${\rm Null}(\nabla)$ is spanned by $\one$. Since $\cL_0$ is self-adjoint on $\fH_\aA$,  the assertion now follows from the Fredholm alternative. 
\end{proof}

 The following  identity will be useful going forward.

\begin{lm}[Chain rule identity]\label{comid} For all $V\in \cM_n(\C)$, $\rho\in \Dens_+$ and $\omega\in \R$,
\begin{equation}\label{comideq}
\int_{0}^1 e^{\omega(s-1/2)} R_\rho\Delta_\rho ^s\Big(V \log(e^{-\omega/2}\rho) -  \log(e^{\omega/2}\rho)V\Big) \dd s = 
e^{-\omega/2} V \rho - e^{\omega/2} \rho V\ .
\end{equation}
\end{lm}

\begin{proof} 
Define $f(s) = e^{\omega(1/2-s)} \rho^{1-s} V \rho^s$. The right side of \eqref{comideq} equals $f(1) - f(0)$ and 
\begin{align*}
	f'(s)& = e^{\omega(1/2-s)} \rho^{1-s} 
			\Big( -\omega V - \log(\rho) V + V \log(\rho) \Big) \rho^s 
		 \\& = e^{\omega(1/2-s)} \rho^{1-s} 
			\Big(  V \log(e^{-\omega/2}\rho) -  \log(e^{\omega/2}\rho) V  \Big) \rho^s \ .
\end{align*}
Thus the left side of \eqref{comideq} equals $\int_0^1 f'(1-s) \dd s$, which yields the result.
\end{proof}

\begin{remark}\label{comidrem} Consider the function $f_\omega$ defined by
\begin{equation}\label{fomdef}
 f_\omega(t)
   := \int_0^1 e^{\omega(s-1/2)} t^s \dd s 
    = e^{\omega/2}\frac{t - e^{-\omega}}{\log t + \omega} \ .
\end{equation} 
Then (\ref{comideq}) can be formulated as
\begin{equation}\label{comideq2}
R_\rho f_\omega(\Delta_\rho)\Big(V \log(e^{-\omega/2}\rho) -  \log(e^{\omega/2}\rho)V\Big)  
  = e^{-\omega/2}V\rho - e^{\omega/2}\rho V\ .
\end{equation}

Notice that for $\omega = 0$,  \eqref{comideq} reduces to the commutator identity
\begin{equation}\label{comideq0}
R_\rho f_0(\Delta_\rho)([V, \log \rho]) = [V,\rho]\ .
\end{equation}
This identity provides a quantum analog of the classical identity for smooth, strictly positive probability densities $\rho(x)$ on $\R^n$:
\begin{equation}\label{classical-chain}
\rho(x) \nabla \log \rho(x) = \nabla \rho(x)\ .
\end{equation}
To see this, note that  if $A$ commutes with $\rho$, $f_\omega(\Delta_\rho)A =A$
so that for each $\omega$, the operation ${\displaystyle A \mapsto R_\rho f_\omega(\Delta_\rho) A}$ is one of the 
non-commutative interpretations of multiplication of $A$ by $\rho$. 
Now applying (\ref{comideq0}) with $V = V_j$, we have 
$R_\rho f_0(\Delta_\rho)(\partial_j  \log \rho) = \partial_j\rho$ 
for each $j$, which yields a quantum analog of (\ref{classical-chain}).
\end{remark}

Lemma~\ref{comid} and the previous remark motivate the following definition:

\begin{defi}\label{multrho}  For $\rho\in \Dens_+$, and $\omega\in \R$, define the operator $[\rho]_\omega : \cM_n(\C) \to \cM_n(\C)$ by 
\begin{equation}\label{multPdef}
[\rho]_\omega  = R_\rho \circ f_\omega(\Delta_\rho)
\end{equation}
For each $\omega$,  $[\rho]_\omega$, which is one of the non-commutative forms of  {\em multiplication by $\rho$}, is evidently invertible, and its inverse, $[\rho]_\omega^{-1} = (1/f_\omega)(\Delta_\rho)\circ R_{\rho^{-1}}$ may then be viewed as the corresponding non-commutative form of {\em division by $\rho$}. 
\end{defi}
We remark that $[\rho]_\omega^{-1}$ is a kernel operator that can be used to define a monotone metric on density matrices in the sense of \cite{Pe96}. However, the Riemannian metric on density matrices that we introduce in this paper will be different.

\begin{lm}\label{smooth}  For all $\omega\in \R$, the maps $\rho \mapsto [\rho]_\omega$ and   $\rho \mapsto [\rho]_\omega^{-1}$ are $C^\infty$ on $\Dens_+$. Furthermore, for all $A$,
\begin{equation}\label{starrho}
([\rho]_\omega A)^* =  [\rho]_{-\omega}A^*  \quad{\rm and\ consequently}\quad   ([\rho]^{-1}_\omega A)^* =  [\rho]^{-1}_{-\omega}A^*\ .
\end{equation}
\end{lm}

\begin{proof} Recall the identities
\begin{equation}\label{intid}
\int_0^1 \lambda^{1-s}\mu^s \dd s 
 = \frac{\lambda - \mu}{\log \lambda - \log \mu} \tand
\int_0^\infty \frac{1}{(t+\lambda)(t+\mu)} \dd t 
 = \frac{\log \lambda - \log \mu}{\lambda - \mu} \ ,
\end{equation}
which hold for $\lambda, \mu > 0$.
By (\ref{fomdef}) and (\ref{multPdef}) we obtain
\begin{align}\label{eq:rho-box}
 [\rho]_\omega
  = R_\rho f_\omega(\Delta_\rho) 
  = \int_0^1 e^{\omega(1/2- s)} L_\rho^s R_\rho^{1-s} \dd s 
  =  \int_0^1  (e^{-\omega/2}L_\rho)^s (e^{\omega/2} R_\rho)^{1-s}  \dd s \ ,
\end{align}
and it follows from (\ref{intid}) that
\begin{align*}
	[\rho]_\omega^{-1}
	  = \int_0^\infty
	    (t + e^{-\omega/2}L_\rho)^{-1}  
(t+ e^{\omega/2} R_\rho)^{-1}  \dd t \ .
\end{align*}
The fact that $\rho\mapsto [\rho]_\omega^{-1}$ is $C^\infty$ now follows immediately from the resolvent identity, and then the $C^{\infty}$-differentiability of $\rho\mapsto [\rho]_\omega$ is clear. Moreover, (\ref{starrho}) follows from \eqref{eq:rho-box}.
\end{proof}

It is now a simple matter to write the quantum heat flow equation $\partial \rho/\partial t = \cL_0^\dagger\rho$ 
in a form that leads directly to its interpretation as gradient flow for the relative entropy with respect to the normalized trace: 
For all $\rho\in \Dens_+$,
\begin{equation*}
\cL_0^\dagger \rho = \dive ([\rho]_0 \nabla \log \rho) \ ,
\end{equation*}
where $[\rho]_0$ is applied to each component of $ \nabla \log \rho$.
As we show in the next section, it follows easily from this formula that there is a Riemannian metric on $\Dens_+$, which is a natural analog of the $2$-Wasserstein metric, such that the quantum heat flow on $\Dens_+$ associated to $\{V_j\}_{j\in \cJ}$ is a gradient flow for the relative entropy of $\rho$ with respect to the normalized trace $\tau$. For a comprehensive treatment of the theory of gradient flows with respect to the $2$-Wasserstein metric, see \cite{AGS} and \cite[Chapters 23-25]{V}.

In fact, the next lemma provides the means to extend this result to the general class of quantum Markov semigroups that satisfy detailed balance with respect to some non-degenerate state $\sigma$. 

\begin{lm} \label{lem:chain-lemma} Let $\cP_t = e^{t\cL}$ be QMS on $\aA$ that satisfies the 
$\sigma$-DBC for $\sigma\in \Dens_+(\aA)$. and let $\cL$ be given in the form (\ref{genform2}).
Then for all $\rho\in \Dens_+$, and all $j\in \cJ$,
\begin{equation*}
\partial_j (\log \rho - \log \sigma)  =  V_j \log(e^{-\omega_j/2}\rho) -  \log(e^{\omega_j/2}\rho)V_j\ .
\end{equation*}
\end{lm}

\begin{proof}  
By (\ref{Vcom}) we have $\Delta_\sigma^s V_j  = e^{-s\omega_j}V_j$, and thus $[V_j,\log \sigma] = - \partial_s|_{s=0}  \Delta_\sigma^s V_j = \omega_j V_j$.
It follows that
\begin{align*}
\partial_j (\log \rho - \log \sigma)  
   = [V_j,\log \rho] - \omega_j V_j 
   =  V_j \log(e^{-\omega_j/2}\rho) 
   	   -  \log(e^{\omega_j/2}\rho) V_j \ ,
\end{align*}
which is the desired identity.
\end{proof}

\begin{thm}\label{gradflowent}  Let $\cP_t = e^{t\cL}$ be QMS on $\aA$ that satisfies the $\sigma$-DBC for $\sigma\in \Dens_+(\aA)$, and let $\cL$ be given in the form (\ref{genform2}). Then, for all $\rho\in \Dens_+$, 
\begin{equation*}
- \cL^\dagger \rho = \sum_{j\in \cJ}  \partial_j^\dagger \Big( [\rho]_{\omega_j} \partial_j(\log \rho - \log \sigma)\Big)\ .
\end{equation*}
\end{thm}

\begin{proof}
Using Lemma \ref{lem:chain-lemma} and \eqref{comideq2} we obtain
\begin{align*}
	\sum_{j\in \cJ}  \partial_j^\dagger \Big( [\rho]_{\omega_j} \partial_j(\log \rho - \log \sigma)\Big)
	& = 	\sum_{j\in \cJ}  \partial_j^\dagger \Big( [\rho]_{\omega_j} \Big( 
	V_j \log(e^{-\omega_j/2}\rho) -  \log(e^{\omega_j/2}\rho)V_j
	\Big)\Big) 
\\	& = 	\sum_{j\in \cJ}  \partial_j^\dagger \Big(  e^{-\omega_j/2}V_j\rho - e^{\omega_j/2}\rho V_j\Big)
\\& = -  \sum_{j \in \cJ}  \Big( e^{-\omega_j/2} [V_j \rho,  V_j^*] 
	   + e^{\omega_j/2} [V_j^*, \rho V_j] \Big) = - \cL^\dagger \rho	\ ,
\end{align*}
where the final identity follows from \eqref{eq:adj-2}.
\end{proof}

\section{Examples}\label{sec:examples}

\subsection{The infinite-temperature Fermi Ornstein-Uhlenbeck semigroup}
\label{sec:inf-temp-Fermi}
As Segal emphasized \cite{S56}, the Fermion number operator for $n$ degrees of freedom can be represented in terms of the generators of a Clifford algebra. This permits the semigroup it generates to be realized as a QMS, a fact that was effectively exploited by Gross \cite{Gr72,Gr75} using Segal's non-commutative integration theory \cite{S53,S65}.

  Let $\{Q_1, \ldots, Q_n\}$ be self-adjoint operators on a finite-dimensional Hilbert space satisfying the \emph{canonical anti-commutation relations} (CAR):
\begin{align*}
 Q_j Q_k + Q_k Q_j = 2 \delta_{jk} \one \;.
\end{align*}
The \emph{Clifford algebra} $\Cln$ is the $2^n$-dimensional algebra generated by $\{Q_j\}_{j=1}^n$. 
Let $\Gamma : \Cln \to \Cln$ be the principle automorphism on $\Cln$, i.e., the unique algebra 
isomorphism satisfying $\Gamma(Q_j) = - Q_j$ for all $j$. The product of all of the generators 
$Q_1Q_2\cdots Q_n$ (in some order) is evidently unitary,
and the CAR imply that it commutes with each $Q_j$ if $n$ is odd, and anti-commutes with each $Q_j$ if $n$ is even. 
Hence when $n$ is odd, the center of $\Cln$ is non-trivial and $\Cln$ is not a factor. 

In the even case $n= 2m$,  form the $m$ self-adjoint unitary operators $i Q_{2j-1}Q_{2j}$, $j=1,\dots n$.
These all commute with one another, and we define
\begin{equation*}
W = i^mQ_1Q_2\cdots Q_{2m}\ .
\end{equation*}
Evidently $W\in \Cln$ is unitary and self-adjoint, and since $W$ anti-commutes with each $Q_j$, 
 the  principle automorphism is inner and is given by
\begin{equation}\label{innerW3}
\Gamma(A)  = WAW = W^*AW = WAW^*  \qquad{\rm for\ all \ } A\in \Cln\ .
\end{equation}

Let $\{0,1\}^n$ be the set of {\em fermion multi-indices}, and for all $\aa = (\alpha_j)_j \in \{0,1\}^n$,
define $Q^\aa = Q_1^{\alpha_1} \cdots Q_n^{\alpha_n}$ and $|\aa| = \sum_{j=1}^n \alpha_j$. 
Let $\tau$ be the canonical trace on $\Cln$, 
determined by $\tau(Q^\aa) := \delta_{0,|\aa|}$.
In a standard representation of $\Cln$ as an algebra of operators on $(\C^2)^{\otimes n}$ due to Brauer and Weyl, $\tau$ is simply the normalized trace. See \cite{CL93,CM13} for references and further background.

Gross \cite{Gr75} defined a differential structure and a Dirichlet form on $\Cln$ as follows: 
For $j\in \{1,\dots,n\}$, let $\sder_j$ be given by
\begin{equation}\label{grossder}
\sder_j (A) =  \frac12(  Q_j A - \Gamma(A) Q_j )\ .
\end{equation}
Each $\sder_j$ is a {\em skew derivation}. That is, for all $A,B\in \Cln$,
$\sder_j(AB) = (\sder_j A)B + \Gamma(A) \sder_j B$.
Gross defined the Dirichlet form $\mathcal{E}(A,B)$ on $\Cln$ by
$\mathcal{E}(A,B)  = \sum_{j=1}^n \tau[ (\sder_j A)^* \sder_j B]$, and defined the operator $\cL$ on $\Cln$ by 
$-\tau[A^* \cL B] = \mathcal{E}(A,B)$. 
Simple computations show that for all $A\in \Cln$, 
\begin{equation}\label{gross}
\sder^\dagger A =  \frac12\big(  Q_j A + \Gamma(A) Q_j \big) \tand  \cL A = \frac12 \sum_{j=1}^n \big(Q_j \Gamma(A) Q_j - A\big)\;.
\end{equation}
It is evident that $\cL Q^\aa= -|\aa|Q^\aa$, hence $-\cL$ is the {\em fermion number operator}.

When $n$ is even, so that $\Gamma(A) = W^*AW$,  we define $V_j = iWQ_j$, so that each 
$V_j$ is both self-adjoint and unitary. Then, using the fact that $V_j^2 = \one$, we may rewrite
 (\ref{gross}) as
\begin{equation*}
 \cL A = \frac12 \sum_{j=1}^n \big(V_j A V_j - A\big)
 	   = - \frac14 \sum_{j=1}^n [V_j,[V_j,A]] \ ,
\end{equation*}
which has the form (\ref{genform3}).
Thus, $e^{t\cL}$ is a QMS satisfying the $\tau$-DBC.  Gross discussed  this  QMS as a fermionic analog of the classical Ornstein-Uhlenbeck semigroup; we refer to it as the {\em infinite temperature  Fermi Ornstein-Uhlenbeck semigroup}.  
One good reason is that it  is generated by the negative of the fermionic number operator. 
Another is that as conjectured by Gross \cite{Gr75} and proved in \cite{CL93}, it has the same optimal hypercontractivity properties as the classical Ornstein-Uhlenbeck semigroup. We shall further develop the analogy here. 
(The infinite temperature part of the name will be justified in the next example.)

To relate the differential structure in (\ref{grossder}) to the one considered here, note that when $n$ is even, so that $\Gamma$ is the inner automorphism given by (\ref{innerW3}), we have, with $V_j$ defined as above, 
\begin{equation}\label{sder1}\sder_jA  = \frac{1}{2i} W[V_j,A]  =  \frac{1}{2i} W \partial_jA  \tand 
\sder_j^\dagger A =   -\frac{1}{2i} \partial_j^\dagger (WA) = -\frac{1}{2i} \partial_j(WA)
\end{equation}
for $j=1,\dots,n$.
Thus, Gross's differential structure in terms of skew derivations may be substituted with 
the present differential structure in terms of derivations when $n$ is even. When $n$ is odd, this is 
achieved by embedding $\Cln$ is an a larger Clifford algebra: One can add one more generator, or, perhaps better, embed the Clifford algebra generated by $\{Q_1,\dots,Q_n\}$, in the 
{\em phase space} Clifford algebra with $2n$ generators
$\{Q_1,\dots,Q_n,P_1,\dots,P_n\}$ that is discussed next.

\subsection{The finite-temperature Fermi Ornstein-Uhlenbeck semigroup}

Let $\aA$ be the Clifford algebra  $\Cln$ of dimension $n=2m$ for some $m\in \N$. Consider a set of generators 
$$\{Q_1,\dots,Q_m,P_1,\dots,P_m\}\ ,$$
where
$$Q_j Q_k + Q_k Q_j =  P_j P_k + P_k P_j = 2\delta_{j,k}\one 
\tand Q_j P_k + P_k Q_j = 0 \quad{\rm for\ all}\ 1 \leq j,k \leq m\ .$$
We think of $\aA$ as the full set of {\em phase space observables}, the subalgebra generated by $\{Q_1,\dots,Q_m\}$ as the algebra of {\em configuration space observables}, and the subalgebra generated by $\{P_1,\dots,P_m\}$ as the algebra of {\em momentum space observables}.

Form the operators
\begin{equation*}
Z_j   = \frac{1}{\sqrt{2}}(Q_j + i P_j)\quad  {\rm so\ that}\quad 
Z_j^* = \frac{1}{\sqrt{2}}(Q_j - i P_j)\ .
\end{equation*}
It is easy to check that 
\begin{equation}\label{createanihilate2}
Z_j Z_k + Z_k Z_j =  0 \quad \text{and} \quad
Z_j Z_k^* + Z_k^* Z_j = 2\delta_{j,k}\one \quad{\rm for \ all}\ 1\leq j,k \leq m\ \ .
\end{equation}
Consider the complementary orthogonal projections $N_j$ and $N_j^\perp$ defined by
\begin{equation*}
N_j = \frac12 Z_j^*Z_j \tand 
N_j^\perp = \frac12  Z_jZ^*_j 
\quad{\rm for\ all}\ \ 1\leq j \leq m\ \ .
\end{equation*}
Then (\ref{createanihilate2}) implies that
\begin{equation}\label{createanihilate4}
 Z_j N_j = N_j^\perp Z_j = Z_j  \tand
 N_j Z_j = Z_j N_j^\perp = 0 \ .
\end{equation}
Note also that
\begin{align}\label{createanihilate4-also}
	 Z_j N_k = N_k Z_j \tand
	 Z_j N_k^\perp = N_k^\perp Z_j \quad{\rm for\ all}\ \ j \neq k\ .
\end{align}
Moreover, $\{N_1,\dots,N_m,N_1^\perp ,\dots,N_m^\perp\}$ is a set of commuting orthogonal projections. 

For each $j$, $Q_jP_j$ commutes with both $Q_k$ and $P_k$ for all $k\neq j$. Hence the operators
$\{Q_1P_1,\dots,Q_mP_m\}$ all commute with one another.  As in the previous example, let $W = i^m \prod_{j=1}^mQ_jP_j$ so that $W$ is self-adjoint and unitary, and for all $A\in \aA$, let $\Gamma(A) = WAW$.
 Note that $Q_jP_jZ_j =  iZ_j$ for each $j$.

For any set of $m$ real numbers $\{e_1,\dots,e_m\}$, and any parameter 
$\beta \in (0,\infty)$, to be interpreted as the {\em inverse temperature}, define the {\em free Hamiltonian} $h$ and the {\em Gibbs state} $\sigma_\beta$ by

 \begin{equation*}
h = \sum_{j=1}^m e_j N_j \tand \sigma_{\beta} = \frac{1}{\tau[e^{-\beta h}]} e^{-\beta h}\ .
\end{equation*}
where $\tau$ is the canonical trace as in Section \ref{sec:inf-temp-Fermi}.

Since the $N_j$ are commuting orthogonal projections, $e^{-\beta h}$ is the product, in any order, of the operators $e^{-\beta e_j}N_j + N_j^\perp$.
Therefore, for each $1\leq j \leq m$,
\begin{align*}
\Delta_{\sigma_\beta}(Z_j) = (e^{-\beta e_j}N_j + N_j^\perp)Z_j (e^{\beta e_j}N_j + N_j^\perp) 
  = e^{\beta e_j } Z_{j} \ ,	
\end{align*}
where we have used \eqref{createanihilate4} and \eqref{createanihilate4-also}. Consequently, $\Delta_{\sigma_\beta}(Z_j^*) = e^{-\beta e_j } Z_j^{*}$.
Since $W$ commutes with every even element of $\aA$, it follows that
 \begin{equation*}
 \Delta_{\sigma_\beta}(W Z_j)  = e^{\beta e_j} W Z_j \tand 
 \Delta_{\sigma_\beta}(Z_j^*W) = e^{-\beta e_j }Z_j^*W \ .
 \end{equation*}

Define the operators
 \begin{equation*}
 V_j = WZ_j , \qquad 1 \leq j \leq m\ ,
 \end{equation*} 
so that $\frac12 V_j^* V_j = N_j$ and  $\frac12 V_j V_j^* = N_j^\perp$.
Then $\{V_1,\dots,V_m, V_1^*,\dots,V_m^*\}$ is set of operators on $\aA$ satisfying the conditions {\it (i)},  {\it (ii)}, {\it (iii)} and {\it (iv)} of Theorem~\ref{strucB}. 
Therefore, the operator $\cL_\beta$ defined by 
\begin{equation}\label{Lbeta}
\cL_\beta A =  \frac14 \sum_{j =1}^m \left[e^{\beta e_j /2} \Big( V_j^* [A, V_j] +  [V_j^* ,A] V_j \Big) 
	   + e^{-\beta e_j/2} \Big( V_j [A, V_j^*] +  [V_j , A] V_j^* \Big)\right] 
\end{equation}
is the generator of a QMS $\cP_t = e^{t\cL_\beta}$ that satisfies the $\sigma_\beta$-DBC. 

It is a simple matter to diagonalize $\cL_\beta$:
For each $1\leq j \leq m$, define the four operators
\begin{equation*}
K_{j,(0,0)} = \one\ ,\quad  K_{j,(1,0)} = Z_j \ ,\quad  K_{j,(0,1)} = Z^*_j \tand K_{j,(1,1)}
 = e^{\beta e_j/2}N_j - e^{-\beta e_j/2}N_j^\perp\ .
\end{equation*}
One readily checks that this set of four operators is orthonormal in any of the inner products $\langle \cdot, \cdot\rangle_s$ based on $\sigma_\beta$.

Using the fact that for each $j$, $V_j$ and $V_j^*$ commute with $P_k$ and $Q_k$ for all $k\neq j$, and using the identities $V_j K_{j,(1,1)} = e^{\beta e_j/2} V_j$ and $K_{j,(1,1)} V_j = - e^{-\beta e_j/2} V_j$ , we readily compute that
\begin{equation}\label{krchuk2}
\cL_\beta Z_j = - \cosh(\beta e_j/2) Z_j \tand
\cL_\beta K_{j,(1,1)} = -2\cosh(\beta e_j/2)K_{j,(1,1)}\ .
\end{equation}
Therefore, for all $0 \leq k,\ell \leq 1$,
\begin{equation*}
\cL_\beta K_{j,(k,\ell)} = -(k+\ell) \cosh(\beta e_j/2)K_{j,(k,\ell)} \ .
\end{equation*}
Let $\aa = (\alpha_1,\dots, \alpha_m)$ denote a generic element of the index set $\{ \{0,1\}\times \{0,1\}\}^m$,
and for $\alpha = (k,\ell) \in \{0,1\}\times \{0,1\}$, define $|\alpha| = k +\ell$.  Then the functions
\begin{equation*}
K_\aa :=  K_{1,\alpha_1}K_{2,\alpha_2} \cdots K_{m,\alpha_m}
\end{equation*}
are an orthogonal (but not normalized) basis for $\fH_\aA$ consisting of eigenvectors of $\cL_\beta$:
\begin{equation}\label{krchuk5}
\cL_\beta K_\aa  = - \left(\sum_{j=1}^m |\alpha_j|\cosh(\beta e_j/2)\right)K_\aa\ .
\end{equation}

It is now easy to check that in the infinite temperature limit (i.e., $\beta \to 0$), 
$\lim_{\beta\to 0} \cL_\beta = \cL_0$ where
\begin{eqnarray*}
\cL_0 A 
	   &=& - \frac14  \sum_{j =1}^m\left(  [V_j,[V_j^*,A]]  + [V_j^*,[V_j,A]]\right)\nonumber\\
	   	   &=&   \frac12 \sum_{j =1}^m\left(  Q_j\Gamma(A)Q_j  + P_j\Gamma(A)P_j^*  - 2A\right) \ ,
\end{eqnarray*}
From the previous example, we recognize $\cL_0$ as the negative of the number operator on $\aA =\Cln$. That is,
in the infinite temperature limit ($\beta\to 0$), we recover the infinite temperature 
Fermi Ornstein-Uhlenbeck semigroup, justifying our nomenclature. 

As in the infinite temperature case, there is a differential calculus that is more closely adapted to $\cL_\beta$: 
For $1\leq j \leq m$, define the operators
\begin{equation*}
\sder_j A = \frac12(Z_jA - \Gamma(A)Z_j) = \frac12 W[V_j,A] \tand 
\overline{\sder}_j A = \frac12(Z_j^*A -\Gamma(A)Z_j^*) = -\frac12 W[V_j^*,A]
\end{equation*}
We readily compute that
\begin{equation}\begin{aligned}\label{gder2}
\sder_j K_{j,(0,0)}  = \sder_j K_{j,(1,0)} = 0\ , \quad \sder_j K_{j,(0,1)} &= K_{j,(0,0)} \\
\tand  \sder_j K_{j,1,1} &  = \cosh(\beta e_j/2)K_{j,(1,0)} \ ,
\end{aligned}\end{equation}
and that 
\begin{equation}\begin{aligned}\label{gder3}
\overline \sder_j K_{j,(0,0)}= \overline \sder_j K_{j,(0,1)} = 0\ , \quad  
\overline\sder_j K_{j,(1,0)} & = K_{j,(0,0)} \\ \tand  \overline \sder_j K_{j,(1,1)} & = - \cosh(\beta e_j/2)K_{j,(0,1)} \ .
\end{aligned}\end{equation}
Again using the fact that for each $j$, $V_j$ and $V_j^*$ commute with
$P_k$ and $Q_k$ for all $k\neq j$, one determines the effect of $\sder_j$ and $\overline \sder_j$ on all of $\aA$. 
The orthonormal basis $\{K_\aa\}$ may be viewed as consisting of analogs of multivariate Krawtchouck polynomials -- the discrete analogs of the Hermite polynomials. 
The differential operators $\sder_j$ and $\overline \sder_j$, which are skew derivations as in the infinite temperature case, have the advantage over the closely related derivations $\partial_j A = [V_j,A]$ and $\overline\partial_j A = [V_j^*,A]$ that they always lower the ``degree'' of any $K_\aa$ by one, 
as one would expect. 
The operators 
$\partial_j A$ and $\overline\partial_j A$ do not do this. 

Using (\ref{gder2}) and (\ref{gder3}) one readily deduces the identities, valid for all
\begin{equation}\label{fermint1}
\sder_j \cL_\beta K_\aa - \cL_\beta \sder_j K_\aa = - \cosh(\beta e_j/2) \sder_j K_\aa 
\end{equation}
and
\begin{equation}\label{fermint2}
\overline \sder_j \cL_\beta K_\aa - \cL_\beta \overline \sder_j K_\aa = - 
\cosh(\beta e_j/2) \overline \sder_j K_\aa\ .
\end{equation}
Finally, we observe that each of the vectors $K_\aa$ is an eigenvector of $\Delta_{\sigma_\beta}$. Moreover, it is easy to see that if $\{e_1,\dots, e_m\}$ is linearly independent over the integers,
then $\Delta_{\sigma_\beta} K_\aa = K_\aa$ if and only if for each $k$, $|\alpha_k| \neq 1$. 
The span of the set of such $K_\aa$ is the same as the span of
\begin{equation}\label{modularFE}
\{N_1,N_1^\perp, \dots, N_m, N_m^\perp\}\ .
\end{equation}
Hence in this case, the modular algebra $\aA_{\sigma_\beta}$ is the algebra generated by the commuting projections in (\ref{modularFE}). 
Let us denote this algebra, which does not depend on $\beta$, by $\aB$.  
While it need not be the modular algebra when $\{e_1,\dots, e_m\}$ is not linearly independent over the integers, it is easy to see (by continuity or computation) that it is always invariant under $\cP_t$. 
 
The projections in (\ref{modularFE}) are not minimal in $\aB$, but the set of the $2^{m}$ distinct non-zero products one can form from them is a full set of minimal projections. 
We may identify this set with the discrete hypercube $\cQ^m = \{0, 1\}^m$. 
Set $\cJ = \{1, \ldots, m\}$, and let $s_j:\cQ^m \to \cQ^m$ define the $j$-th coordinate swap defined by $s_j (x_1,  \ldots, x_m) = (x_1, \ldots, - x_j, \ldots, x_m)$. 
Let $\bx$ denote a generic point of $\cQ^m$. 
Define ${\displaystyle E_\bx = \prod_{j=1}^m  N_j^{x_1}(N_j^\perp)^{1-x_1}}$.
The restriction $\widecheck \cP_t$ of $\cP_t$ to $\aB$ is a nearest neighbor random walk on $\cQ^m$ with transition rates that are readily computed using Theorem~\ref{restrict}. 
 
For a standard representation in which the elements of $\aA$ operate on $\C^{2^m}$, and $\tau$ is the normalized trace, each $E_\bx$ is rank one, so that the transition rate matrix $D$ defined in (\ref{modexp2}) is simply $D_{\bx,\bx'} = \tr[ E_\bx, \cL E_{\bx'}]$.
Using (\ref{krchuk2}) through (\ref{krchuk5}), one readily computes that $D_{\bx,\bx'} = 0$ unless $\bx' = s_j(\bx)$ for some $j$, and in that case
$$
D_{\bx,\bx'} = \begin{cases} {\displaystyle \frac{2\cosh(\beta e_j)}{1 + e^{-\beta e_j}}} & x_j =1\\
{\displaystyle \frac{2\cosh(\beta e_j)}{1 + e^{\beta e_j}}} & x_j =0 \ ,
\end{cases}
$$
and this gives the jump rates along the edges of $\cQ^m$ for the classical Markov chain corresponding to $\widecheck \cP_t$.

\subsection{The Bose Ornstein-Uhlenbeck semigroup}

A set $\{Q_1,\dots,Q_m,P_1,\dots,P_m\}$ of self-adjoint operators on a Hilbert space $\H$ is a representation of the {\em Canonical Commutation Relations} (CCR), in case for all $1 \leq j,k \leq m$,
\begin{equation*}
[Q_j,Q_k] = 0\ , \quad [P_j,P_k] = 0 \tand [Q_j,P_k] = i\delta_{j,k}\one \ .
\end{equation*}
All representations of the CCR are necessarily infinite-dimensional, since otherwise we would have $\tr[[Q_j,P_j]]=0$ which is incompatible with $[Q_j,P_j] =i\one$. The CCR algebra is the $C^*$-algebra generated by the unitaries
$\{e^{-t_1Q_1},\dots,e^{it_mQ_m},e^{-s_1P_1},\dots,e^{is_mP_m}\}$ for all $t_1,\dots,t_m,s_1,\dots, s_m$, or, what is the essentially the same thing, the Weyl operators.  
Not only is $\H$ necessarily infinite-dimensional, but the operators $\{Q_1,\dots,Q_m,P_1,\dots,P_m\}$ are unbounded.

Therefore, the CCR algebra for $m$ Bose degrees of freedom lies outside the scope of the theory being developed in this paper. 
However, even without fully extending this theory to infinite dimensions, we shall be able to deduce new results for an important QMS on $\aA$, namely the Bose Ornstein-Uhlenbeck semigroup. 

To keep things simple in this excursion into the infinite-dimensional case, we take $m=1$. 
Exactly as in the Fermi case, we form the operators
\begin{equation*}
Z = \frac{1}{\sqrt{2}}(Q+ i P)\quad  {\rm so\ that}\quad Z^* = \frac{1}{\sqrt{2}}(Q - i P)\ .
\end{equation*}
It is easy to check that 
\begin{equation}\label{createanihilate2B}
[Z,Z^*] = \one \ .
\end{equation}

In one standard representation that we may as well fix here, $\H = L^2(\R,\gamma(x){\rm d}x)$ where $\gamma(x) = (2\pi)^{-1/2}e^{-x^2/2}$ and $Z = \partial /\partial x$. 
Then a simple computation shows that $Z^* = x - \partial /\partial x$, and (\ref{createanihilate2B}) is satisfied. 
Define the Hamiltonian $h$ by $h = Z^*Z$. 
It is evident that for each $k$, the linear space of polynomials in $x$ of degree at most $k$ is invariant under $h$. 
Since $h$ is self-adjoint, this means the eigenfunctions of $h$ are orthogonal polynomials in $\H$, and hence are the Hermite polynomials. 
It is well known and easy to check that the $k$th Hermite polynomial is an eigenfunction of $h$ with eigenvalue $k$. 
That is, $h$ is the Bose number operator (for one degree of freedom). 
Fixing an inverse temperature $\beta\in (0,\infty)$, we define $\sigma_\beta$ as 
\begin{equation*}\label{bose1}
\sigma_\beta = \Big(\tr\big[ e^{-\beta h}\big]\Big)^{-1}e^{-\beta h} \ .
\end{equation*}
Note that $\tr\left[ e^{-\beta h}\right]$ is finite by what we have said concerning the spectrum of $h$. 
One readily finds that
$[Z,h] = Z$,
which is the differential version of the identity
$$\Delta_{\sigma_\beta}(Z) = e^{\beta}Z\ .$$
It follows that $Z$ and $Z^*$ are eigenfunctions of the modular operator. 
(Note that since they are unbounded, they do not belong to the CCR algebra, and are only affiliated to its von Neumann algebra closure.)  

Define $V_1 = Z$ and $V_2 = Z^*$. Then $\{V_1,V_2\}$ is a set of operators satisfying conditions {\it (iii)} and {\it (iv)} of Theorem~\ref{strucB} (with $\omega_1 = -\beta$ and $\omega_2 = \beta$), but not conditions {\it (i)} and {\it (ii)}, since in the infinite-dimensional case, it is in general too much to ask that the $V_j$ be trace-class. (However, since $V_1$, $V_2$ and $\one$ are eigenvectors  of $\Delta_\sigma$ with distinct eigenvalues there is a natural sense in which they are orthogonal so that a natural analog of {\it(i)} and {\it (ii)} is valid.)

In any case, we may define 
\begin{eqnarray}\label{bose3}
\cL_\beta A &=& \frac12\left[e^{\beta  /2} \Big( Z^* [A, Z] +  [Z^* ,A] Z \Big) 
	   + e^{-\beta /2} \Big( Z [A, Z^*] +  [Z , A] Z^* \Big)\right] \nonumber\\
	   &=& e^{\beta/2} \left(Z^*AZ - \tfrac12\{Z^* Z,A\}\right)  + 
	   e^{-\beta/2} \left(ZAZ^* - \tfrac12\{Z Z^*,A \}\right) \ ,
\end{eqnarray}
where $\{A,B\}$ denotes the anti-commutator $AB + BA$. 
The operator $\cL_\beta$ is in fact the generator of an ergodic QMS, as shown in \cite{CFL00}. 
These authors construct the QMS first on the infinite-dimensional analog of $\fH_\aA$, which in this case strictly contains $\aA$, and then show that the resulting semigroup has the {\em Feller property}; i.e., it preserves $\aA$. 

In \cite{CFL00}, another detailed balance condition based on self-adjointness with respect to the KMS inner product is used. 
However, Theorem~\ref{detequiv} and what we have said above about the modular operator shows that the semigroup also satisfies the $\sigma_\beta$-DBC as defined here. 
In this sense, the example falls into our framework. 

Simple computations show that for $\partial_1 A = [Z,A]$ and $\partial_2 A = [Z^*,A]$,
\begin{equation}\label{intwbose}
\partial_j \cL_\beta A -  \cL_\beta\partial_j  A =  - \sinh(\beta/2)\partial_j A
\end{equation}
for $j=1,2$ and all $A$ in a dense
domain of analytic vectors for $\cL_\beta$ that is discussed in \cite{CFL00}. 
We shall use this identity on this domain later to prove a sharp entropy dissipation  inequality for this semigroup, as conjectured in \cite[equation (9)]{HKV16}. Note that the corresponding formula for the Fermi Ornstein-Uhlenbeck semigroup involves $\cosh$ in place of $\sinh$. This reflects the fact that in the Fermi case, taking the infinite temperature limit ($\beta \to 0$) yields a 
QMS with a stationary state, while for the Bose Ornstein-Uhlenbeck semigroup, this is not the case: the $\beta\to 0$ limit cannot be taken in \eqref{bose1}.

The modular generator $h$ has non-degenerate spectrum, and so the modular algebra 
$\aA_{\sigma_\beta}$ in this case is simply the set of all operators that commute with 
$\sigma_\beta$, which is the same thing as the algebra generated by the 
spectral projections of $h$. In particular, 
$\aA_{\sigma_\beta}$ is commutative and independent of $\beta$.  
It is easy to see that the restriction
$\widecheck \cP_t$ of $\cP_t$ to $\aA_{\sigma_\beta}$ corresponds, as in 
Theorem~\ref{restrict}, to a birth-death process on $\N$.

\section{Riemannian metrics and gradient flow}\label{Riemannian}

Let $\cP_t = e^{t\cL}$ be QMS on $\aA$ that satisfies the 
$\sigma$-DBC for $\sigma\in \Dens_+(\aA)$. 
In this section we define a Riemannian metric on $\Dens_+$ that is determined by $\cL$, and for which, as we shall see, the flow given by the dual semigroup $\cP_t^\dagger$, is gradient flow for the relative entropy with respect to $\sigma$.  
Let $\cL$ be given in the standard form (\ref{genform2}).
Throughout this section, $\{V_j\}_{j\in \cJ}$ and  $\{\omega_j\}_{j\in \cJ}$ are fixed, and we assume that $\cP_t$ is ergodic. 

Let $\rho(t), t\in (t_0,t_1)$, be any differentiable path in $\Dens_+$ regarded as a convex subset of $\aA$. 
For each $t\in (t_0,t_1)$, let $\bdot \rho(t) \in \aA$ denote the derivative of $\rho(t)$ in $t$.   
If $\rho(t)$ is any differentiable path in $\Dens_+$ defined on $(-\epsilon,\epsilon)$ for some $\epsilon > 0$ such that $\rho(0) = \rho_0$, then $\tr[\bdot\rho(0)] = 0$, so that by Theorem~\ref{Poiss}, there is an affine subspace of $\fH_{\aA,\cJ}$ consisting of elements ${\bf A}$ for which
\begin{equation}\label{tan1}
\bdot\rho(0) = \dive {\bf A}\ .
\end{equation}

We wish to rewrite (\ref{tan1}) as an analog of the classical continuity equation for the time evolution of a probability density $\rho(x,t)$ on $\R^n$:
\begin{equation}\label{tan2}
\frac{\partial}{\partial t} \rho(x,t) + \dive [{\bf v}(x,t) \rho(x,t)] = 0\ .
\end{equation}
In the classical case, for $\rho$ strictly positive, any expression of the form 
\begin{equation}\label{tan3}
\frac{\partial}{\partial t} \rho(x,t) = \dive [{\bf a}(x,t) ] 
\end{equation}
gives rise to (\ref{tan2}) with ${\bf v}(x,t) = -\rho^{-1}(x,t){\bf a}(x,t)$. 
Conversely, given (\ref{tan2}) and defining
${\bf a}(x,t) = - \rho(x,t) {\bf v}(x,t)$, (\ref{tan3}) is satisfied. 
In the quantum case, there are many different ways to multiply and divide by $\rho\in \Dens_+$. 

Definition~\ref{multrho} gives a one-parameter family of ways to multiply $A \in \aA$  by $\rho$ that is relevant here. 
In the next definition, we extend this to multiplication of vector fields ${\bf A}\in \fH_{\aA,\cJ}$ by $\rho$.

\begin{defi}\label{multdiv} 
Let $\vec \omega \in \R^{|\cJ|}$. For $\rho\in \Dens_+$ we define the linear operator $[\rho]_{\vec \omega}$ on $\fH_{\aA,\cJ}$ by 
$$[\rho]_{\vec \omega}\big(A_1,\dots,A_{|\cJ|}\big) = \big([\rho]_{\omega_1}A_1,\dots,[\rho]_{\omega_{|\cJ|}}A_{|\cJ|}\big) \ .$$ 
\end{defi}

Note that $[\rho]_{\vec \omega}$ is invertible with
\begin{equation}\label{mbf3}
[\rho]_{\vec\omega}^{-1}  \big( A_1,\dots,A_{|\cJ|} \big) = 
\big([\rho]_{\omega_1}^{-1}A_1,\dots,[\rho]_{\omega_{|\cJ|}}^{-1}A_{|\cJ|} \big) \ .
\end{equation}
where we have used the fact that $R_\rho$ and $\Delta_\rho$ commute. 

We are now ready to write (\ref{tan1}) in the form of a continuity equation: 
Pick some $\vec \omega \in \R^{|\cJ|}$, and define ${\bf V}$ by ${\bf V} = - [\rho]_{\vec \omega}^{-1}{\bf A}$.
Then evidently (\ref{tan1}) becomes
\begin{equation}\label{conteq1}
\bdot\rho(0) + \dive ([\rho]_{{\vec \omega}} {\bf V}) = 0 \ .
\end{equation}

The vector field ${\bf A}$ in (\ref{tan1}) is not unique; however according to Theorem~\ref{Poiss}, the set of such vector fields is an affine space, and thus, in our finite-dimensional setting a closed convex set. 
It follows immediately that while the vector field ${\bf V}$ in (\ref{conteq1}) is not unique, 
the set of such vector fields is a closed affine subspace of $\oplus^{|\cJ|}\aA$, and consequently there is a unique element of minimal norm in 
$\oplus^{|\cJ|}\aA$ for any Hilbertian norm on $\oplus^{|\cJ|}\aA$.  We now define the class of Hilbertian norms that is relevant here:

\begin{defi}\label{bfipdef}   For each $\rho\in \Dens_+$, and the given generator $\cL$, 
define an inner product $\langle \cdot , \cdot \rangle_{{\cL,\rho}}$
on $\oplus^{|\cJ|}\aA$ by
\begin{equation*}
\langle {\bf W} , {\bf V} \rangle_{{\cL,\rho }}  
= \sum_{j\in \cJ} \langle {\bf W}_j, [\rho]_{\omega_j} {\bf V}_j\rangle_{\fH_\aA} \ .
\end{equation*}
We write $\|{\bf V}\|_{{\cL ,\rho}}$ for the corresponding Hilbertian norm. 
\end{defi}
This norm can be viewed as a non-commutative analog of a weighted $L^2$-norm for vector fields.

\begin{thm}\label{thm:unique-vf} Let $\rho(t)$ be a differentiable path in $\Dens_+$ defined on
$(-\epsilon,\epsilon)$ for some $\epsilon>0$ such that $\rho(0) = \rho_0$. 
Then there is a unique vector field ${\bf V} \in \oplus^{|\cJ|}\aA$ of the form ${\bf V} = \nabla U$ with $U \in \cA$, for which the non-commutative continuity equation
\begin{equation}\label{conteq3}
\bdot\rho(0)
	 = - \dive( [\rho_0]_{\vec \omega } {\bf V}) 
   	 = - \dive( [\rho_0]_{\vec \omega } \nabla U)
\end{equation}
holds. Moreover, $U$ can be taken to be traceless, and is then self-adjoint and uniquely determined.  Furthermore, if ${\bf W}$ is any other vector field such that $\bdot\rho(0) = - \dive ( [\rho_0]_{\vec \omega } {\bf W})$, then 
\begin{equation*}
\|{\bf V}\|_{\cL ,\rho_0} <  \|{\bf W}\|_{\cL,\rho_0} \ .
\end{equation*}
\end{thm}

\begin{proof}
In view of the discussion above, it remains to show that the unique norm-minimizing vector field ${\bf V}$ is a gradient. To see this, let ${\bf A}$ be an arbitrary divergence-free vector field, set ${\bf W} := [\rho_0]_{\vec\omega}^{-1}{\bf A}$, and ${\bf V}_\eps := {\bf V} + \eps {\bf W}$, so that $\bdot\rho(0) + \dive ( [\rho_0]_{\vec \omega } {\bf V}_\eps) = 0$ for all $\eps$. Since $\|{\bf V}\|_{\cL ,\rho_0} <  \|{\bf V}_\eps\|_{\cL,\rho_0 }$ for all $\eps$, it follows that $\langle {\bf V}, {\bf W}  \rangle_{{\cL,\rho_0 }} = 0$, and therefore $\ip{{\bf V},{\bf A}}_{\fH_\aA} = 0$. This means that ${\bf V}$ is orthogonal to the set of divergence-free vector fields, hence it is the gradient of some $U\in \cA$. By subtracting a multiple of the identity, we may take $U$ to be traceless, and then $U$ is uniquely determined, in view of Theorem \ref{erg} and the ergodicity of $\cP_{t}$.

To show that $U$ is self adjoint, define the operator $\cL_\rho$ by 
\begin{equation}\label{lrhodef}
\cL_\rho A = \dive ([\rho_{\vec \omega}]\nabla A)\ .
\end{equation}
A direct computation yields
$$\cL_\rho A  =  \sum_{j\in \cJ} \left( [\rho]_{\omega_j}(V_jA-AV_j)\right)V_j^* -  
\sum_{j\in \cJ}  V_j^*\left( [\rho]_{\omega_j}(V_jA-AV_j)\right)\ .$$
Then using (\ref{starrho}) of Lemma~\ref{smooth}, 
$$(\cL_\rho A)^*  = \sum_{j\in \cJ} \left( [\rho]_{-\omega_j}(V_j^*A^*-A^*V_j^*)\right)V_j - \sum_{j\in \cJ}  V_j\left( [\rho]_{-\omega_j}(V_j^*A^*-A^*V_j^*)\right) 
\ .$$
Now use the fact that $\{V_j\}_{j\in \cJ} = \{V_j^*\}_{j\in \cJ}$ and that for all $j\in \cJ$, $\Delta_\sigma(V_j) = e^{-\omega_j}V_j$ and $\Delta_\sigma(V_j^*) = e^{\omega_j}V_j^*$. It follows that 
$(\cL_\rho A)^* = \cL_\rho A^*$.
Using (\ref{lrhodef}) we write (\ref{conteq3}) as ${\displaystyle \bdot \rho(0) = -\cL_{\rho(0)} U}$ for the $U$ found above. Since $\bdot \rho(0)$ is self-adjoint, it follows from what we have just shown  that  we also have ${\displaystyle \bdot \rho(0) = -\cL_{\rho(0)} U^*}$. By the uniqueness of $U$, $U$ is self-adjoint. 
\end{proof}

\begin{defi}\label{remet}  
For each $\rho\in \Dens_+$, we identify the tangent space $T_\rho$ at $\rho = \rho_{0}$, with the set of gradients vector fields $\cG :=  \{ \nabla U  \ : \ U  \in \cA\ ,\ U = U^* \}$  through the one-to-one correspondence provided by (\ref{conteq3}).  
We define the Riemannian metric $g_{\cL }$ on $\Dens_+$ by
\begin{equation*}
\|\bdot\rho(0)\|^2_{g_{\cL,\rho }} = \|{\bf V}\|_{\cL,\rho }^2 \ 
\end{equation*}
where $\bdot\rho(0)$ and ${\bf V}$ are related by (\ref{conteq3}). 
\end{defi}

The metric we have just defined is $C^\infty$. Indeed, let $\aA$ be $m$-dimensional and let 
${A_1,\dots, A_{m-1}}$ be an orthonormal  set of $m-1$ self-adjoint traceless elements of $\fH_\aA$.
Then we can define a coordinate map $u:\Dens_+ \to \R^{m-1}$ by
$$u(\rho) = \big(\tr[A_1\rho], \dots, \tr[A_{m-1}\rho]\big)\ .$$
Note that $u(\tau)  = 0$. Evidently $u$ is a one-to-one map of $\Dens_+$ onto an open bounded convex subset of $\R^{m-1}$. We give $\Dens_+$, as usual, the corresponding differential structure. Conveniently, an atlas of just one chart covers the manifold. 

Let $u^k(\rho) = \tr[A_k\rho]$ by the $k$th coordinate function. The $k$th coordinate vector field is tangent 
to the curve $t \mapsto \rho + t A_k$ for $t$ in the open interval in which the right hand side belongs to 
$\Dens_+$.  The operator $\dive [\rho]_{\vec \omega} \nabla$ is invertible on
the orthogonal complement of the identity; i.e., on the span of ${A_1,\dots, A_{m-1}}$.  Define the
$k$th {\em potential function}  $X_k(\rho)$ to be the unique traceless solution $X$ of 
\begin{equation*}
\dive [\rho]_{\vec \omega} \nabla X = A_k\ .
\end{equation*}
It then follows that for the curve $t \mapsto \rho + t A_k$,  $\bdot \rho(0) = 
\dive [\rho]_{\vec \omega} \nabla X_k(\rho)$. This means that the $k$th coordinate  tangent vector field
$\partial/\partial u^k$ is given by
\begin{equation*}
\frac{\partial}{\partial u^k} = \nabla X_k(\rho) \ ,
\end{equation*}
Therefore, in this coordinate system, the $k,\ell$ component of the metric tensor is given by
\begin{equation*}
[g_\cL(\rho)]_{k,\ell} =  \sum_{j\in \cJ} \langle  \nabla X_k(\rho), [\rho]_{\omega_j}  \nabla X_\ell(\rho)\rangle_{\fH_\aA}  \ ,
\end{equation*}
By Lemma~\ref{smooth}, for each $j$, $\rho \mapsto [\rho]_{\omega_j}$ is $C^\infty$, and it follows from this that the map $\rho \mapsto [\dive [\rho]_{\vec \omega} \nabla]^{-1}$, where the inverse is the inverse on the orthogonal complement of $\one$, is $C^\infty$. Thus, for each $k,\ell$, 
$[g_\cL(\rho)]_{k,\ell} $ is a $C^\infty$ function of $\rho$. 

\medskip 
Now let $\F : \Dens_+ \to \R $ be a differentiable function.
The differential of $\F$, denoted ${\displaystyle \frac{\delta \F}{\delta \rho}(\rho)}$, is the unique traceless self-adjoint element in $\aA$ satisfying
\begin{align}\label{eq:differential}
\lim_{t\to 0} \frac{1}{t}\big(\F(\rho+t A) - \F(\rho)\big) = \tr\left[ \displaystyle \frac{\delta \F}{\delta \rho}(\rho)A\right]
\end{align}
for all traceless self-adjoint $A\in \aA$.
This notation is traditional in the context of gradient flows for the $2$-Wasserstein metric, and it 
allows us to reserve the symbol $D$ for covariant derivatives on our Riemannian manifold.)

The corresponding gradient vector field, denoted ${\rm grad}_{g_{\cL} }\F(\rho)$, will be interpreted using the identification of the tangent space given in Definition \ref{remet}: it is the unique element in $\cG$ satisfying
\begin{equation}\label{eq:gradient}
\ddt \F(\rho(t)) \bigg|_{t=0} = \bip{  \mathrm{grad}_{g_{\cL}} \F(\rho) , \nabla U  }_{\cL,\rho} 
\end{equation}
for all differentiable paths $\rho(t)$ defined on $(-\epsilon,\epsilon)$ for some $\epsilon>0$ with $\rho(0) = \rho$ and $\bdot\rho(0) + \dive( [\rho]_{\vec \omega } \nabla U) = 0$ for some self-adjoint $U$.
Combining \eqref{eq:differential} and \eqref{eq:gradient}, it follows that 
\begin{align*}
-  \Bip{\frac{\delta \F}{\delta \rho}(\rho),\,  \dive\big( [\rho]_{\vec \omega } \nabla U\big) }_{\fH_\aA} = \bip{ \mathrm{grad}_{g_{\cL}} \F(\rho) ,\,[\rho]_{\vec \omega } \nabla U }_{\fH_{\aA,\cJ}} \ .
\end{align*}
Since this argument holds for arbitrary paths $\rho(t)$, Theorem \ref{thm:unique-vf} implies that this identity holds for arbitrary $U$. Therefore, we have proved:

\begin{thm}\label{gradformula} For a differentiable function $\F$ on $\Dens_+$, 
the Riemannian gradient of $\F$ with respect to the Riemannian metric $g_{\cL}$ is given by
$${\rm grad}_{g_{\cL} }\F(\rho)  =  \nabla  \frac{\delta \F}{\delta \rho}(\rho) \ ,$$
and the corresponding gradient flow equation (for steepest descent) is
\begin{equation*}
\bdot \rho(t) = \dive \Big( [\rho(t)]_{\vec\omega} \nabla \frac{\delta \F}{\delta \rho}(\rho(t)) \Big)\ .
\end{equation*}
\end{thm}

Let $\cP_t = e^{t\cL}$ be a QMS on $\aA$ that satisfies the 
$\sigma$-DBC for $\sigma\in \Dens_+(\aA)$.
Recall that the {\em relative entropy with respect to $\sigma$} is the functional $D(\cdot \| \sigma)$  on $\Dens_+$ defined by (\ref{relent}).
An easy calculation shows that for $\F(\rho) = D(\rho || \sigma)$, 
$$
\frac{\delta \F}{\delta \rho} = \log \rho - \log \sigma\ .
$$
Therefore, Theorem~\ref{gradflowent} and Theorem \ref{gradformula} yield: 
\begin{thm}\label{gradflowentgen} Let $\cP_t = e^{t\cL}$ be QMS on $\aA$ that satisfies the 
$\sigma$-DBC for $\sigma\in \Dens_+(\aA)$.  Then 
\begin{equation}\label{lind}
\frac{\partial}{\partial t} \rho = \cL^\dagger \rho
\end{equation}
 is gradient flow for the relative entropy $D(\cdot ||\sigma)$ in the metric
$g_{\rho, \cL}$ canonically associated to $\cL$ through its representation in the form (\ref{genform2}).
\end{thm}

In \cite{CM13}, we proved the special case of Theorem~\ref{gradflowentgen} in which $\cP_t$ is the infinite temperature Fermi Ornstein-Uhlenbeck semigroup, except that there we defined the metric in terms of the differential calculus associated to the skew derivations $\sder_j$, defined in (\ref{grossder}) instead of the derivations $\partial_j$ used here. The two metrics are in fact the same, and the alternate form of the metric in terms of the skew derivatives will be useful to us in the next section.

Therefore we explain the equivalence, using the notation introduced in section~\ref{sec:examples}.  Let $\rho(t)$ be a smooth path in $\Dens_+(\Cln)$ defined on a neighborhood of $0$ with $\rho(0) = \rho$.
Suppose that 
\begin{equation}\label{equiv1}
\bdot \rho(0) = - \sum_{j=1}^n \partial_j([\rho]_0 \partial_j U)
\end{equation}
for some self-adjoint $U \in \Cln$. (We recall that since $V_j$ is self-adjoint in this case, 
$\partial_j^\dagger = \partial_j$, and that $\omega_j= 0$ for each $j$.)  
Then by (\ref{sder1}) and the integral representation for $[\rho]_0$, we can rewrite (\ref{equiv1}) as 
\begin{eqnarray*}
\bdot \rho(0) &=& -4\sum_{j=1}^n \sder _j(W[\rho]_0 (W\sder_j U))
= -4\sum_{j=1}^n \int_0^1 \sder _j(W \rho^s W\sder_j U \rho^{1-s}){\rm d}s\nonumber\\
&=& -4\sum_{j=1}^n \int_0^1 \sder _j( \Gamma(\rho^s) \sder_j U \rho^{1-s}){\rm d}s\ .
\end{eqnarray*}
The operation $A \mapsto \int_0^1 \Gamma(\rho^s) A \rho^{1-s}{\rm d}s$ is precisely the non-commutative 
analog of ``multiplication by $\rho$'' that was used in \cite{CM13}. Thus apart form the trivial factor of $4$,
the realization of the tangent space and interpretation of continuity equation in \cite{CM13} is the same as it is here;
the two formulation of the continuity equation are equivalent. 

The same applies to the metric. With $\bdot \rho(0)$ as above, 
$\|\bdot \rho(0)\|_{g_\cL}^2$ as we have defined it here is given by
\begin{eqnarray}\label{equiv3}
\|\bdot \rho(0)\|_{g_\cL}^2  &=& \sum_{j=1}^n \ip{\partial_j, U [\rho]_{0} \partial_j U}_{\fH_{\Cln}}
= 4 \sum_{j=1}^n \langle W\sder_j U, [\rho]_{0} (W\sder_j U)\rangle_{\fH_{\Cln}} \nonumber\\
&=& 4 \sum_{j=1}^n \int_0^1 \langle \sder_j U, 
\Gamma(\rho^s) \sder_j U \rho^{1-s}\rangle_{\fH_{\Cln}}{\rm d}s\ .  
\end{eqnarray}
The ultimate term in (\ref{equiv3}) is, apart from a trivial factor of $4$, precisely 
how the metric tensor was defined in \cite{CM13}. 

This shows two things: First, that Theorem~\ref{gradflowentgen} is an extension, and not merely an analog, of our work in \cite{CM13}. Second, it makes available to us the differential calculus based on the skew derivations when studying the geometry associated to the Fermi Ornstein-Uhlenbeck semigroup. (We have explicitly discussed the infinite temperature case, but the same reasoning applies in general.)  Since the skew derivations have the property of ``lowering polynomial degree by one'' for the eigenfunctions of $\cL$, as discussed between (\ref{gder3}) and (\ref{fermint1}), this alternate formulation of the metric will be extremely helpful in the next section.

\section{Geodesic convexity and relaxation to equilibrium}

In this section we develop the advantages of having written the evolution equation (\ref{lind}) as gradient flow for the relative entropy. 
We draw on work of Otto and Westdickenberg \cite{OW05} and Daneri and Savar\'e \cite{DS08}. 
Both pairs of authors were primarily interested in infinite-dimensional problems concerning metrics on spaces of probability densities, but several of their results are new and interesting in finite dimension. 
The approach of Otto and Westdickenberg is thoroughly developed in the finite-dimensional setting in Section 2 of \cite{DS08}. 
We briefly summarize what we need. 

Let $(\cM,g)$ be any smooth, finite-dimensional Riemannian manifold.  
For $x,y$ in $\cM$, the Riemannian distance $d_g(x,y)$ between $x$ and $y$ is given by minimizing an action integral of paths $\gamma:[0,1]\to \cM$ running from $x$ to $y$:
\begin{equation*}
d_g^2(x,y) = \inf\left\{ \int_0^1 \|\bdot \gamma(s)\|^2_{g(\gamma(s))} 
\dd s\ :\ \gamma(0) =x,\ \gamma(1) = y\right\} \ ,
\end{equation*}
where
\begin{equation*}
\|\bdot \gamma(s)\|^2_{g(\gamma(s))} =  g_{\gamma(s)}(\bdot \gamma(s), \bdot \gamma(s)) \ .
\end{equation*}
(If the infimum is achieved, any minimizer $\gamma$ will be a geodesic.)  
If $F$ is a smooth function on $\cM$,
let ${\rm grad}_g F$ denote its Riemannian gradient. Consider the semigroup $S_t$ 
of transformations on $\cM$
given by solving $\dot \gamma(t) = -{\rm grad}_g F(\gamma(t))$; we assume for now that nice global solutions exist. The semigroup $S_t$, $t\geq 0$, is {\em gradient flow for $F$}. 

For $\lambda \in \R$, the function $F$ is $\lambda$-convex  in case whenever $\gamma:[0,1]\to \cM$ is a distance minimizing geodesic, then for all $s\in (0,1)$,
\begin{equation*}
\frac{{\rm d}^2}{{\rm d}s^2} F(\gamma(s)) \geq \lambda g(\bdot \gamma(s), \bdot \gamma(s))\ .
\end{equation*}

It is a standard result that whenever $F$ is $\lambda$-convex, the gradient flow for $F$ is $\lambda$-contracting in the sense that for all $x,y\in \cM$ and $t > 0$,
\begin{equation}\label{act3}
\ddt d_g^2(S_t(x),S_t(y))  \leq -2\lambda d_g^2(S_t(x),S_t(y)) \ .
\end{equation}

Otto and Westdickenberg \cite{OW05}  developed an approach to geodesic convexity that takes \eqref{act3} as its starting point. 
Let $\{\gamma(s)\}_{s\in [0,1]}$ be any smooth path in $\cM$ with $\gamma(0)= x$ and $\gamma(1) = y$.  
They use the gradient flow transformation $S_t$ to define a one-parameter family of paths $\gamma^t:[0,1]\to \cM$, $t\geq 0$ defined by
\begin{equation*}
\gamma^t(s) = S_t\gamma(s)\ .
\end{equation*}
Since $\gamma^t$ is admissible for the variational problem that defines $d_g(S_t(x),S_t(y))$, it is immediate that for each $t\geq 0$, 
\begin{equation}\label{act5}
d_g^2(S_t(x),S_t(y)) \leq \int_0^1 \left\| \frac{{\rm d}}{{\rm d}s}\gamma^t(s)\right\|^2_{g(\gamma^t(s))} \dd s \ .
\end{equation}
In the present smooth setting it is shown in \cite[(2.8) -- (2.11)]{DS08}  that if for all smooth curves
$\gamma:[0,1]\to\cM$, 
\begin{equation}\label{act6}
\frac{{\rm d}}{{\rm d}t}\bigg|_{0+}  \left( \left\| \frac{{\rm d}}{{\rm d}s}\gamma^t(s)\right\|^2_{g(\gamma^t(s))}\right)
\leq -2\lambda 
 \left\| \frac{{\rm d}}{{\rm d}s}\gamma^0(s)\right\|^2_{g(\gamma^0(s))}\ ,
\end{equation}
for all $s\in (0,1)$, then $F$ is geodesically $\lambda$-convex.

To see the connection between (\ref{act6}) and the contraction property, suppose that $x$ and $y$ are connected by a minimal geodesic $\gamma$ so that 
$$d_g^2(x,y) = \int_0^1 \left\| \frac{{\rm d}}{{\rm d}s}\gamma(s)\right\|^2_{g(\gamma(s))} \dd s\ .$$
(If $x$ and $y$ are sufficiently close, this is the case.)
Then (\ref{act5}) and (\ref{act6}) combine to yield 
\begin{equation*}
\ddt\bigg|_{0+}  d_g^2(S_t(x), S_t(y))  \leq -2\lambda 
d_g^2(x,y)\ ,
\end{equation*}
and then, provided that $S_t(x)$ and $S_t(y)$ continue to be connected by a minimal geodesic for all $t$, the semigroup property of $S_t$ yields the exponential $\lambda$-contractivity of the flow:
\begin{equation}\label{act8}
d_g(S_t(x),S_t(y))   \leq e^{-\lambda t}d_g(x,y)\ .
\end{equation}
The local argument in \cite{DS08} proves the geodesic $\lambda$-convexity of $F$ when (\ref{act6}) is valid for all smooth paths in $\cM$, and thus leads to (\ref{act8}) without any assumptions of geodesic completeness.

When (\ref{act6}) is valid for some $\lambda>0$, and hence also (\ref{act8}) for the same $\lambda$, $F$ has at most one fixed point $x_0$ in $\cM$, which is necessarily a strict minimizer of $F$ on $\cM$.  
We may normalize $F(x_0) = 0$, and then under the geodesic $\lambda$-convexity of $F$, is it well known that for all $x$,
\begin{equation}\label{act9}
\frac{{\rm d}}{{\rm d}t} F(S_t(x)) \leq - 2\lambda F(S_t(x))\ ,
\end{equation}
which gives us another way to measure the rate of convergence to the fixed point under the flow $S_t$. 

There is also a more direct route from \eqref{act6} to \eqref{act9}. 
If $\gamma(t)$ is given by the gradient flow of $F$ through $\gamma(t) = S_t(x)$, then
\begin{equation}\label{act10}
\ddt F(\gamma(t)) = -\|{\rm grad}_g F(\gamma(t))\|^2_{g(\gamma(t))} \ .
\end{equation}
Define the {\em energy function} $E$ associated to $F$ by
\begin{equation}\label{act10b}
E(x) = \|{\rm grad}_g F(x)\|^2_{g(x)}\ .
\end{equation}
Then (\ref{act6}) applied with $\gamma^t(s) = S_{s+t}(x)$ 
together with the semigroup property yields
\begin{equation*}
\frac{{\rm d}}{{\rm d}t} E(S_t(x)) \leq -2\lambda E(S_t(x))\ .
\end{equation*}
Hence (\ref{act6}) not only leads to the contractivity property (\ref{act8}), but also to the exponential convergence estimates
\begin{equation}\label{act12}
F(S_t(x)) \leq e^{-2\lambda t}F(x) \quad{\rm and}\quad E(S_t(x)) \leq e^{-2\lambda t}E(x)\ . 
\end{equation}
Moreover, combining (\ref{act12}) with  (\ref{act10}) and (\ref{act10b}), we obtain the inequality
\begin{equation}\label{act13}
F(x) \leq  \frac{1}{2\lambda}E(x)\ . 
\end{equation}
In our setting, when $F$ is a relative entropy function, (\ref{act13}) will be a 
{\em generalized  logarithmic Sobolev inequality}. 

The relations between (\ref{act8}) and the bounds in (\ref{act12}) and geodesic convexity of $F$ have all been discussed by Otto \cite{O01} as the basis of his approach to quantitative estimates on the rates of relaxation for solutions of the porous medium equation.

Thus, to prove geodesic convexity of $F$, and hence (\ref{act8}) and (\ref{act12}), it suffices to prove (\ref{act6}).  
This first derivative estimate can provide a much easier route to a proof of  $\lambda$-convexity of $F$ than direct calculation of the Hessian of $F$ followed by an estimate of its least eigenvalue. 
The point of view of Otto and Westdickenberg is that this approach can be especially fruitful in an infinite-dimensional setting (such as that of \cite{O01}) given all the regularity issues to go along with computing the Hessian of $F$. 
In the remainder of this section, we shall show that it is also quite fruitful in our finite-dimensional setting. 
Related work in the commutative setting can be found in \cite{EM12,EMT15,FM16,LiMi13,Ma11,Mi13}.
For the difficulties relating to direct computation and analysis of the Hessian even for the Fermi Ornstein-Uhlenbeck semigroup, see our previous paper \cite{CM13}.

\subsection{Geodesic convexity using intertwining relations}

For the rest of this section, let $\sigma\in \Dens_+$, and fix  $\cP_t = e^{t\cL}$, an ergodic QMS that satisfies the $\sigma$-DBC. 
Let $\cL$ be given in the standard form (\ref{genform2}), so that the data specifying 
 $\cL$ are the sets $\{V_j\}_{j\in \cJ}$ and  $\{\omega_j\}_{j\in \cJ}$.  
 Let $\nabla: \fH_{\aA} \to \fH_{\aA,\cJ}$ and $\dive: \fH_{\aA,\cJ} \to \fH_{\aA}$ be the associated non-commutative gradient and divergence (as opposed to the associated Riemannian gradient and divergence).

Let $\rho:[0,1]\to \Dens_+$ be a smooth path in $\Dens_+$, and define the one-parameter family of paths, $\rho^t(s)$, $(s,t) \in [0,1]\times[0,\infty)$ by 
\begin{equation*}
\rho^t(s)  = \cP_t^\dagger \rho(s)\ .
\end{equation*}
By what has been explained above, it we can prove that
\begin{equation}\label{act23}
\frac{{\rm d}}{{\rm d}t} \bigg|_{0+} 
\left( \left\| \frac{{\rm d}}{{\rm d}s}\rho^t(s)\right\|^2_{g(\rho^t(s))}\right)\leq -2\lambda 
 \left\| \frac{{\rm d}}{{\rm d}s}\rho^0(s)\right\|^2_{g(\rho^0(s))}
\end{equation}
for all smooth $\rho:[0,1]\to \cM$ and all $s\in (0,1)$, we will have 
proved the geodesic convexity of the relative entropy functional, and 
consequently, we shall have proved 
\begin{equation*}
D(\cP_t^\dagger \rho || \sigma) \leq e^{-2\lambda t}D(\rho || \sigma)\ .
\end{equation*}

We now present a simple sufficient condition for (\ref{act23}) that we 
shall be able to verify in a number of interesting examples.

\begin{defi} A semigroup $\vec \cP_t$ on $\fH_{\aA,\cJ}$ {\em intertwines with} a semigroup
$\cP_t$ on $\fH_{\aA}$ in case  for all $t>0$, and all $A \in \fH_{\aA}$,
\begin{equation}\label{intertwine} 
\nabla \cP_t A = \vec \cP_t \nabla A\ .
\end{equation}
\end{defi}
By duality, the intertwining relation $\nabla\circ  \cP_t = \vec{\cP_t} \circ\nabla$ 
implies the identity
\begin{align*}
\cP_t^\dagger  \dive (\bbA) = \dive(\vec{\cP_t}^\dagger \bbA)\;, \qquad \text{for } \bbA \in \fH_{\aA,\cJ} \ .
\end{align*}
We will be particularly interested in cases in which for some $\lambda \in \R$,
\begin{equation}\label{goodcase}
\vec \cP_t \bbA =  
(e^{-\lambda t} \cP_t A_1,\dots, e^{-\lambda t}\cP_t A_{|\cJ|})\ .
\end{equation}

\begin{remark}\label{LedouxProof} A classical example is provided by the 
Mehler formula for the  classical Ornstein-Uhlenbeck semigroup, which was 
first studied by Mehler in 1866.  For $\beta>0$, let 
$\gamma_\beta(x) = (\beta/2\pi)^{n/2}e^{-\beta |x|^2/2}$ be the centered Gaussian density on $\R^n$ with zero mean and variance $n/\beta$. For $t>0$ and bounded continuous functions $f$ on $\R^{n}$, define $P_t f$ by
\begin{equation}\label{Meh1}
P_tf(x) = \int_{\R^n} f(e^{-t}x + (1-e^{-2t})^{1/2}y) \gamma_\beta(y) \dd y \ .
\end{equation}
Then $P_t$ is a classical Markov semigroup; namely the Mehler or 
Ornstein-Uhlenbeck semigroup.  The dual semigroup $P_t^\dagger$ acting on 
probability densities $\rho$ on $\R^n$ is defined by
$$\int_{\R^n} P_t^\dagger\rho(x) f(x)  \dd x = \int_{\R^n}\rho(x) P_t f(x) \dd x$$
A change of variables yields the dual Mehler formula: 
\begin{equation*}
P_t^\dagger \rho(x) = \int_{\R^{n}} \rho \big(e^{-t}x - (1-e^{-2t})^{1/2}y\big) 
\gamma_\beta\big((1-e^{-2t})^{1/2}x + e^{-t}y\big)  \dd y\ .
\end{equation*}
A Taylor expansion in (\ref{Meh1}) and then integration by parts show that 
$f(x,t) := P_tf(x)$ and $\rho(x,t) := P_t^\dagger \rho(x)$ satisfy
\begin{equation*}
\frac{\partial}{\partial t} f(x,t) = 
\left(   \frac{1}{\beta} \dive - x\right) \nabla f(x,t)
\quad{\rm and}\quad \frac{\partial}{\partial t} \rho(x,t) =  
\dive \left(   \frac{1}{\beta} \nabla + x\right)\rho(x,t)\ .
\end{equation*}
It is immediate from (\ref{Meh1}) that
\begin{equation}\label{Meh4}
\nabla P_t f(x) = \ \vec P_t \nabla f(x) 
\end{equation}
where $\vec P_t( v_1, \dots v_n) (x) = e^{-t}( P_t v_1(x), \dots , P_t v_n(x)) $.
We shall see below that an identity similar to \eqref{Meh4} is readily proved for the Fermi and Bose Ornstein-Uhlenbeck semigroup.  
\end{remark}

Using \eqref{Meh4}, which is a direct analog of (\ref{intertwine}) and (\ref{goodcase}), Ledoux \cite[p. 447]{Le92} gave a very simple proof of the optimal logarithmic Sobolev inequality for the classical Ornstein-Uhlenbeck semigroup. 
A key element in his proof is the joint convexity of $(a,r) \mapsto |a|^2/r$ on $\R^n\times (0,\infty)$, for which we will need a suitable non-commutative analogue.

The latter is provided by a well-known convexity inequality for matrices, which asserts that, for all $\omega \in \R$, the mapping
\begin{align}\label{eq:convex-trace}
	 (\rho, A) \mapsto \langle  A, [\rho]_{{\omega}}^{-1}
 A\rangle_{\fH_\aA} 
   = \Tr\bigg[\int_0^\infty (t \one  + e^{-\omega/2} \rho)^{-1} A^* (t \one  + e^{\omega/2} \rho)^{-1} A \dd t \bigg] 
\end{align}
is jointly convex on $\Dens_+ \times \aA$; see \cite{HP12,HP12B}. Note that if $\rho$ and $A$ are scalars, the right-hand side reduces to $A^2/\rho$. 
The non-commutative convexity result ultimately derives from Lieb's concavity Theorem \cite{L73}. Since $\cP_t^\dagger$ is completely positive, it follows from \eqref{eq:convex-trace} that 
\begin{align}\label{eq:contractive-metric}
 \langle \cP_t^\dagger A, [\cP_t^\dagger\rho]_{{\omega}}^{-1}
 \cP_t^\dagger A\rangle_{\fH_\aA}	 \leq \langle  A, [\rho]_{{\omega}}^{-1} 
 A\rangle_{\fH_\aA} \ .
\end{align}
There is a well-developed theory of {\em monotone metrics} beginning with work of Chentsov and Morozova \cite{MoCh} for classical Markov processes and its non-commutative extension initiated by Petz \cite{Pe96}, and further developed in \cite{LeRu,HK99,PWPR06,HKPR12,TKRWV}. Other results from this theory will be useful in further developments.

Now consider any smooth path $\rho:[0,1]\to \Dens_+$, and for each $s\in (0,1)$ write
\begin{equation*}
\bdot \rho(s) = \dive \bbA(s)
\end{equation*}
where $\bbA(s)$ is the solution of $\bdot \rho(s) = \dive \bbA(s)$ that minimizes  
$\langle \bbA, [\rho]_{{\vec \omega }}^{-1}
\bbA\rangle_{\cL,\rho}$ so that, by the definitions in Section~\ref{Riemannian},
\begin{equation*}
g_{\cL,\rho}(\bdot \rho(s),\bdot \rho(s)) = 
\sum_{j\in \cJ} \langle A_j(s), [\rho(s)]_{{\omega_j }}^{-1}
A_j(s)\rangle_{\fH_\aA}\ .
\end{equation*}
Set $\rho^t(s) := \cP_t^\dagger \rho(s)$, and suppose that the semigroup $\vec \cP_t$ defined by \eqref{goodcase} intertwines with $\cP_t$. 
It follows that
$$
\frac{{\rm d}}{{\rm d}s}\rho^t(s) = \cP_t^\dagger \dive \bbA(s) = 
\dive \vec \cP_t^\dagger \bbA(s)\ .$$
Consequently, by \eqref{goodcase} and \eqref{eq:contractive-metric},
\begin{eqnarray}
\left\| \frac{{\rm d}}{{\rm d}s}\rho^t(s)\right\|^2_{g(\rho^t(s))}  &\leq& 
e^{-2\lambda t} \sum_{j\in \cJ} \langle \cP_t^\dagger A_j(s), 
[\cP_t^\dagger\rho(s)]_{{\omega_j }}^{-1}
\cP_t^\dagger A_j(s)\rangle_{\fH_\aA}\nonumber\\
&\leq& e^{-2\lambda t} \sum_{j\in \cJ} \langle  A_j(s), [\rho(s)]_{{\omega_j }}^{-1}
 A_j(s)\rangle_{\fH_\aA}  
 = e^{-2\lambda t} 
 \left\| \frac{{\rm d}}{{\rm d}s}\rho(s)\right\|^2_{g(\rho(s))}\ ,\nonumber
\end{eqnarray}
which clearly implies (\ref{act23}).
Altogether we have proved:

\begin{thm}\label{sumthm} 
Let $\sigma\in \Dens_+$, and let $\cP_t = e^{t\cL}$ be an ergodic QMS that satisfies the $\sigma$-DBC. 
Let $\nabla$ and $\dive$ denote the associated non-commutative gradient and divergence. 
Suppose that for some $\lambda> 0$, the semigroup $\vec \cP_t$ defined by (\ref{goodcase}) intertwines with 
 $\cP_t$. 
Then the relative entropy with respect to $\sigma$ is geodesically $\lambda$-convex on $\Dens_+$ for the Riemannian metric $(g_{\cL,\rho})_{\rho}$.
Moreover, the exponential convergence estimate
\begin{align*}
	D(\cP_t^\dagger \rho || \sigma) \leq e^{-2\lambda t}D(\rho || \sigma)
\end{align*}
holds, as well as the generalized logarithmic Sobolev inequality
\begin{align}\label{eq:gen-LSI}
	D(\rho || \sigma) \leq  \frac{1}{2\lambda} \tau\big[-\cL^\dagger(\rho)\big(\log \rho - \log \sigma\big)\big] \ .
\end{align}
\end{thm} 

While it is a problem of ongoing research to extend Theorem~\ref{sumthm} to the infinite dimensional setting, the part of it concerning entropy and entropy production inequalities is relatively robust, as it relies most essentially on Lieb's convexity result, while the part of it concerning geodesic convexity is more involved and requires more work to generalize. Therefore, with regard to the infinite dimensional Bose-Ornstein-Uhlenbeck semigroup, we are not presently in a positions to make any statements about geodesic convexity of the entropy, but the situation is much better concerning entropy and entropy production.

Let $\mathcal{D}[\rho] = \tau\big[-\cL^\dagger(\rho)\big(\log \rho - \log \sigma\big)\big]$ be the entropy dissipation functional, which is minus the  derivative of 
$D(\cP_t^\dagger \rho || \sigma)$ at $t=0$.  By Theorem~\ref{gradflowent},
\begin{equation*}
- \cL^\dagger \rho = \sum_{j\in \cJ}  \partial_j^\dagger \Big( [\rho]_{\omega_j} \partial_j(\log \rho - \log \sigma)\Big) = {\rm div}({\bf A})\ .
\end{equation*}
The entropy dissipation functional $\mathcal{D}[\rho]$ is then given by
$$\mathcal{D}[\rho] =  \sum_{j\in \cJ} \langle  A_j, [\rho]_{{\omega_j }}^{-1}
 A_j\rangle_{\fH_\aA} \ .$$
 
 Now replace $\rho$ by $\rho(t) := \cP_t^\dagger \rho$ and let us assume that $\rho$ is in the domain of $\cL^\dagger$, a trivial assumption in the finite dimensional case.  Then 
 $$\frac{\rm d}{{\rm d}t} \rho(t) =  \cP_t^\dagger( \cL^\dagger \rho) =  \cP_t^\dagger({\rm div}({\bf A}))\ ,$$
 and assuming that $\vec \cP_t$ defined by \eqref{goodcase} intertwines with $\cP_t$, we have that 
 $${\bf A}(t) = e^{-\lambda t} (\cP_t^\dagger  A_1,\dots, \cP_t^\dagger A_{|\cJ|})\ .$$
 Therefore, by the convexity argument used in the proof of Theorem~\ref{sumthm}
 $$\mathcal{D}[\cP_t^\dagger \rho]   =  e^{-2\lambda t} \sum_{j\in \cJ} \langle \cP_t^\dagger A_j, 
[\cP_t^\dagger\rho]_{{\omega_j }}^{-1}
\cP_t^\dagger A_j\rangle_{\fH_\aA}  \leq  e^{-2\lambda t}  \mathcal{D}[ \rho]  \ .$$
By a standard argument, the generalized log-Sobolev inequality (\ref{eq:gen-LSI}) follows immediately for $\rho$ in the domain of $\cL^\dagger$. 
In summary, if one is interested more in entropy dissipation inequalities than the geodesic convexity of the entropy Theorem~\ref{gradflowent}
together with an intertwining relation allows one to bypass the Riemannian structure. This fact also underlines the utility of writing the evolution equation as gradient flow for the entropy, which is essential for the argument. 

In the rest of this section we explain how intertwining formulas may be proved. We consider two examples, already introduced, namely the Bose and Fermi Ornstein-Uhlenbeck semigroups. In the Fermi case, we are within the scope of the finite dimensional picture developed here, and we will be able to prove the geodesic convexity of the relative entropy. In the Bose case, we are in an infinite dimensional setting, and work remains to be done to rigorously prove the geodesic convexity in this case. However, by what has been explained in the preceding paragraphs, we shall obtain a rigorously valid generalized logarithmic Sobolev inequality.

\subsection{Intertwining via commutation formulas}

For both the Fermi and Bose Ornstein-Uhlenbeck semigroups $\cP_t$, there is a 
Mehler type formula for 
$\cP_t$ from which the intertwining can be readily checked. In the Fermi case, 
this can be found in formulas (4.1) and (4.2) of \cite{CL93}, and the formula in the Bose case is a simple adaptation of this.

Fortunately however, it is not necessary to find an explicit formula for the action of 
the semigroup $\cP_t$ to prove the intertwining identity and (\ref{goodcase}). In case where such identities are true, they can often be readily checked using the form of the generator ${\cL}$. 

\begin{lm}\label{odelm} Suppose that for some numbers $a_j$, $j\in \cJ$, 
\begin{align}\label{intertw2A}
[\partial_j,\cL]  = -a_j\partial_j \ 
\end{align}
for each $j\in \cJ$.  Then defining  $\vec \cP_t$ on $\fH_{\aA,\cJ}$ by
\begin{equation*}
\vec \cP_t(A_1,\dots,A_{|\cJ|}) = (e^{-ta_1} \cP_t A_1,\dots,e^{-ta_{|\cJ|}}\cP_t A_{|\cJ|})\ ,
\end{equation*}
we have the intertwining relation $\partial_j \cP_t = \vec \cP_t \partial_j$ on $\aA$. 
\end{lm}

\begin{proof}
Let $A\in \aA$ and define $A(t) = \partial_j \cP_t A$. Then $A(0) =\partial_j A$ and 
$$\frac{{\rm d}}{{\rm d}t}A(t) = \partial_j \cL \cP_t A = 
\cL \partial_j \cP_t A  - a_j \partial_j \cP_t A = [\cL - a_j I]A(t)\ .$$
It follows that $t \mapsto e^{ta_j}A(t)$ is the unique solution of
${\displaystyle \frac{{\rm d}}{{\rm d}t}X(t) = \cL X(t)}$   with 
$X(0) =\partial_j A$, which is of course $\cP_t \partial_j A$. Therefore,
$\partial_j\cP_t A  =   e^{-ta_j} \cP_t \partial_j A $.
\end{proof}

\begin{thm}\label{BoseEnt} Let $\cP_t$ be the Bose Ornstein-Uhlenbeck semigroup with generator
$\cL_\beta$ given in (\ref{bose3}), and let $\sigma_\beta$ be its invariant state. Then for all $\rho\in \Dens_+$, 
$$D(\cP_t \rho||\sigma_\beta) \leq e^{-2\sinh(\beta/2) t} D( \rho||\sigma_\beta)\ .$$
\end{thm}

\begin{proof} Using (\ref{intwbose}), we may apply Lemma~\ref{odelm}, and then the remarks following Theorem~\ref{sumthm}. (Note that for all $t>0$,
$\cP_t^\dagger$ is in the domain of $\cL^\dagger$ so that the generalized log-Sobolev inequality is valid.)
\end{proof}

Results in \cite[Appendix D]{HKV16} show that the constant $2\sinh(\beta/2)$ in Theorem~\ref{BoseEnt} cannot be improved.

We may make a similar application of Lemma~\ref{odelm}, and then Theorem~\ref{sumthm} itself  to the Fermi Ornstein-Uhlenbeck semigroup. 
However, in this case, it is not the differential structure in terms of the derivations $\partial_j$ for which we have \eqref{intertw2A}, but the skew derivations $\sder_j$ and $\overline{\sder}_j$. 
This was proved in (\ref{fermint1}) and (\ref{fermint2}). 
However, the metric can be written in terms of $\sder_j$ and $\overline{\sder}_j$ just as well, bringing in the principle automorphism $\Gamma$, and indeed, this is how the metric was written in \cite{CM13}. 
This permits us to argue as above in the Bose case, and we conclude:

\begin{thm}\label{FermiEnt} For $\beta \geq 0$, let $\cP_t$ be the Fermi Ornstein-Uhlenbeck semigroup with generator
$\cL_\beta$ given in (\ref{Lbeta}), and let $\sigma_\beta$ be its invariant state. Then for all $\rho\in \Dens_+$, 
$$D(\cP_t \rho||\sigma_\beta) \leq e^{-2\lambda_\beta t} D( \rho||\sigma_\beta)$$
where
$\lambda_\beta = \min\{ \cosh(\beta e_j/2)\ : \  j=1,\dots, m\}$.  Moreover, the realtive entropy functional $\rho\mapsto D( \rho||\sigma_\beta)$
is geodesically $\lambda_\beta$ convex in the Riemannain metric associated to $\cL_\beta$.
\end{thm}

\subsection{Talagrand type inequalities}  

Our final results in this section extend a result from outr earlier paper \cite{CM13} to the present more generla setting. 
The proof is step-for-step the one from our previous paper, with only minor modifications, and we shall thereofre be brief. 
However, it is worth recording the more general result since this subject has recently attracted the attention of other researchers \cite{JZ,RD}. 

The connection between logarithmic Sobolev inequalities and transport inequalities of Talagrand type \cite{Tal96} was originally discovered and developed by Otto and Villani \cite{OV00}. The fact that a Talgrand type inequlaity holds for our transport metric is further evidence that it is indeed a {\em bona-fide} transport metric.   

\begin{thm}[Talagrand type inequality]\label{tal}  Let $\cL$ be the generator of an ergodic QMS that satisifes the $\sigma$-DBC
with respect to $\sigma\in \Dens_+$. Suppose that the generalized logarithmic Sobolev inequality 
(\ref{eq:gen-LSI}) is valid for some $\lambda>0$. 
Let $d(\rho_1,\rho_2)$ denote the  Riemannain distance on $\Dens_+$ associated to $\cL$.  For all $\rho\in \Dens_+$, 
\begin{equation}\label{tal1}
d(\rho,\sigma) \leq  \sqrt{\frac{2D(\rho||\sigma)}{\lambda}}\ .
\end{equation}
\end{thm}

\begin{proof} Given $\rho\in \Dens_+$, define 
$\rho(t) = \cP_t\rho$
for $t\in (0,\infty)$.
Since $\lim_{t\to\infty}\rho(t) = \sigma$,  it follows that
$$d(\rho,\sigma) \leq  {\rm arclength}[\rho(\cdot)] =   \int_0^\infty \sqrt{ g_{\rho(t)}(\dot \rho(t), \dot \rho(t))}\dd t\ .$$
Since the evolution described by $t\mapsto \cP_t^\dagger \rho$ is gradient flow for the relative entropy, 
${\displaystyle g_{\rho(t)}(\dot \rho(t), \dot \rho(t)) = -\frac{{\rm d}}{{\rm d t}}D(\rho(t)||\sigma)}$
so that 
for any $0 \leq t_1 < t_2< \infty$,
\begin{equation}\label{tala}
\int_{t_1}^{t_2} \sqrt{ g_{\rho(t)}(\dot \rho(t), \dot \rho(t))}\dd t \leq \sqrt{t_2-t_1}
 \sqrt{
 D(\rho(t_1)||\sigma) - D(\rho(t_2)||\sigma)
 }\ .
\end{equation}
Fix any $\epsilon>0$. Define the sequence of times $\{t_k\}$, $k \in \N$, 
$$D(\rho(t_k)||\sigma) = e^{-k\epsilon}D(\rho||\sigma)\ .$$
(Since $t\mapsto D(\rho(t)||\sigma)$ is strictly decreasing, $t_k$ is well defined.)
Since $D(\rho(t)||\sigma) \leq e^{-2\lambda t} D(\rho||\sigma)$, for each $k$,

$$t_k - t_{k-1} \leq \frac{\epsilon}{2\lambda}\ .$$
Then by (\ref{tala}), with this choice of $\{t_k\}$, 
\begin{align*}
\int_{t_{k-1}}^{t_k} \sqrt{ g_{\rho(t)}{\rho(t)}(\dot \rho(t), \dot \rho(t))}\dd t 
& \ \leq \
 \sqrt{\frac{\epsilon}{2\lambda}( e^{-(k-1)\epsilon} - e^{-k\epsilon}){D(\rho||\sigma)}}
\\& \ = \ \sqrt{\frac{D(\rho||\sigma)}{2\lambda}} e^{-k\epsilon/2}\sqrt{ \epsilon (e^{\epsilon} -1)}\ .
\end{align*}
Since
$$\lim_{\epsilon\to 0}\left(\sum_{k=1}^\infty e^{-k\epsilon/2}\sqrt{ \epsilon (e^{\epsilon} -1)} \right) =
\lim_{\epsilon\to0}\left(\sum_{k=1}^\infty  e^{-k\epsilon/2} \epsilon \right) = \int_0^\infty e^{-x/2}\dd x =2\ ,$$
we obtain the desired bound.  
\end{proof}

\appendix

\section{Proof of Theorem~\ref{strucB}}  

In this appendix we present a simple and self-contained proof of Theorem~\ref{strucB}.
The starting point is an isometry that is crucial to the characterization of quantum Markov semigroup generators given by Gorini, Kossakowski and Sudarshan \cite{GKS76}:

For any finite dimensional Hilbert space $\fH$, let $\cC_2(\fH)$    denote the linear operators from $\fH$ to $\fH$ 
equipped with the {\em normalized} Hilbert-Schmidt inner product $\langle A,B\rangle_{\cC_2(\fH)} = ({\rm dim}(\H))^{-1}\tr[A^*B]$
 so that $\|\one\|_{\fH} = 1$. As above, we use $\dagger$ for the Hermitian adjoint in $\cC_2(\fH)$.  A special case deserves a special notation: let $\fH_n$ denote the $n\times n$ complex 
 matrices $\cM_n(\C)$ equipped with this same normalized Hilbert-Schmidt inner product. 
 
A particular orthonormal basis in $\fH_n$ plays a distinguished role in what follows: For $1 \leq i,j \leq n$, let $E_{i,j}$
denote the $n\times n$ matrix whose $i,j$ entry is $1$, and whose other entries are all $0$. If $\{e_1,\dots,e_n\}$ is the standard basis of $\C^n$, then $E_{i,j}$ is the rank one operator that is written as $|e_i \bk e_j|$ in a standard quantum mechanical notation introduced before.  In this notation, one has
\begin{equation}\label{choi1}
E_{i,j} \otimes E_{i,j} = (|e_i \bk e_j|)\otimes (|e_i \bk e_j|) = |e_i\otimes e_i \bk e_j\otimes e_j|\ .
\end{equation}
It follows that 
\begin{equation}\label{choi2}
\frac1n \sum_{i,j=1}^n E_{i,j} \otimes E_{i,j} = |\Psi \bk \Psi|  \quad{\rm where}\quad \Psi = \frac{1}{\sqrt{n}} \sum_{j=1}^n e_j\otimes e_j
\end{equation}
is a rank one projection in $\C^n\otimes \C^n$, and, in particular, is positive. This observation is due to Choi, and some of the simple but  fundamental conclusions he drew from it are related below. More immediately, $\{E_{i,j}\}_{1\leq i,j \leq n}$ is an orthonormal basis of $\fH_n$ called the {\em matrix unit basis}.

There is a natural  identification of $\cC_2(\fH_n)$ with 
$\fH_n\otimes \fH_n$ that takes advantage of the multiplication on 
$\fH_n$: For $A,B \in \fH_n$, define
the operator $\#(A\otimes B): \fH_n \to \fH_n$  by 
$$ \#(A\otimes B) : \cM_n(\C) \to \cM_n(\C) \ , \qquad \#(A\otimes B)X = AXB \ , $$ for all $X\in \fH_n$. It follows that the adjoint of $\#(A\otimes B)$ as an operator on $\fH_n$ is given by
	$\big( \#(A\otimes B) \big)^\dagger = \#(A^* \otimes B^*)$.
Moreover, 
\begin{align}\label{eq:sharp-trace}
	\tr[\#(A\otimes B)] = \tr[A] \tr[B] \ . 
\end{align}
The left-hand side of (\ref{eq:sharp-trace}) involves the trace for operators on $\fH_n$, whereas the right-hand side involves the trace for operators on $\C^n$.  

\begin{lm}\label{GKSlem}
Let $\{F_\alpha\}$ and $\{G_\beta\}$ be two orthonormal bases of $\fH_n$, so that $\{F_\alpha\otimes G_\beta\}$ is an orthonormal basis of $\fH_n\otimes \fH_n$. 
Then $\{ \#(F_\alpha\otimes G_\beta)\}$ is orthonormal in $\cC_2(\fH_n)$. 
In particular, the map $\#$ is unitary  from $\fH_n\otimes \fH_n$ into $\cC_2(\fH_n)$. 
\end{lm}


\begin{proof}[Proof of Lemma~\ref{GKSlem}] We may compute the trace on $\fH_n$ using 
any orthonormal basis, and 
using  the matrix unit orthonormal basis  $\{E_{i,j}\}_{1\leq i,j \leq n}$, we have
for all $\alpha,\beta$ and $\mu,\nu$,
\begin{eqnarray*}
\langle \#(F_\alpha\otimes G_\beta), \#(F_\mu\otimes G_\nu)\rangle_{\cC_2(\cM_n(\C))}
 &=& n^{-2}\sum_{i,j=1}^n \tr[ (F_\alpha E_{i,j} G_\beta)^*F_\mu E_{i,j} G_\nu]\nonumber\\
  = n^{-2}\sum_{i,j=1}^n \tr[ G_\beta^*E_{j,i} F_\alpha^* F_\mu E_{i,j} G_\nu]\nonumber
    &=& n^{-2}\sum_{i,j=1}^n (G_\nu G_\beta^*)_{j,j}(F_\alpha^* F_\mu)_{i,i} \nonumber\\
    &=& n^{-1}\tr[G_\nu G_\beta^*]n^{-1}\tr[F_\alpha^*F_\mu] = \delta_{\alpha,\mu}\delta_{\beta,\nu}\ .
    \end{eqnarray*}
\end{proof}

Consider any linear transformation $\cK$  on $\cM_n(\C)$, and hence on $\fH_n$. Let 
$\{F_\beta\}$ be any orthonormal basis of $\fH_n$.
Then $\{F_\alpha^*\}$ is also an orthonormal basis of $\fH_n$, and by 
Lemma~\ref{GKSlem}, $\{\#(F_\alpha^*\otimes F_\beta)\}$ 
is an orthonormal basis of $\cC_2(\fH_n)$. Thus $\cK$ has the expansion
\begin{equation}\label{GKSexp}
\cK = \sum_{\alpha,\beta} c_{\alpha,\beta} \#(F_\alpha^*\otimes F_\beta)\ .
\end{equation}
or, what is the same, 
\begin{equation}\label{GKSexp2}
\cK(A)  = \sum_{\alpha,\beta} c_{\alpha,\beta} F_\alpha^*A F_\beta 
\qquad{\rm for\ all}\ A\in \cM_n(\C)\ ,
\end{equation}
where the coefficients $c_{\alpha,\beta}$ are given by
\begin{equation}\label{coef}
c_{\alpha,\beta} = \langle  \#(F_\alpha^*\otimes F_\beta), \cK\rangle_{\cC_2(\fH_n)}\ .  
\end{equation}

This orthonormal expansion is fundamental to the work of Gorini, Kossakowski and Sudarshan on the structure of generators of quantum Markov semigroups. 

\begin{defi}
The $n^2 \times n^2$ matrix $c_{\alpha,\beta}$ with entries given by (\ref{coef}) is called the 
{\em GKS matrix for the operator $\cK$ with respect to the orthonormal basis
$\{F_\alpha\}$}. When we wish to emphasize the dependence on $\cK$, we write $c_{\alpha,\beta}(\cK)$. 
\end{defi}

\begin{remark}\label{identrem}  Let $A,B\in \cM_n(\C)$ and consider the case $\cK = \#(A\otimes B)$.  By the isometry proved in Lemma~\ref{GKSlem}, the GKS matrix of $\cK$ is given by
\begin{align*}
c_{\alpha,\beta} &= \langle  \#(F_\alpha^*\otimes F_\beta), \#(A\otimes B)\rangle_{\cC_2(\fH_n)} = 
\langle F_\alpha^*\otimes F_\beta, A\otimes B \rangle_{\fH_n\otimes \fH_n}\\
&= n^{-2} \tr[F_\alpha A] \tr[F^*_\beta B]\ .
\end{align*}
In particular, the identity transformation on $\cM_n(\C)$ results from the choice $A = B =\one$, and so the GKS matrix of the identity transformation is  the rank-one matrix
\begin{equation}\label{identGKS}
c_{\alpha,\beta} = n^{-2} \tr[F_\alpha] \tr[F^*_\beta]\ .
\end{equation}
This formula will be useful later on. 
\end{remark}

The following lemma is from \cite{GKS76}; for the convenience of the reader we give a short proof.

\begin{lm}\label{GKSlem-2} Let $\cK$ be a linear operator on $\cM_n(\C)$, and let 
$\{F_\alpha\}$ be an orthonormal basis of $\cM_n(\C)$. 
Then the GKS matrix of $\cK$ with respect to $\{F_\alpha\}$ is self-adjoint if and only if $(\cK A)^* = \cK A^*$ for all $A\in \cM_n(\C)$. 
\end{lm}

\begin{proof}
	Write $\cK$ as in \eqref{GKSexp} and define  $\widetilde \cK(A) := (\cK A^*)^*$. Then
	\begin{align*}
		\widetilde \cK(A) 
		= \Big(\sum_{\alpha,\beta} c_{\alpha,\beta} F_\alpha^*A^* F_\beta\Big)^*
		= \sum_{\alpha,\beta} \overline{c_{\alpha,\beta}} F_\beta^* A F_\alpha
		= \sum_{\alpha,\beta} \overline{c_{\beta,\alpha}}  F_\alpha^* A F_\beta \ .
	\end{align*} 
By the uniqueness of the expansion \eqref{GKSexp}, $\tilde \cK = \cK$ if and only if 
$c_{\alpha,\beta}  = \overline{c_{\beta,\alpha}}$ for all $\alpha, \beta$.
\end{proof}

The GKS matrix of a linear transformation $\cK$ from $\cM_n(\C)$ to $\cM_n(\C)$ is closely related to the {\em Choi matrix}
of $\cK$.  Let $\{E_{i,j}\}_{1\leq i,j\leq n}$ be  matrix unit basis of $\fH_n$. 
 The Choi matrix of $\cK$ is the element of $\cM_{n^2}(\C)$
\begin{equation}\label{choi3}
C(\cK) = \sum_{i,j=1}^n \cK(E_{i,j}) \otimes E_{i,j}\ ,
\end{equation}
viewed as the $n\times n$ block matrix with entries in $\cM_n(\C)$  whose $i,j$ entry is $\cK(E_{i,j})$. 

If we now identify $\C^n\otimes \C^n$ with $\cM_n(\C)$ by identifying $ v\otimes w$ with $\sum_{i,j=1}^n v_iw_j E_{i,j}$,
$C(\cK)$ becomes an operator on $\fH_n$.

The identity provided by the following lemma was pointed out in \cite{PZ}, and used there to simplify part of the proof \cite{GKS76} of their theorem on the structure of generators of quantum Markov semigroups on $\cM_n(\C)$. 

\begin{lm} Let $\cK$ be a linear operator on $\cM_n(\C)$, and let $C(\cK)$ be defined by (\ref{choi3}).
Identify  $\C^n\otimes \C^n$ with $\cM_n(\C)$ by identifying $ v\otimes w$ with $\sum_{i,j=1}^n v_iw_j E_{i,j}$, so that 
$C(\cK)$ is identified with an operator on $\fH_n$. Then for all $F,G\in \fH_n$,
\begin{equation}\label{choi44}
\langle G, C(\cK) F\rangle_{\fH_n} = \langle  \#(G\otimes F^*), \cK\rangle_{\cC_2(\fH_n)}
\end{equation}
\end{lm}

\begin{proof}  By direct computation,
\begin{eqnarray*}  \langle G, C(\cK) F\rangle_{\fH_n}  &=& 
\sum_{k,\ell,m,p=1}^n \overline{G}_{k,m}[\cK(E_{i,j})]_{k,\ell} [E_{i,j}]_{m,p} F_{\ell,p}\\
&=& \sum_{k,\ell,m,p=1}^n F_{\ell,p} [E_{j,i}]_{p,m} \overline{G}_{k,m}[\cK(E_{i,j})]_{k,\ell}  \\
&=& \sum_{k,\ell=1}^n (F E_{j,i} G^*)_{\ell,k}[\cK(E_{i,j})]_{k,\ell}  \\
&=& \sum_{k,\ell=1}^n (G E_{i,j} F^*)^*_{\ell,k}[\cK(E_{i,j})]_{k,\ell} = \langle  \#(G\otimes F^*), \cK\rangle_{\cC_2(\fH_n)}\ . \\
\end{eqnarray*}
\end{proof} 

A fundamental theorem of Choi \cite{Choi} says that a linear transformation $\cK$ on $\cM_n(\C)$ is 
completely positive if and only if its Choi matrix $C(\cK)$ is positive as an operator on $\fH_n$. Indeed, in the notation of (\ref{choi2})
$$C(\cK) = n\cK\otimes \one_{\cM_n(\C)}\left(|\Psi\bk \Psi |\right)\ ,$$
Hence,  when $\cK$ is completely positive, $C(\cK)$ is positive. The converse is also true: 
Choi used an elementary spectral decomposition \cite{Choi} to show  that when $C(\cK)$ is positive, 
then $\cK$ is completely positive. 

\begin{remark}\label{charrem}
The identity (\ref{choi44}) shows that if $\{F_\alpha\}$ is any orthonormal basis of $\fH_n$, then the GKS matrix of $\cK$ with respect to this basis is positive if and only if $\cK$ is completely positive. In other words, the GKS representation (\ref{GKSexp2}) of a linear operator $\cK$ on $\cM_n(\C)$ 
is well-suited to the question of whether $\cK$ is completely positive or not. 
\end{remark}

Going forward, it will be convenient to assume that our orthonormal bases $\{F_\alpha\}$ of $\fH_n$ are indexed by
$\alpha \in \{1,\dots,n\} \times \{1,\dots,n\}$, and we write $\alpha = (\alpha_1,\alpha_2)$.  For such bases, we make the following definition:

\begin{defi} Let $\cL$ be an operator on $\cM_n(\C)$ such that $\cL \one = 0$ and 
$(\cL A)^* = \cL A^*$ for all $A\in \cM_n$. 
Let $\{F_\alpha\}$ be any orthonormal basis of $\fH_n$ such that 
$F_{(1,1)} = \one$. Let $c_{\alpha,\beta}$ be the GKS matrix for 
$\cL$ with respect to $\{F_\alpha\}$. 
The $(n^2-1)\times (n^2-1)$ matrix with entries $c_{\alpha,\beta}$ where 
$\alpha$ and $\beta$ range over the set 
$\{ (i,j)\ :\ 1\leq i,j \leq n \quad{\rm and} \quad (i,j) \neq (1,1)\}$ is called the {\em reduced GKS matrix of $\cL$ for the basis $\{F_\alpha\}$.}
\end{defi}

The following lemma is due to Parravinci and Zecca \cite{PZ}:

\begin{lm}\label{PZlem} Let $\cL$ be a linear operator on $\cM_n(\C)$, and let $\cP_t = e^{t\cL}$.  Let $\{F_\alpha\}$ be an orthonormal basis for $\fH_n$ with $F_{(1,1)} = \one$. Let $c_{\alpha,\beta}$ be the GKS matrix of $\cL$ with respect to $\{F_\alpha\}$.  Then $\cP_t$ is completely positive for all $t\geq 0$ if and only if the reduced GKS matrix of $\cL$ is positive.
\end{lm}

\begin{proof} Suppose that $\cP_t$ is completely positive for each  $t>0$. By (\ref{identGKS}), the GKS matrix of the identity transformation is 
\begin{equation}\label{idenag}
c_{\alpha,\beta}(I) = \delta_{\alpha,(1,1)} \delta_{\beta,(1,1)}\ .
\end{equation}
In particular, the reduced GKS matrix of the identity is zero.
Then since
$$c_{\alpha,\beta}(t^{-1}(\cP_t - I)) = t^{-1}c_{\alpha,\beta}(\cP_t) -  t^{-1}c_{\alpha,\beta}(I)\ ,$$
it follows that the reduced GKS matrix of $t^{-1}(\cP_t - I)$ coincides with the reduced GKS matrix of 
$t^{-1}\cP_t $, and by Remark~\ref{charrem} this is positive. Taking the limit $t\to 0$, we conclude that the reduced GKS matrix of $\cL$ is positive.

Conversely, suppose that the reduced GKS matrix of $\cL$ is positive. For small $t>0$,
$$c_{\alpha,\beta}(\cP_t) = c_{\alpha,\beta}(I) + t c_{\alpha,\beta}(\cL) + o(t)$$
By (\ref{idenag}), this is positive for all sufficiently small $t$. By Remark~\ref{charrem}, $\cP_t$ is completely positive for all sufficiently small $t>0$. 
Then by the semigroup property, $\cP_t$ is completely positive for all $t>0$. 
\end{proof}

We now temporarily put aside complete positivity, and consider a linear transformation $\cL$ on $\cM_n(\C)$ 
such that $\cL$ preserves self-adjointness, and such that $\cL \one = 0$.

\begin{thm}\label{GKSA} Let $\cL$ be a linear operator on $\cM_n(\C)$ such that 
$\cL \one = 0$ and $(\cL A)^* = \cL A^*$ for all $A\in \cM_n$. 
Let $\{F_\alpha\}$ be any orthonormal basis of $\fH_n$ such that $F_{(1,1)}= \one $. 
Let $c_{\alpha,\beta}$ be the GKS matrix of $\cL$ for $\{F_\alpha\}$. 
Then $\cL$ is given by 
\begin{equation}\label{dagform}
\cL A =  -i[H,A]  + \frac12 \sum_{\alpha,\beta\ \neq (1,1)} 
c_{\alpha,\beta}\left( F_\alpha^*[A,F_\beta] +   [F^*_\alpha,A] F_\beta\right)
\end{equation}
 where $H$ is the traceless self-adjoint matrix given by
\begin{equation}\label{dagform2}
H  = \frac{1}{2 i} \sum_{\beta\neq(1,1)} (c_{(1,1),\beta} 
F_\beta - c_{\beta,(1,1)}F^*_{\beta})\ .
\end{equation}
\end{thm}

Notice that only the \emph{reduced} GKS matrix figures in the second term on the right 
 in (\ref{dagform}).

\begin{proof}[Proof of Theorem~\ref{GKSA}] 
 Let 
$c_{\alpha,\beta}$ be the GKS matrix for $\cL$ with respect to $\{F_\alpha\}$ where $F_{(1,1)} = \one$. 
Then by Lemma~\ref{GKSlem-2}, $c_{\alpha,\beta}$ is a self-adjoint matrix. 
By (\ref{GKSexp2}), for all $A \in \cM_n(\C)$,
\begin{equation} \label{unitex}
\cL A   = \sum_{\alpha,\beta} c_{\alpha,\beta} F_\alpha^*A F_\beta 
=  G^*A + AG + \sum_{\alpha,\beta\ \neq (1,1)} c_{\alpha,\beta} F_\alpha^*A F_\beta 
\end{equation}
where
${\displaystyle G =  \frac{c_{(1,1),(1,1)}}{2} \one
 + \sum_{\beta\neq(1,1)} c_{(1,1),\beta} F_\beta }$.
Let $K = \frac12(G + G^*)$ and $H = \frac1{2i}(G - G^*)$ be the self-adjoint matrices such that  $G = K+iH$. 
Then (\ref{unitex}) becomes
$$\cL A = -i[H,A]  + KA + AK  + \sum_{\alpha,\beta\ \neq (1,1)} c_{\alpha,\beta} 
F_\alpha^*A F_\beta\ .$$
Then $\cL \one  = 0$ implies
${\displaystyle K = -\frac12 \sum_{\alpha,\beta\ \neq (1,1)} c_{\alpha,\beta} F^*_\alpha F_\beta}$,
and thus, for all $A\in \cM_n(\C)$, 
\begin{eqnarray*}
\cL A &=&  -i[H,A]  + \sum_{\alpha,\beta\ \neq (1,1)} 
c_{\alpha,\beta}\left( F_\alpha^*A F_\beta - 
\frac12 F^*_\alpha F_\beta A -\frac12 A F^*_\alpha F_\beta\right)\nonumber\\
&=& -i[H,A]  + \frac12 \sum_{\alpha,\beta\ \neq (1,1)} 
c_{\alpha,\beta}\left( F_\alpha^*[A,F_\beta] +   [F^*_\alpha,A] F_\beta\right)\ .
\end{eqnarray*}
Since $\tr[G] = \tfrac{n}{2} c_{(1.1),(1,1)} \in \R$, and $\tr[H]$ is the imaginary part of $\tr[G]$, it follows that $\tr[H]= 0$. 
\end{proof}

\begin{thm}\label{strucA} Let $\cL$ be the generator of a QMS that is self-adjoint with respect to the inner product $\langle \cdot,\cdot\rangle_s$ for some $s\in [0,1]$, $s \neq 1/2$. Let $c_{\alpha,\beta}$ be the GKS matrix of $\cL$ with 
respect to a modular  orthonormal basis of $\fH_\aA$ defined in Definition~\ref{modbadef}.
Then for all $\alpha,\beta$,
\begin{equation}\label{block}
e^{\omega_\alpha}c_{\alpha,\beta} = c_{\alpha,\beta}e^{\omega_\beta} \ ,
\end{equation}
and
\begin{equation}\label{block2}
c_{\alpha,\beta}  = e^{-\omega_\alpha} c_{\beta',\alpha'} \ ,
\end{equation}
where the $\omega_\alpha$ are defined in (\ref{hform2}). 
In particular, $c$ commutes with the diagonal matrix 
$[\delta_{\alpha,\beta}e^{\omega_{}\beta}]$, so that the eigenspaces of the latter are eigenspaces of $c$. 
\end{thm}

\begin{remark} The conditions (\ref{block}) and (\ref{block2}) 
are independent of $s$. Furthermore, (\ref{block}) implies that
\begin{equation}\label{fulc}
\omega_\alpha \neq \omega_\beta\quad  \Rightarrow \quad c_{\alpha,\beta} = 0\ .
\end{equation}
Therefore with an ordering of the indices $\alpha$ so that 
$\alpha \geq \beta \iff \omega_\alpha \geq \omega_\beta$, the matrix 
$[c_{\alpha,\beta}]$ is block-diagonal. 
\end{remark}

\begin{proof}[Proof of Theorem~\ref{strucA}]  Since $\cL$ is the generator of a quantum Markov semigroup,  $\cL$ preserves self-adjointness and $\cL\one =0$.
Thus, for any orthonormal basis $\{F_\alpha\}$ of $\fH_n$ such that $F_{(1,1)} = \one$, Theorem~\ref{GKSA} applies. 
We now fix such a basis and focus on the additional consequences of 
self-adjointness with respect to $\langle \cdot, \cdot\rangle_s$. 

By Lemma~\ref{Al2}, $\cL$ commutes with the modular group, and this means that for all $A$,
$$\cL A =  \sigma^{-1}( \cL (\sigma A\sigma^{-1})) \sigma\ .$$
In terms of  the GKS expansion for $\cL$, and making use of (\ref{hform5}),
\begin{eqnarray}
 \sum_{\alpha,\beta} c_{\alpha,\beta} F_\alpha^*A F_\beta &=& \cL A 
=  \sigma^{-1}( \cL (\sigma A\sigma^{-1}))\sigma  =   
\sum_{\alpha,\beta} 
c_{\alpha,\beta} \sigma^{-1}F_\alpha^*\sigma A \sigma^{-1}F_\beta\sigma \nonumber\\
&=&  \sum_{\alpha,\beta} c_{\alpha,\beta} 
e^{\omega_\beta - \omega_\alpha} F_\alpha^*A F_\beta\ .\nonumber
\end{eqnarray}
Now \eqref{block} follows from the uniqueness of the coefficients. 
To prove (\ref{block2}), note that for any $A,B$,
$$\tr[B^*\cL A] = \sum_{\alpha,\beta}\tr[ B^* c_{\alpha,\beta}F^*_\alpha A F_\beta] = 
\sum_{\alpha,\beta}\tr[ c_{\alpha,\beta} F_\beta B^*  F^*_\alpha A ] = 
\sum_{\alpha,\beta}\tr[( \overline{c_{\alpha,\beta}} F_\alpha B F^*_\beta)^* A ] \ .
 $$
Using Lemma \ref{GKSlem-2} we conclude that
${\displaystyle \cL^\dagger B  = \sum_{\alpha,\beta}  \overline{c_{\alpha,\beta}} F_\alpha B F^*_\beta  = 
\sum_{\alpha,\beta}  c_{\beta,\alpha} F_\alpha B F^*_\beta}$.
Then
\begin{multline*}
\langle \cL A,B\rangle_s = \tr[(\cL(A))^*\sigma^{1-s}B \sigma^s] = \\
\tr[A^* \cL^\dagger (\sigma^{1-s}B \sigma^s)]
= \langle A, \sigma^{s-1} \cL^\dagger (\sigma^{1-s}B \sigma^s) \sigma^{-s}\rangle_s\ .
\end{multline*}
The self-adjointness of $\cL$ with respect to $\langle \cdot,\cdot\rangle_s $ then yields
$\cL(B) = \sigma^{s-1}\cL^\dagger(\sigma^{1-s} B\sigma^{s})\sigma^{-s}$
for all $B$. Using the GKS expansion for a modular basis and (\ref{hform5}),
$$\sum_{\alpha,\beta} c_{\alpha,\beta} F_\alpha^* B F_\beta = 
\sum_{\alpha,\beta} c_{\beta,\alpha} \sigma^{s-1} 
F_\alpha \sigma^{1-s} B \sigma^{s}F_\beta^* \sigma^{-s}
=\sum_{\alpha,\beta} c_{\beta,\alpha}e^{(1-s)\omega_\alpha}e^{s\omega_\beta} 
F_\alpha  B F_\beta^* \ .$$
Since $F_\gamma^* = F_{\gamma'}$ and $\omega_\gamma = -\omega_{\gamma'}$ for all $\gamma$, we can rewrite this as
$$\sum_{\alpha,\beta} c_{\alpha,\beta} F_\alpha^* B F_\beta = 
\sum_{\alpha,\beta} c_{\beta',\alpha'} e^{(s-1)\omega_\alpha}e^{-s\omega_\beta}
F_\alpha^* B F_\beta\ .$$
By the uniqueness of the coefficients, it follows that $e^{-s\omega_\alpha}c_{\alpha,\beta}e^{s\omega_\beta} = e^{-\omega_\alpha}c_{\beta',\alpha'}$ for all $\alpha,\beta$.
However, by (\ref{block}), $c$ commutes with the $s$th power of 
$[\delta_{\alpha,\beta}e^{\omega_\alpha}]$, and thus 
$e^{-s\omega_\alpha}c_{\alpha,\beta}e^{s\omega_\beta} = c_{\alpha,\beta}$. 
This proves (\ref{block2}).
\end{proof}

\begin{proof}[Proof of Theorem~\ref{strucB}]  By assumption $\cP_t$ has an extension $\widehat \cP_t$ from $\cA$ to a QMS on all of $\cM_n(\C)$. It suffices to treat $\widehat \cP_t$, and to simplify the notation we assume the extension is done and $\cP_t$ is a QMS on $\cM_n(\C)$. 

  Let $\{F_\alpha\}$ be a modular basis for $\sigma$, and consider the GKS expansion
\begin{equation}\label{alform1}
\cL A = \sum_{\alpha,\beta} c_{\alpha,\beta} F_\alpha^* A F_\beta\ .
\end{equation}
By Theorem~\ref{GKSA}, we can rewrite (\ref{alform1}) as
\begin{equation}\label{alform1B}
\cL A =  -i[H,A]  + \frac12 \sum_{\alpha,\beta\ \neq (1,1)} 
c_{\alpha,\beta}\left( F_\alpha^*[A,F_\beta] +   [F^*_\alpha,A] F_\beta\right)\ ,\
\end{equation}
where by (\ref{fulc}), (\ref{dagform2}) reduces to 
\begin{eqnarray*}
H  &=&
 \frac{1}{2i\sqrt{n}} \sum_{\beta\neq(1,1),\ \omega_\beta =0} (c_{(1,1),\beta} 
 F_\beta - c_{\beta,(1,1)}F^*_{\beta})
 \nonumber\\
 &=& \frac{1}{2i\sqrt{n} } \sum_{\beta\neq(1,1),\ \omega_\beta =0} 
 (c_{(1,1),\beta}  - c_{\beta',(1,1)})F_{\beta}\ .
\end{eqnarray*}
By (\ref{block2}) and (\ref{fulc}),
$c_{(1,1),\beta} = c_{\beta', (1,1)}$, and therefore $H=0$.

Making use of the fact that for all $\gamma$, $F_\gamma^* = F_{\gamma'}$, 
we replace $\alpha$ with 
$\beta'$ and $\beta$ with $\alpha'$ and use (\ref{block2})  to rewrite (\ref{alform1}) as
\begin{equation}\label{alform2}
\cL A = \sum_{\alpha,\beta} c_{\alpha,\beta} F_{\alpha'}A F^*_{\beta' }= 
\sum_{\alpha,\beta} c_{\beta',\alpha'} F_\beta A F^*_\alpha =  
\sum_{\alpha,\beta} c_{\alpha,\beta} 
e^{\omega_\alpha}F_\beta A F^*_\alpha \ .
\end{equation}
By Theorem~\ref{GKSA}, we can rewrite (\ref{alform2}) as
\begin{equation}\label{alform2B}
\cL A =  -i[\hat H,A]  + \frac12\sum_{\alpha,\beta\ \neq (1,1)} 
c_{\alpha,\beta}e^{\omega_\alpha} \left(  F_\beta [A, F_\alpha^*] + [F_\beta,A]F^*_\alpha \right)\
\end{equation}
By (\ref{fulc}), (\ref{dagform2}) reduces to 
${\displaystyle 
\hat H  = \frac{1}{2i\sqrt{n}} \sum_{\alpha\neq(1,1),\ \omega_\alpha =0} (c_{(1,1),\alpha} 
F^*_\alpha - c_{\alpha,(1,1)}F_{\alpha})}$. The same argument that led to 
$H =0$ leads to $\hat H =0$. 

Averaging (\ref{alform1B}) and (\ref{alform2B}), taking into account $H = \hat H = 0$, we obtain
\begin{equation}\label{alform3}
\cL A = \frac14 \sum_{\alpha,\beta \neq (1,1)} c_{\alpha,\beta} 
[\left( F_\alpha^*[A,F_\beta] +   [F^*_\alpha,A] F_\beta\right) + 
e^{\omega_\alpha}\left(  F_\beta [A, F_\alpha^*] + [F_\beta,A]F^*_\alpha \right)]\ .
\end{equation}
Now let $U$ be an $(n^2-1)\times (n^2-1)$ unitary matrix   that diagonalizes the reduced 
GKS matrix $c_{\alpha,\beta}$ of $\cL$ and which commutes with the matrix
$\delta_{\alpha,\beta}e^{\omega_\alpha}$, $\alpha,\beta 
\neq(1,1)$ so that $U_{\gamma,\alpha} =0 $ 
unless $\omega_\gamma = \omega_\alpha$. 
We may then write
\begin{equation}\label{alform4}
c_{\alpha,\beta} =  \frac12\sum_{\gamma\neq (1,1)} 
U^*_{\alpha,\gamma} e^{-\omega_\gamma/2}c_\gamma U_{\gamma,\beta}
\end{equation}
Each $c_\gamma$ is non-negative since, by Lemma~\ref{PZlem} the {\em reduced} 
GKS matrix $c_{\alpha,\beta}$ is positive, and since 
$U_{\gamma,\alpha} =0 $ unless $\omega_\gamma = \omega_\alpha$, 
we also have ${\displaystyle 
e^{\omega_\alpha}c_{\alpha,\beta} =  \frac12\sum_{\gamma} 
U^*_{\alpha,\gamma} e^{\omega_\gamma/2}c_\gamma U_{\gamma,\beta}}$.
Defining $V_\gamma = \sum_{\beta}U_{\gamma,\beta}F_\beta$, 
we may rewrite (\ref{alform3}) as
\begin{equation}\label{alform5}
\cL A =  \frac12 \sum_{\gamma\neq (1,1)} c_{\gamma} [e^{-\omega_\gamma/2}
\left( V_\gamma^*[A,V_\gamma] +   [V^*_\gamma,A] V_\gamma\right) 
+ e^{\omega_\gamma/2}\left(  V_\gamma [A, V_\gamma^*] + 
[V_\gamma,A]V_\gamma^* \right)]\ .
\end{equation}
By symmetry, we may assume without loss of generality that $c_\gamma = c_{\gamma'}$ where
as before $(\gamma_1,\gamma_2)' = (\gamma_2,\gamma_1)$ so that $V_\gamma^* = V_{\gamma'}$.
Then the expression simplifies to
\begin{equation*}  
\cL A =   \sum_{\gamma\neq (1,1)} c_{\gamma} [e^{-\omega_\gamma/2}
 V_\gamma^*[A,V_\gamma]  
+ e^{\omega_\gamma/2}[V_\gamma,A]V_\gamma^* ]
\end{equation*}
which is (\ref{genform}). 
Simply using (\ref{alform4}) directly in (\ref{alform1B}) leads to the alternate form
$$\cL A =   \sum_{\gamma\neq (1,1)} c_{\gamma} e^{-\omega_\gamma/2}
\left( V_\gamma^*[A,V_\gamma] +   [V^*_\gamma,A] V_\gamma\right)\ .$$

Again since $U_{\gamma,\alpha} =0 $ unless 
$\omega_\gamma = \omega_\alpha$, (\ref{hform5}) implies that 
for all $\gamma$ and all $t$,
\begin{equation}\label{alform5b}
\sigma^t V_\gamma \sigma^{-t} = e^{-t\omega_\gamma}V_\gamma
\end{equation}
Letting $\cJ = \{ (k,\ell) \, : \, 1\leq k,\ell,\leq n \quad{\rm and}\quad (k,\ell)\neq (1,1)\ \}$, 
we see that under the hypotheses of the theorem, $\cL$ must have the form (\ref{genform}), and 
(\ref{Vcom}) is the differential statement of (\ref{alform5}).
The final step is to absorb the $c_j$'s nto the $V_j$'s:  Since $c_j \geq 0$ for each $j$, we can absorb these by
by making the replacement $V_j \to \sqrt{c_j}V_j$.  This proves that the generator $\cL$ of a QMS satisfying the $\sigma$-DBC has the from specified in Theorem~\ref{strucB}.

For the converse, if $\cL$ has the specified form,  one restores the $c_j$'s by normalizing the $V_j$'s, and then writes 
$\cL$ in its GKS form for this orthonormal basis (after including $\one$ and any $V_j$'s with $c_j= 0$).  The reduced
GKS matrix of $\cL$ is unchanged as the argument starting from (\ref{unitex}) shows. Thus, by Remark~\ref{charrem}, 
$\cL$ generates a completely positive semigroup $\cP_t$, and evidently $\cL\one = 0$, so that $\cP_t$ is a QMS. The 
$\sigma$-DBC is then readily checked (using the fact that the $V_j$ are eigenvectors of $\Delta_\sigma$). 
\end{proof}

\begin{remark}   It is easy to check the existence of an extension of $\cP_t$ from $\cA$ to $\cM_n(\C)$ in many relevant cases; e.g., when $\cA$ is a Clifford algebra with an odd number of generators. One might hope that
there is a general extension using the conditional expectation. 

Recall that for any unital $C^*$-subalgebra $\aA$ of $\cM_n(\C)$, there is the {\em conditional expectation} $E_\aA$ which is the orthogonal projection in $\fH_{\cM_n(\C)}$ onto $\fH_\aA$ \cite{Um54}. This may be written as an average over the unitaries in the commutant of $\aA$ \cite{C10,Uh71}; the connected component of this group $\mathfrak{U}$ that contains the identity is a Lie subgroup of $SU(n)$, on which there exists a normalized Haar measure $\mu$, and then for all $X\in \cM_n(\C)$,
${\displaystyle E_\aA X = \int_{\mathfrak{U}} U^* XU {\rm d}\mu}(U)$.
Evidently, $E_\aA$ is a completely positive map with $E_\aA \one = \one$. 
That is, $E_\aA$ is a quantum Markov operator.

If $\cK$ is a linear transformation on $\aA$, define a linear transformation $\widehat \cK$ on $\cM_n(\C)$ by $\widehat \cK = \cK\circ E_\aA$.
Note that when a linear operator $\cP$ on $\aA$  is completely positive, so is $\widehat\cP$, and the 
restriction of $\widehat\cP$ to $\aA$ is simply $\cP$. 
Moreover, if $\cL_1$ and $\cL_2$ are two linear transformations of $\aA$, then $\widehat\cL_2\widehat \cL_1 = \widehat{\cL_2\cL_1}$.  
In this way we can ``lift'' any QMS $\cP_t$ on $\aA$ up to a one-parameter family of Markov operators $\widehat \cP_t$ on $\cM_n(\C)$ such that for all $s,t\geq 0$, $\widehat{\cP_t} \widehat{\cP_s} = \widehat{\cP_{t+s}}$. 
This construction fails to yield a semigroup only because $\lim_{t\to 0}\widehat \cP_t = E_\aA$
and not $\lim_{t\to 0}\widehat \cP_t = I_\aA$.
However, if $\cP_t = e^{t\cL}$, then ${\displaystyle \lim_{t\to 0} \frac{1}{t} (\widehat \cP_t - E_\aA) = \widehat \cL}$.
The operator $\widehat\cL$ is evidently a self-adjointness preserving linear transformation from 
$\cM_n(\C)$ to $\cM_n(\C)$, and $\widehat \cL \one = 0$.  If  $\cP_t$ satisfies the 
$\sigma$-DBC for $\sigma\in \Dens_+(\aA)$, then for all $A,B\in \cM_n(\C)$, 
$$\tau [\sigma B^* \widehat \cP_t A] = \tau [\sigma B^* E_\aA  \cP_t(E_\aA A)] = 
 \tau [(E_\aA ( B\sigma)^*)   \cP_t(E_\aA A)] =  \tau [\sigma (E_\aA B)^*   \cP_t(E_\aA A)] \ ,$$
 where we have used the fact that since $\sigma\in \aA$, $E_\aA(\sigma B^*) = \sigma E_\aA B^*$.  It follows
 that for each $t\geq 0$,  $\widehat \cP_t$ is self-adjoint  with respect to the $\sigma$-GNS inner product 
 on $\cM_n(\C)$. Consequently, the same is true of  $\widehat \cL$.  
 
Therefore, the proof of Theorem~\ref{strucB} given just above shows that $\widehat \cL$ has the form specified in Theorem~\ref{strucB} except that some $c_j$'s might be negative:  Applicability of Lemma~\ref{PZlem} requires that $\lim_{t\to 0} \cP_t = I_{\cM_n(\C)}$. 
\end{remark}

 \section{Note on KMS-symmetry}  
 
We give a construction of a class of operators $\cK$ that satisfy  
 $(\cK A)^* = \cK A^*$ for all $A$ and that are self-adjoint with respect to the  $\sigma$-KMS inner product $\langle \cdot, \cdot \rangle_{1/2}$ for some $\sigma\in \Dens_+$, but which do not commute with $\Delta_\sigma$, and consequently are not self-adjoint with respect to the $\sigma$-GNS inner product. 
An operator $\cK$ on $\cM_n(\C)$ that is self-adjoint with respect to 
$\langle \cdot,\cdot\rangle_{1/2}$ for some $\sigma\in \Dens_+$ is called {\em KMS-symmetric}.  

When $\cK$ is completely positive with $\cK(\one) = \one$, it has a Kraus representation
\begin{equation*}
\cK(A) = \sum_{j=1}^m K_j^* AK_j  \quad {\rm where}\quad \sum_{j=1}^m K^*_jK_j = \one\ 
\end{equation*}
for some set $\{K_1,\dots,K_m\} \subset \cM_n(\C)$. Evidently, 
$\cK^\dagger A = \sum_{j=1}^m K_j AK_j^*$. Suppose that $\cK^\dagger\sigma = \sigma$
with $\sigma\in \Dens_+$.  The {\em dual set of Kraus operators} 
$\{\widehat K_1,\dots,\widehat K_m\}$ is given by
\begin{equation*}
\widehat K_j = \Delta^{1/2}_\sigma K_j^*  = \sigma^{1/2} K_j^* \sigma^{-1/2} \ .
\end{equation*} 
Then
${\displaystyle  \sum_{j=1}^m \widehat K^*_j \widehat K_j = \one }$   and  
$\displaystyle{  
\sum_{j=1}^m \widehat K_j  \sigma \widehat K_j^* = \sigma}$.  
It follows that the operator $\widehat \cK$ defined by
${\displaystyle \widehat \cK A =  \sum_{j=1}^m \widehat K^*_j 
A\widehat K_j}$ is completely positive with $\widehat \cK(\one)= \one$ and 
${{\widehat{\cK}}}\,\mbox{}^\dagger (\sigma)= \sigma$. 
A simple calculation shows that $\langle \widehat \cK B , A \rangle_{1/2}  = \langle B, \cK A\rangle_{1/2}$ for all $A,B\in \cM_n(\C)$.
Thus, $\widehat \cK$ is the adjoint of $\cK$ with respect to the $\sigma$-KMS inner product, and $\widehat \cK \cK$ is a completely positive operator on $\cM_n(\C)$ such that $\widehat \cK \cK(\one) = \one$ and $(\widehat \cK \cK)^\dagger \sigma = \sigma$. 

Define a quantum Markov semigroup $\cP_t$ by
\begin{equation*}
\cP_t = \sum_{n=0}^\infty e^{-t} \frac{t^n}{n!} (\widehat \cK \cK)^n = e^{t(\widehat \cK \cK - I)}   \ .
\end{equation*} 
Evidently  $\cP_t$ is KMS-symmetric for each $t>0$ since $\widehat \cK \cK$ is KMS-symmetric. Furthermore, $\cP_t$ commutes with $\Delta_\sigma$ for each $t>0$ if and only if $\widehat \cK \cK$ commutes with $\Delta_\sigma$. 
We will show that the latter is not generally the case. 

To construct counterexamples, consider $n=2$; the construction that follows is 
readily generalized. Let $\{u_1,u_2\}$ be an orthonormal basis in $\C^2$ and let  $\{v_1,v_2\}$ be a set of two linearly independent unit vectors in $\C^2$ that are {\em not} orthogonal. 
Define the rank-one operators 
$K_1$ and $K_2$ by $K_j = |v_j\bk u_j|$, $j=1,2$. 
Evidently, $K_1^*K_1 + K_2^*K_2 = \one$,
and we define $\cK A = K_1^*AK_1+K_2^*AK_2$ so that $\cK \one = \one$. 
Then the range of $\cK^\dagger$ is spanned by 
$\big\{|v_1\bk v_1| \ , \, |v_2\bk v_2|\big\}$.  
A simple computation yields
$$\cK^\dagger\big(\alpha_1 |v_1\bk v_1|  + \alpha_2 |v_2\bk v_2| \big)  = 
\beta_1 |v_1\bk v_1|  + \beta_2 |v_2\bk v_2| \ ,$$
where
\begin{equation}\label{gramm}
 \left(\begin{array}{c} \beta_1\\ \beta_2\end{array}\right) = \left[ \begin{array}{cc}
1- a & b\\ a & 1-b\end{array}\right] \left(\begin{array}{c} \alpha_1\\ \alpha_2\end{array}\right)\ ,
\end{equation}
with $a = |\langle v_1,u_2\rangle|^2$ and  $b = |\langle v_2,u_1\rangle|^2$, 
and since $\{v_1,v_2\}$ is not orthogonal, $a+b > 0$. 
The vector $\left(\begin{array}{c} b\\ a\end{array}\right) $ is an eigenvector 
with eigenvalue $1$, and hence
\begin{equation}\label{KMSsig}
\sigma = \frac{b}{a+b} |v_1\bk v_1| + \frac{a}{a+b} |v_2\bk v_2|
\end{equation}
satisfies $\cK^\dagger \sigma = \sigma$, and hence, as we just noted, $(\widehat{\cK} \cK)^\dagger \sigma = \sigma$. 
The other eigenvalue of the matrix in (\ref{gramm}) is $1-a-b < 1$, so that (\ref{KMSsig}) gives the unique invariant state. 
It follows that the eigenvalues of $\cK^\dagger$ are $1$, $1 -a-b$ 
and $0$, with
\begin{equation*}
{\rm Null}(\cK^\dagger) =  
{\rm Span} \big\{|u_1 \bk u_2| \ , \ |u_2\bk u_1| \big\} \ .
\end{equation*}
Consequently, the null space of $\cK$ is $2$-dimensional as well, and same holds for the null space of $\widehat \cK \cK$, since $\widehat \cK$ is the $\sigma$-KMS dual of $\cK$. 
By ergodicity, it is the only eigenspace of $\widehat \cK \cK$ with this property.

Generically, $\sigma$ will have two distinct eigenvalues. (For example, take $\{u_1,u_2\}$ to be the standard basis of $\C^2$, and take $v_1 = 
\frac{1}{\sqrt{2}}\left(\begin{array}{c} 1\\ 1\end{array}\right)$ and 
$v_2 = 
\frac{1}{\sqrt{5}}\left(\begin{array}{c} 1\\ 2\end{array}\right)$, so that 
 $\sigma = \frac{1}{7}\left[\begin{array}{cc} 2 & 3\\ 3 & 5\end{array}\right]$.)
Let $\{\eta_1,\eta_2\}$ be an orthonormal basis of eigenvectors of $\sigma$. 
If $\cP_t$ and $\Delta_\sigma$ were to commute, then Remark \ref{halfdif} would imply that $|\eta_1 \bk \eta_2|$ and $|\eta_2 \bk \eta_1|$ are linearly independent eigenvectors of $\widehat \cK \cK$ with the same eigenvalue. 
This eigenvalue  can only be zero by the above. But then $0 = \cK(|\eta_1 \bk \eta_2|) = \cK(|\eta_2 \bk \eta_1|)$.
Hence we would have
$$0 =\cK (|\eta_2 \bk \eta_1|) = \langle v_1,\eta_2\rangle \langle \eta_1,v_1\rangle |u_1\bk u_1| + 
\langle v_2,\eta_2\rangle \langle \eta_1,v_2\rangle |u_2\bk u_2|\ ,$$
which would mean that
\begin{equation}\label{count}
	\langle v_1,\eta_2\rangle \langle \eta_1,v_1\rangle = 0 \tand 
	\langle v_2,\eta_2\rangle \langle \eta_1,v_2\rangle = 0\ .
\end{equation}
Suppose $\langle v_1,\eta_2\rangle = 0$. 
Then since $\{v_1,v_2\}$ is not orthogonal, and $\{\eta_1,\eta_2\}$ is, $\langle v_2,\eta_1\rangle \neq 0$. 
The second equality in (\ref{count}) then yields $\langle v_2,\eta_2\rangle$, but we cannot have both $\langle v_1,\eta_2\rangle = 0$ and $\langle v_2, \eta_2 \rangle = 0$ since this would imply that $\eta_2 = 0$. 

Under this condition,  the first equality in (\ref{count}) would yield $\langle \eta_1,v_1\rangle = 0$. As above, this would imply
$\langle \eta_2,v_2\rangle \neq 0$, and hence $\langle \eta_1,v_2\rangle = 0$. We cannot have both 
$\langle v_1,\eta_1\rangle = 0$ and $\langle v_2,\eta_1\rangle=0$ since this would imply that $\eta_1 = 0$. 
Hence $\cK (|\eta_2 \bk \eta_1|)= 0$ is impossible. 

Thus, with this choice of $\cK$, $\cL :=\widehat \cK \cK - I$ is the generator of a quantum Markov semigroup with invariant state $\sigma\in \Dens_+$ such that $\cL$ is self-adjoint with respect to the $\sigma$-KMS inner product
$\langle \cdot,\cdot \rangle_{1/2}$, but such that $\cL$ does not commute with $\Delta_\sigma$. 
It follows that $\cL$ is not self-adjoint with respect to the GNS inner product. 

\medskip
\noindent{\bf Acknowledgements} E.C. was partially supported by NSF grant DMS 1501007, and thanks IST Austria for hospitality during a visit in June 2015. 
E.C. thanks the Mittag-Leffler Institute for hopitality during the final work on this paper. Both authors thank the Erwin Schr\"odinger Institute in Vienna for hospitality during a visit in June 2016.

 \end{document}